\documentclass[11pt]{amsart}
\usepackage{amssymb,amsfonts,amsthm,amsmath,mathrsfs,xspace,hyperref}
\usepackage{graphicx}
\usepackage[francais,english]{babel}
\usepackage{inputenc}
\usepackage[T1]{fontenc}
\usepackage[centering]{geometry}
\usepackage{csquotes}
\usepackage{color}
\usepackage{enumerate}
\usepackage{mathabx}
\usepackage{stmaryrd}
\usepackage{enumitem}
  \newcommand{\R}{\ensuremath{\mathbb{R}}}%
  \newcommand{\Z}{\ensuremath{\mathbb{Z}}}%
	\newcommand{\Q}{\ensuremath{\mathbb{Q}}}%
  \newcommand{\N}{\ensuremath{\mathbb{N}}}%
				\newcommand{\X}{\ensuremath{\mathcal{X}}}%
                \renewcommand{\P}{\ensuremath{\mathcal{P}}}%

                       \newcommand{\mk}{\ensuremath{\mathfrak{m}}}%
                                \newcommand{\nk}{\ensuremath{\mathfrak{n}}}%
                                                                                                                                \newcommand{\hk}{\ensuremath{\mathfrak{h}}}%

        \renewcommand{\H}{\ensuremath{\mathcal{H}}}%
        \newcommand{\GL}{\ensuremath{\operatorname{GL}}}%
        \newcommand{\vol}{\ensuremath{\operatorname{vol}}}%
                \newcommand{\relvol}{\ensuremath{f}}%
                                \newcommand{\asdim}{\ensuremath{\operatorname{asdim}}}%

        \newcommand{\corps}{\ensuremath{\mathbb{K}}}%
                \newcommand{\Lie}{\ensuremath{\operatorname{Lie}}}%

	\newcommand{\fix}{\ensuremath{\textrm{Fix}}}%
    \newcommand{\acts}{\ensuremath{\curvearrowright}}%
  \newcommand{\sub}{\ensuremath{\operatorname{Sub}}}%
    \newcommand{\aut}{\ensuremath{\operatorname{Aut}}}%

			\newcommand{\Fit}{\ensuremath{\operatorname{Fit}}}
			\newcommand{\Res}{\ensuremath{\operatorname{Res}}}
			\newcommand{\FC}{\ensuremath{\operatorname{FC}}}

\newcommand{\ass}{\ensuremath{\mathrm{ass}}}
\newcommand{\ann}{\ensuremath{\mathrm{ann}}}
\newcommand{\rad}{\ensuremath{\mathrm{rad}}}

\newcommand{\p}{\ensuremath{\mathfrak{p}}}

\theoremstyle{definition}
  \newtheorem{defin}{Definition}[section]

	 \newtheorem{notation}[defin]{Notation}
  \newtheorem{main-question}{Question}
    \newtheorem{main-problem}[main-question]{Problem}

\theoremstyle{plain}
  \newtheorem{thm}[defin]{Theorem}
  \newtheorem{thmmain}{Theorem}

      \newtheorem*{main-thm}{Theorem}
    \newtheorem*{main-prop}{Proposition}
  \newtheorem{prop}[defin]{Proposition}
    \newtheorem{prop-def}[defin]{Proposition-Definition}
      \newtheorem{conjecture}{Conjecture}
         \newtheorem{conjecturebis}{Conjecture}
    \newtheorem{main-defin}{Definition}

  \newtheorem{cor}[defin]{Corollary}
    \newtheorem*{main-cor}{Corollary}

   \newtheorem{lem}[defin]{Lemma}

\theoremstyle{remark}
  \newtheorem{remark}[defin]{Remark}

  \newtheorem{example}[defin]{Example}
  
\begin{document}
  	
  \date{July 13, 2022}	
  
\title{Growth of actions of solvable groups}

  	\author{Adrien Le Boudec}
  	\address{CNRS, UMPA - ENS Lyon, 46 all\'ee d'Italie, 69364 Lyon, France}
  	\email{adrien.le-boudec@ens-lyon.fr}
  	
  	\author{Nicol\'as Matte Bon}
  	\address{
  		CNRS,
  		Institut Camille Jordan (ICJ, UMR CNRS 5208),
  		Universit\'e de Lyon,
  		43 blvd.\ du 11 novembre 1918,	69622 Villeurbanne,	France}
  	
  	\email{mattebon@math.univ-lyon1.fr}
  	
  	\thanks{Supported by the LABEX MILYON (ANR-10-LABX-0070) of Universite de Lyon, within the program "Investissements d'Avenir" (ANR-11-IDEX-0007) operated by the French National Research Agency.}

  	\maketitle
  	
\begin{abstract}
Given a finitely generated group $G$, we are interested in common geometric properties of all graphs of faithful actions of $G$. In this article we focus on their growth. We say that a group $G$ has a Schreier growth gap $f(n)$ if every faithful $G$-set $X$ satisfies  $\vol_{G, X}(n)\succcurlyeq f(n)$, where $\vol_{G, X}(n)$ is the growth of the action of $G$ on $X$. 	Here we study Schreier growth gaps for finitely generated solvable groups. 

We prove that if a metabelian group $G$ is either finitely presented or torsion-free, then $G$ has a Schreier growth gap $n^2$, provided $G$ is not virtually abelian. We also prove that if $G$ is a metabelian group of Krull dimension $k$, then $G$ has a Schreier growth gap $n^k$. For instance the wreath product $C_p \wr \Z^d$ has a Schreier growth gap $n^d$, and $\Z \wr \Z^d$ has a Schreier growth gap $n^{d+1}$. These lower bounds are sharp. For solvable groups of finite Pr\"ufer rank, we establish a Schreier growth gap $\exp(n)$, provided $G$ is not virtually nilpotent. This covers all solvable groups that are linear over $\Q$. Finally for a vast class of torsion-free solvable groups, which includes solvable groups that are linear, we establish a Schreier growth gap $n^2$.
\end{abstract}

\section*{Introduction}

Let $G = \langle S \rangle$ be a finitely generated group. We are interested in the  growth of faithful actions of $G$. A set $X$ endowed with a $G$-action will be called a $G$-set. The Schreier graph of the action of $G$ on $X$  is the graph $\Gamma(G, X)$  with vertex set $X$, and edges $(x, sx)$ for all $x \in X$ and $s \in S$. The terminology Schreier graph is usually used for transitive actions, but in the present setting we do \textit{not} require that actions are transitive (see e.g.\ Remark \ref{rmq-finite-orbits}). The \textbf{growth of the action} of $G$ on $X$ is defined by
 \[\vol_{G, X}(n)=\max_{x\in X} |S^n \cdot x|.\] In other words, $\vol_{G, X}(n)$ is the maximal cardinality of a ball of radius $n$ in $\Gamma(G, X)$. Given functions $f, g\colon \N \to \N$, we write $f(n)\preccurlyeq g(n)$ if there is a constant $C>0$ such that $f(n)\leq Cg(Cn)$, and $f(n)\simeq g(n)$ if $f(n)\preccurlyeq g(n)$ and $g(n) \preccurlyeq f(n)$. The function $\vol_{G, X}(n)$ does not depend on the choice of $S$ up to $\simeq$. 
 
 For the action of $G$ on itself by  left translations, the function $\vol_{G,X}(n)$ is the classical growth of the group $G$, which we denote by $\vol_G(n)=|S^n|$. The latter is one of the most widely studied asymptotic invariants of groups, after the results of Milnor \cite{Mil-sol}, Wolf \cite{Wolf}, Gromov \cite{Gromov-poly}, Grigorchuk \cite{Gri-growth} and many other developments. By contrast, few general results are  available on growth of actions.

 It is clear that for every  $G$-set $X$ we have $\vol_{G, X}(n) \preccurlyeq  \vol_G(n)$, but various groups admit faithful actions for which the function $\vol_{G, X}(n)$ is much smaller. Classical instances of groups having natural actions with $\vol_{G, X}(n)\simeq n$ are B.H. Neumann's examples of continuously many non-isomorphic finitely generated groups  \cite{Neumann-manygroups}, the Grigorchuk group \cite{Bar-Gri-Hecke}, or the topological full group of a $\Z$-action on the Cantor space. Moreover actions of linear growth, and more generally the analysis of  graphs of actions, played a crucial role in the recent developments on topological full groups and other related groups, see  \cite{Ju-Mo, Nek-simple-dyn, Nek-frag} and the recent preprint \cite{BNZ}. Other examples of groups admitting faithful actions of linear growth are non-abelian free groups (an observation that can be attributed to Schreier \cite{Schreier}), and all right-angled Artin groups (and hence all their infinite subgroups) \cite{Salo}. 
 
 Conversely, it is natural to ask whether there are obstructions to the existence of actions of small growth. We introduce the following definition:

\begin{main-defin}
	Let $f\colon \N\to \R_+$. A finitely generated group $G$ has a {\textbf{Schreier growth gap} $f(n)$} if every faithful $G$-set $X$ satisfies  $\vol_{G, X}(n)\succcurlyeq f(n)$.
\end{main-defin}

For every infinite group $G$ we always have $\vol_{G, X}(n)\succcurlyeq n$ for every faithful $G$-set $X$. Hence we are interested in Schreier growth gaps $f(n)$ where the function $f$ is super-linear, meaning that $f(n)/n$ is unbounded. The word \enquote{gap} implicitly refers to that situation. 

Every infinite group $G$ with Kazhdan property (T) has a Schreier growth gap $\exp(n)$. This is a standard consequence of the fact that for every $G$-set $X$ the unitary representation of $G$ on $\ell^2(X)$ has a spectral gap. This remark can be traced back to Kazhdan, and is pointed out by Gromov in \cite[Remark 0.5.F]{Gro-asdim}  (see also \cite[Th. B]{Stuck-growth}, and \cite[\S 4]{Ju-dlS} for a statement that does not assume transitivity of actions, and also \cite[\S 7]{Cor-MathZ}). In a different direction, Schreier growth gaps were established in \cite{MB-graph-germs} and \cite{LB-MB-comm-lemm}, respectively in the setting of topological full groups of \'etale groupoids and of branch groups acting on rooted trees.

Given the above manifestations of this phenomenon, we want to initiate a systematic study of Schreier growth gaps for  finitely generated groups. The purpose of this article is to establish various Schreier growth gaps among finitely generated solvable groups. Beyond the fact that the role played by solvable groups in the classical theory of growth of groups suggests to consider this setting, it is natural here to consider \enquote{small} groups $G$, in view of the fact that if a group $G$ has a Schreier growth gap $f(n)$, then the same is true for any group having $G$ as a subgroup. 

\subsection*{The method: non-foldable subsets and confined subgroups} Before discussing our main results, we outline our approach to study Schreier growth gaps. It is based on the following notion of  independent interest.

\begin{main-defin}
A subset $\mathcal{L}$ of a group $G$ is \textbf{non-foldable} if for every faithful $G$-set $X$, for every finite subset $\Sigma \subset \mathcal{L}$ there exists $x\in X$ such that the orbital map $g\mapsto gx$ is injective on $\Sigma$. 
\end{main-defin} 

 Roughly speaking, a subset $\mathcal{L}$ is non-foldable if for every faithful $G$-set $X$, the graph $\Gamma(G, X)$ contains Lipschitz embedded copies of arbitrarily large finite subsets of $\mathcal{L}$. In terms of growth, this implies that $\vol_{G, X}(n)$ must be at least equal to the maximal cardinality of a ball of radius $n$ in $\mathcal{L}$, where $\mathcal{L}$ is equipped with the induced metric from $G$. Our method to establish Schreier growth gaps consists in exhibiting non-foldable subsets that are as large as possible. Non-foldable subsets provide  information on the geometry of the graphs $\Gamma(G, X)$ beyond the notion of growth. For example they also provide lower bounds for their {asymptotic dimension}. As the growth, the asymptotic dimension of the graph $\Gamma(G, X)$ is an invariant of the action of $G$ on $X$, that is monotone when passing to a finitely generated subgroup. The non-foldable subsets that we exhibit are natural and explicit, so that our results also provide computable lower bounds for the asymptotic dimension of $\Gamma(G, X)$. 

The study of non-foldable subsets of a group $G$ is crucially related to the study of closed $G$-invariant subsets for the conjugation action of $G$ on the space  $\sub(G)$ of subgroups of $G$. The latter is a compact space  with the topology induced from the set $2^G$ of all subsets of $G$. Recall that a subgroup of $H$ of a group $G$ is \textbf{confined} if the closure of the $G$-orbit of $H$ in $\sub(G)$ does not contain the trivial subgroup $\{1\}$. Equivalently, $H$ is confined if there is a finite set $P$ of non-trivial elements of $G$ which intersects all conjugates of $H$. Confined subgroups are natural generalisations of uniformly recurrent subgroups (URSs) of Glasner and Weiss \cite{GW-urs}. Confined subgroups and URSs were studied and found applications recently in \cite{Kenn-urs,LBMB-sub-dyn,Elek-simple-alg,MB-Tsan, LB-lattices,MB-graph-germs,Fraczyk-urs,LB-MB-comm-lemm,LB-MB-confined-ht,CLB-commens,Bou-Houd}. A common point in \cite{LBMB-sub-dyn, MB-graph-germs, LB-MB-comm-lemm} is a complete classification, or a strong structural result, of confined subgroups and URSs for certain families of groups defined by an action by homeomorphisms. This  global rigidity behaviour for confined subgroups and URSs leads to Schreier growth gaps \cite{MB-graph-germs, LB-MB-comm-lemm}. A major difference in the present article is that solvable groups may admit a large pool of confined subgroups and URSs. This was already illustrated by Glasner--Weiss \cite{GW-urs}. In our present setting it is essential to first guess what will be the relevant non-foldable subsets $\mathcal{L}$, which has the consequence of restricting the confined subgroups that need to be studied.

We shall now state our results. For simplicity here we state the conclusions that we draw about Schreier growth gaps, and refer the reader to the core of the article for statements about non-foldable subsets. 

\subsection*{Polycyclic groups}

At this point it is worth mentioning that if $G$ is an infinite virtually abelian group, then $G$ always admits a faithful $G$-set with linear growth (Proposition \ref{prop-virtually-abelian}), and thus does not satisfy any Schreier growth gap. Thus virtually abelian groups should be considered as  trivial  for the problem considered in this paper, and shall systematically be excluded.

The following proposition treats the case of polycyclic groups.

\begin{main-prop}[Proposition \ref{prop-n4-growth} and Corollary \ref{cor-poly-growth}] Let $G$ be a polycyclic group. 
	\begin{enumerate}[label=\roman*)]
		\item \label{item-prop-noeth-exp} If $G$ is not virtually nilpotent, then $G$ has a Schreier growth gap  $\exp(n)$.
		\item If $G$ is virtually nilpotent and not virtually abelian, then $G$ has a Schreier growth gap $n^4$.
	\end{enumerate}
\end{main-prop}

The proof of this result is elementary. Recall that polycyclic groups are the solvable groups that have all their subgroups finitely generated. This property notably implies that the space $\sub(G)$ of subgroups of $G$ is a countable compact space. Simple compactness considerations then allow to reduce the understanding of the growth of actions of nilpotent and polycyclic groups to the classical results on the growth of these groups. We refer to Section \ref{s-noetherian} for details. Here we shall point out that this approach has limited scope, and does not generalize to other solvable groups.

\subsection*{Metabelian groups}

%The next step to consider is the case of metabelian groups. Even for these groups, the situation becomes much more diversified compared to the polycyclic case. One major difference is that there are groups of exponential growth which admit faithful actions of linear growth (and also actions with almost arbitrarily prescribed growth, see \S \ref{subsec-behavior}). The archetype of such an example is the lamplighter group $C_p\wr \Z$, where $C_p$ is the cyclic group of prime order $p$.  Lamplighter actions with linear growth played a role e.g.\ in \cite{Mat-exp,DFG-sol,GW-urs}. Relying on a result of Olshanskii \cite{Olshanskii-KK}, one can see that more generally every metabelian group of the form $N \rtimes \Z$, where $N$ is a locally finite group, admits a faithful action of linear growth (Proposition  \ref{prop-growth-ol}). Hence there is no Schreier growth gap that is uniform for all (non-virtually abelian) metabelian groups, unlike in Proposition \ref{prop-intro-noeth}.  Nevertheless we prove that such gaps do hold for some  natural classes of metabelian groups:

The next step to consider is the case of metabelian groups. Even for these groups, the situation becomes much more diversified compared to the polycyclic case. One major difference is that there are groups of exponential growth which admit faithful actions of linear growth. The archetype of such an example is the lamplighter group $C_p\wr \Z$, where $C_p$ is the cyclic group of  order $p$. (In fact the lamplighter group admits action of almost arbitrarily prescribed growth, see \S \ref{subsec-behavior}). Hence there is no Schreier growth gap that is uniform for all (non-virtually abelian) metabelian groups, unlike in the polycyclic case. Nevertheless we prove that such uniform gaps do hold in the following two situations:

%\begin{thmmain} \label{thm-intro-metab-presfin}
%	Let $G$ be a finitely generated metabelian group that is not virtually abelian. If $G$ is finitely presented, then $G$ has a Schreier growth gap  $n^2$. 
%\end{thmmain}
%
%
%\begin{thmmain} \label{t-intro-metab-torsion-free}
%	Let $G$ be a finitely generated metabelian group that is not virtually abelian. If $G$ is torsion-free, then $G$ has a Schreier growth gap  $n^2$. 
%\end{thmmain}

\begin{thmmain} \label{thm-intro-metab-quad}
	Let $G$ be a finitely generated metabelian group that is not virtually abelian. Suppose that $G$ satisfies at least one of the following: \begin{enumerate}[label=\roman*)]
		\item $G$ is finitely presented;
		\item $G$ is torsion-free.
	\end{enumerate}
Then $G$ has a Schreier growth gap  $n^2$. 
\end{thmmain}

We note that the quadratic bound is optimal both in the finitely presented case and in the torsion-free case, as it is realized respectively by Baumslag's finitely presented metabelian groups (\S \ref{subsec-Baumslag}) and by the wreath product $\Z \wr \Z$ (see below). 

Beyond the classical lamplighter group, every wreath product $G=A\wr B$ of two finitely generated abelian groups admits faithful actions of polynomial growth. A natural example is the action of $G$ on the Cartesian product $X=B\times A$ called the \textbf{standard wreath product action} (whose definition is recalled in \S \ref{s-wreath-actions}). With natural choices of generators, the graph of this action is obtained by taking a copy of the Cayley graph of $A$ and attaching to each vertex a copy of the Cayley graph of $B$. When $A = B = \Z$ this graph  is a comb (Figure \ref{fig-comb}). The growth of the standard action of $A\wr B$  is equivalent to the growth of the abelian group $B\times A$. So for the lamplighter group $G=C_p \wr \Z^d$,  the standard wreath product action satisfies $\vol_{G, X}(n)\simeq n^d$, while for $G=\Z \wr \Z^d$ we have $\vol_{G, X}(n)\simeq n^{d+1}$. 

The next result states that for these groups the growth of the standard action of $G$ is the minimal possible growth of all faithful $G$-actions:

\begin{main-thm}[Theorem \ref{thm-wreath}] 
	For every $d \geq 1$, the following hold:
	\begin{enumerate}[label=\roman*)]
		\item The group $C_p\wr \Z^d$ has a Schreier growth gap $n^d$.
		
		\item The group $\Z \wr \Z^d$ has a Schreier growth gap  $n^{d+1}$. 
	\end{enumerate}
\end{main-thm}

This result is a simple illustration of a more general theorem about metabelian groups, in which we obtain a Schreier growth gap in terms of an algebraic invariant called the Krull dimension. Let $G$ be a finitely generated metabelian group. Recall that whenever $1\to M\to G\to Q\to 1$ is a short exact sequence  such that $M$ and $Q$ are abelian, $M$ can be seen as a finitely generated module over the group ring $\Z Q$. This point of view plays a crucial role in the study of metabelian groups since the seminal work of Hall \cite{Hall}. The Krull dimension of the $\Z Q$-module $M$ does not depend on $(M,Q)$ provided that $G$ is not virtually abelian (see \S \ref{s-Krull}). This positive integer is called the \textbf{Krull dimension of $G$}, and was introduced by Cornulier. It implicitly appears in  \cite{Cornulier-CBrank}. This notion was further studied by Jacoboni, who established estimates for the return probability for the random walk on a metabelian group $G$ in terms of the Krull dimension of $G$ \cite{Lison}. 

The following result relates the Krull dimension of a metabelian group to the possible growth of all its faithful actions:

\begin{thmmain}\label{thm-metabelian-krull}
	Let $G$ be a finitely generated metabelian group which is not virtually abelian, and let $k = \dim_{\mathrm{Krull}}(G)$. Then $G$ has a Schreier growth gap $n^k$. 
\end{thmmain}

We have $\dim_{\mathrm{Krull}}(C_p \wr \Z^d)=d$ and $\dim_{\mathrm{Krull}}(\Z \wr \Z^d)=d+1$, so the above statement for wreath products follows from Theorem \ref{thm-metabelian-krull}. As another example,  for the free metabelian group $\mathbb{FM}_d$ on $d$ generators, we obtain a Schreier growth gap $n^{d+1}$. This estimate is sharp, in the sense that there is an action realizing the lower bound. The case of $\mathbb{FM}_d$ turns out to be very particular among free solvable groups, as for the free solvable group $\mathbb{FS}_{d, \ell}$ of rank $d$ and solvability length $\ell \ge 3$, we prove a Schreier growth gap $\exp(n)$ (Theorem \ref{thm-free-solvable}). 

The main tool in the proof of the above results is Theorem \ref{thm-explicit-metab} below, which can be described as our main result on metabelian groups. Given a finitely generated group in a short exact sequence $1\to M\to G\to Q\to 1$ where $M,Q$ are abelian, Theorem \ref{thm-explicit-metab} provides a  criterion which guarantees that $G$ has a Schreier growth gap $n^d$, where the exponent $d$ depends on  the  $\Z Q$-module $M$. Beyond wreath products and free metabelian groups, this estimate is sharp  for other interesting families of metabelian groups. We illustrate this in  \S \ref{s-metabelian-first-examples} with Baumslag's finitely presented metabelian groups  \cite{Baumslag-group}. In this case the lower bound provided by Theorem \ref{thm-explicit-metab} is quadratic. Here we mention that these groups have Krull dimension one, so that there are examples for which Theorem \ref{thm-explicit-metab} provides a better bound than Theorem \ref{thm-metabelian-krull}.

\subsection*{Solvable groups}
We now discuss solvable groups of higher length.  A prominent class of finitely generated solvable groups is the class of groups of finite Pr\"ufer rank. Recall that a group $G$ has finite  \textbf{Pr\"ufer rank} if there exists $k\in \N$ such that every finitely generated subgroup of $G$ is generated by at most $k$ elements (the rank of $G$ is then the least $k$ with this property). In the sequel we abbreviate finite Pr\"ufer rank by  \textbf{finite rank}. The class of solvable groups of finite  rank contains polycyclic groups, but it is much richer. Basic examples of groups of finite rank that are not polycyclic are the Baumslag-Solitar groups $\mathrm{BS}(1,n)$. Every finitely generated solvable group that is linear over $\Q$ has finite rank. Examples of groups that are not of finite rank are wreath products $A \wr B$ where $A$ is non-trivial and $B$ is infinite. The class of finitely generated solvable groups of finite rank admits several algebraic characterisations: it coincides with the minimax groups \cite[Ch.\ 5]{Lennox-Rob}, and by a theorem of P.\ Kropholler it also coincides with those solvable groups that do not admit a lamplighter group $C_p \wr \Z$ as a subquotient \cite{Kropholler84}. On a more geometric perspective, finite rank solvable groups have been studied by Pittet--Saloff-Coste in \cite{Pittet-Saloff-Coste} and Cornulier--Tessera in \cite{CT-Banach}. The following result asserts that these groups satisfy the strongest possible Schreier growth gap.

\begin{thmmain}\label{thm-intro-prufer}
	Let $G$ be a finitely generated solvable group of finite rank, and assume that $G$ is not virtually nilpotent. Then $G$ has a Schreier growth gap  $\exp(n)$. 
\end{thmmain}

So this result extends the previously mentioned result for polycyclic groups to finite rank solvable groups. However we point out that the proof here is substantially more involved than in  the polycyclic case. First we reduce the general case to the case of torsion-free finite rank solvable groups. Such a group is linear over $\Q$, and its nilpotent radical is subject to Malcev theory on torsion-free divisible nilpotent groups. We deduce Theorem \ref{thm-intro-prufer} from a general result on groups that are extensions $1\to N\to G\to Q\to 1$ where $N$ is a nilpotent group of finite rank and the action of $Q$ on $N^{ab}\otimes \Q$ satisfies a certain irreducibility condition (Theorem \ref{t-strongly-irreducible}), together with a result of independent interest which asserts that finite rank solvable groups always contain subgroups of a particular form, to which the previous result applies (Proposition \ref{prop-non-virt-nilp-subgroup}).

\bigskip

The last problem that we discuss is whether the quadratic Schreier growth gap for torsion-free metabelian groups from Theorem \ref{thm-intro-metab-quad} extends to torsion-free solvable groups. We conjecture that this is the case: 

\begin{conjecture} \label{conj-intro-torsion-free}
	Let $G$ be a finitely generated solvable group which is virtually torsion-free, and not virtually abelian. Then $G$ has a Schreier growth gap  $n^2$. 
\end{conjecture}

We prove that the conjecture is true in many situations, which cover all familiar classes of torsion-free solvable groups.  The first observation to make is that Theorem \ref{thm-intro-metab-quad}  immediately implies that if a torsion-free group $G$ contains a finitely generated metabelian subgroup which is not virtually abelian, then $G$ has a Schreier growth gap  $n^2$. Classical arguments imply that this situation covers all groups which are virtually nilpotent-by-abelian. By a theorem of Malcev, this includes all solvable linear groups (i.e.\ isomorphic to a subgroup of $\GL(n, \corps)$, where $\corps$ is a field). Hence the conjecture is true for these groups:

\begin{main-cor}[Corollary  \ref{cor-linear-case-quadratic}] 
Let $G$ be a finitely generated solvable linear group which is virtually torsion-free. If $G$ is not virtually abelian, then $G$ has a Schreier growth gap  $n^2$. 
\end{main-cor}

The next result proves the conjecture under a reinforcement of the torsion-free assumption on $G$. Note that if a group $G$ admits a series $\{1\}=N_0\unlhd N_1\unlhd \cdots \unlhd N_k=G$ such that the successive quotients $N_{i+1}/N_i$ are torsion-free, then $G$ is torsion-free. We refer to \S \ref{subsec-fittingseries} for the definition of the Fitting series.  

\begin{main-thm}[Theorem \ref{thm-fitting-torsionfree}] 	Suppose that $G$ is a finitely generated solvable group such that the successive quotients in the Fitting series of $G$ are torsion-free, and $G$ is not virtually abelian. Then $G$ has a Schreier growth gap  $n^2$. 
\end{main-thm}

 The proof of that result does not only rely on the case of metabelian groups. It involves a rather technical mechanism that allows in some cases to lift the desired conclusion from a quotient to the ambient group (see \S  \ref{subsec-lift-exp} and \S \ref{subsec-fittingseries}).

While the above results establish the conjecture under additional assumptions on $G$, it is also natural to add assumptions on the actions. The following result asserts that Conjecture \ref{conj-intro-torsion-free} is true if we restrict to transitive actions, or more generally to actions with finitely many orbits.

\begin{main-thm}[Theorem \ref{t-torsion-free-quasitransitive}]
	Let $G$ be a finitely generated solvable group which is virtually torsion-free and not virtually abelian. Let $X$ be a faithful $G$-set such that the action of $G$ on $X$ has finitely many orbits. Then $\vol_{G, X}(n) \succcurlyeq  n^2$. 
\end{main-thm}

\subsection*{Applications: subgroups of topological full groups}

We end this introduction by mentioning that Schreier growth gaps provide a natural quantitative obstruction to the existence of embeddings between groups. Explicitly, if a group $L$ is known to admit a faithful action of growth $g(n)$ and if a group $G$ has a Schreier growth gap $f(n)$ with $f(n) \npreccurlyeq g(n)$, then $G$ cannot embed into $L$ (see Proposition \ref{prop-monoton}). Interesting examples of groups $L$ that naturally come with an action of possibly small growth are topological full groups of group actions on the Cantor set. The subgroup structure of such groups is in general quite mysterious. Our results can be immediately applied to obtain restrictions on the solvable subgroups of various topological full groups. As an illustration, we state here one example of such an application.  Recall that Matui showed that for every minimal $\Z$-action on the Cantor set which is not an odometer, the topological full group contains the lamplighter group $C_2\wr \Z$ \cite{Mat-exp}. By contrast, our result on Schreier growth gaps of wreath products has the following consequence:

%\begin{main-cor}
%The groups $\Z \wr \Z$ and  $C_p \wr \Z^2$ do not embed into the topological full group of any $\Z$-action on the Cantor set. 
%\end{main-cor}

\begin{main-cor}[see \S \ref{subsubsection-full-gp}]
	Let $d \geq 1$ and let $A$ be a non-trivial finitely generated abelian group. If the wreath product $A \wr \Z^d$ embeds into the topological full group of a $\Z$-action on the Cantor set, then $A$ is finite and $d = 1$.
\end{main-cor}
 
The fact that the group $C_p \wr \Z^2$ does not embed into the topological full group of the full shift over $\Z$ was conjectured in \cite{Salo}. We refer to \S \ref{s-non-embeddings} for  additional applications of our results in this setting (notably Corollary \ref{cor-nV-poly}).
 
  \subsection*{Guidelines}
  Sections  \ref{s-finite-rank}--\ref{sec-metab}--\ref{sec-torsionfree} are the core of the article, and our main results stated in this introduction are proven there.

 \subsection*{Acknowledgements}
 We are grateful to A. Erschler, J. Frisch and T. Zheng for interesting conversations on Schreier graphs of the lamplighter group which inspired Proposition \ref{p-lamplighter-given-growth}. We also thank Y. Cornulier and P. de la Harpe for useful comments on a preliminary version of this work.

 \setcounter{tocdepth}{1}
 \tableofcontents
 
\section{Preliminaries} \label{s-preliminaries}

\subsection{Notation} \label{subsec-notation}

Let $G$ be a group. We denote by $G^{(i)}$ the derived series of $G$, defined inductively by $G^{(1)} = G' = [G,G]$ and $G^{(i+1)} = [G^{(i)},G^{(i)}] $. The abelianization $G/G'$ is denoted $G^{ab}$. The lower central series of $G$ is denoted by  $\gamma_i(G)$. So $\gamma_1(G)=G$ and $\gamma_{i+1}(G)=[\gamma_{i}(G),G]$ for all $i \geq 1$. The upper central series is denoted $Z_i(G)$, with the convention $Z_1(G)=Z(G)$ and $Z_{i+1}(G)/Z_i(G)=Z(G/Z_i(G))$ for all $i$.

The \textbf{Fitting subgroup} is denoted $\Fit(G)$. It is the subgroup generated by all nilpotent normal subgroups of $G$. The \textbf{FC-center} $\FC(G)$ is the subgroup of $G$ consisting of elements with a finite conjugacy class. 

The set of elements of finite order in a group $G$ will be denoted $T(G)$. Recall that when $G$ is nilpotent, $T(G)$ is a subgroup of $G$, which is finite if $G$ is finitely generated.

Recall that a set $X$ on which $G$ acts is called a $G$-set. All $G$-sets are supposed to be non-empty. We say that $X$ is faithful, or transitive, whenever the action of $G$ on $X$ has this property. 
 
 For $f,g, \colon \N \to \N$, we write $f\preccurlyeq g$ if there exists a constant $C>0$ such that $f(n)\leq Cg(Cn)$ for sufficiently large $n$. If $f\preccurlyeq g$ and $g\preccurlyeq f$, we write $f\simeq g$. 
 
\begin{notation}
Let $G$ be a group with an abelian normal subgroup $M$. Since $M$ acts trivially by conjugation on itself, the conjugation action of $G$ on $M$ factors through $Q = G/M$, so that $M$ naturally has the structure of a $\mathbb{Z} Q$-module. When adopting this point of view, we will use additive notation for $M$, and the notation $qm$ for $q \in Q$ and $m \in M$ will be used for the module operation of $q$ on $m$.
\end{notation}

\subsection{Graphs of actions}

Let $G$ be a finitely generated group, and fix a finite symmetric generating subset $S$. If $X$ is a $G$-set, we denote by $\Gamma(G, X)$ the graph whose vertex set is $X$,  and for every $x\in X$ and $s\in S$ there is an edge connecting  $x$ to  $sx$. The graph $\Gamma(G, X)$ is called the \textbf{Schreier graph of the action} of $G$ on $X$. Note that  $\Gamma(G, X)$ is not connected in general: its connected components are the $G$-orbits in $X$.  We intentionally omit $S$ in the notation, as we will only be interested in properties of these graphs that do not depend on the choice of $S$. We consider the simplicial distance $d$ on  $\Gamma(G, X)$, where $d(x, y)$ is defined as the length of the shortest path from $x$ to $y$ (ignoring orientation of edges), with the convention that $d(x, y)=+\infty$ if $x, y$ lie in different connected components. Distinct generating subsets yield bi-Lipschitz equivalent metrics.

 \begin{remark} \label{rmq-finite-orbits}
An extreme case that is covered by our setting is the case where all $G$-orbits in $X$ are finite, so that  $\Gamma(G, X)$ is the disjoint union of finite graphs. This corresponds to $G$-actions that are given by a family of finite index subgroups $(G_i)$ of $G$, where $G$ acts on the union of coset spaces $G/G_i$. Note that this action is faithful if and only if there is no non-trivial normal subgroup of $G$ contained in $\bigcap_i G_i$.
 \end{remark}

\subsection{Growth of actions}  \label{subsec-growth}

Let $(X, d)$ be a metric space (where we allow the distance $d$ to take the value $+\infty$).  We denote by $B_x(n)$ the ball of radius $n$ around a point $x\in X$.  We  say that $(X, d)$ is \textbf{uniformly locally finite} if for every $n$ we have $\sup_{x\in X} |B_x(n)|<\infty$. For example, any graph of bounded degree is uniformly locally finite. If $(X, d)$ is uniformly locally finite, we denote $\vol_X(n)=\sup_{x\in X} |B_x(n)|$. This function is invariant under the equivalence relation $\simeq$  if one replaces  $d$ by a quasi-isometric metric.

In the sequel $G$ is a finitely generated group, and $S$ a finite symmetric generating set. Let $d_S$ be the associated right-invariant word metric on $G$. Recall that the growth of $G$ is denoted $\vol_G(n) = |S^n|$.

\begin{notation}
	For a subset $\mathcal{L}$ of $G$, we write $\relvol_{(G, \mathcal{L})}:=\vol_{(\mathcal{L}, d_S|_{\mathcal{L}})}$. 
\end{notation}

\begin{remark}
If $\mathcal{L} = H$ is a subgroup of $G$, then for every $n$ all balls $B_h(n)$ are isometric since $H$ acts transitively on itself by right translations. Hence in this case $\relvol_{(G,H)}(n) = |H \cap B_e(n)|$ is the classical relative growth of $H$ in $G$.
\end{remark}

\begin{defin}
The \textbf{growth of the action} of $G$ on a $G$-set $X$ is $\vol_{G, X}(n)= \vol_{(\Gamma(G, X), d)}(n)$. 
\end{defin}

Equivalently,  \[\vol_{G, X}(n)=\max_{x\in X} |S^n \cdot x|.\] Since distinct generating subsets of $G$ yield bi-Lipschitz equivalent metrics on $\Gamma(G, X)$, $\vol_{G, X}(n)$ does not depend on $S$ up to the equivalence relation $\simeq$.

The following proposition establishes basic properties of growth of actions. The proof is straightforward, and we omit it.

\begin{prop}[Monotonicity] \label{prop-monoton}
Let $G$ be a finitely generated group, and $X, Y$ $G$-sets.
\begin{enumerate}[label=\roman*)]
\item If $H$ is a finitely generated subgroup of $G$, then $\vol_{H, X}(n)\preccurlyeq\vol_{G, X}(n)$.
\item If there is a surjective $G$-equivariant map $X \to Y$, then $\vol_{G, X}(n) \succcurlyeq  \vol_{G, Y}(n)$. \item If there is a $G$-equivariant map $X\to Y$ whose fibers have uniformly bounded cardinality, then $\vol_{G, X}(n)\preccurlyeq\vol_{G, Y}(n)$. In particular if there is an injective $G$-equivariant map $X\to Y$, then $\vol_{G, X}(n)\preccurlyeq\vol_{G, Y}(n)$.
\end{enumerate}
\end{prop}

The next lemma relates the growth of actions of a group and of  its finite index subgroups. Recall that when $H$ is a subgroup of $G$, then to every $H$-set  $X$ one can associate a $G$-set, called the \textbf{induced $G$-set}. $\operatorname{Ind}_H^G(X)$ is defined by  $\operatorname{Ind}_H^G(X):=(G\times X)/H$, where the quotient is taken with respect to the diagonal action $h\cdot (g, x)=(gh^{-1}, hx)$. The action of $G$ on $G\times X$ given by $g_1\cdot (g_2, x)=(g_1g_2, x)$ descends to an action on $\operatorname{Ind}_H^G(X)$.

\begin{prop} \label{prop-finite-index}
Let $G$ be a group and $H$ be a subgroup of finite index in $G$. Then for every $G$-set $X$ we have $\vol_{G, X}(n) \simeq \vol_{H, X}(n)$. Conversely for every $H$-set $X$, the induced $G$-set $Y:=\operatorname{Ind}_H^G(X)$ satisfies $\vol_{G, Y}(n)\simeq \vol_{H, X}(n)$. 
\end{prop}

\begin{proof}
The first claim is straightforward, and we only justify the second claim. Let $N$ be the intersections of all conjugates of $H$, which is normal and of finite index in $G$.  Set
set $Y:=\operatorname{Ind}_H^G(X)$ and consider the natural $G$-equivariant projection $p\colon Y\to G/H$.  Since $N$ acts trivially on $G/H$, it preserves each fiber of the map $p$. The fiber $p^{-1}(H)$ is naturally a copy of $X$, and the action of $N$ on it coincides with the action obtained by restricting the $H$-action, so that by the first claim we have $\vol_{N, p^{-1}(H)}(n)\simeq \vol_{H, X}(n)$. If $gH$ is another $H$-coset, the map $x\mapsto gx$ induces an identification of $p^{-1}(H)$ to $p^{-1}(gH)$ which is $N$-equivariant up to the automorphism of $N$ induced by conjugation by $G$,   so that we also have $\vol_{N, p^{-1}(gH)}(n) \simeq \vol_{H, X}(n)$. Since $Y$ is the disjoint union of the finitely many fibers of the map $p$, we obtain (using the first claim again) that $\vol_{G, Y}(n)\simeq \vol_{N, Y}(n)\simeq \vol_{H, X}(n)$. 
\end{proof}

%\subsection{Direct products} \label{subsec-directproducts}
%
%For future reference we isolate the following remark.
%
%\begin{lem} \label{l-product-actions}
%	Let $G$ be a finitely generated group which is a direct product $G=Q_1\times \cdots \times Q_k$. For each $i$, let $X_i$ be a faithful $Q_i$-set. Then $X=\sqcup_i X_i$ is a faithful $G$-set, where $G$ acts on each $X_i$ through $Q_i$, and we have 
%	\[\vol_{G, X}(n)=\max_{i=1}^k \vol_{Q_i, X_i}(n).\]
%	In particular if   all $Q_i$s have polynomial growth with $\vol_{Q_i}(n)\simeq n^{d_i}$, then there is a faithful $G$-set $X$ such that  $\vol_{G, X}(n)\simeq n^d$, where $d=\max_i d_i$. 
%\end{lem}
%
%In particular we have the following for virtually abelian groups. 
%
%\begin{prop}\label{prop-virtually-abelian}
%	If $G$ is an infinite finitely generated virtually abelian group,  then there exists a faithful $G$-set $X$ such that $\vol_{G, X}(n)\simeq n$. 
%\end{prop}
%
%\begin{proof}
%	The group $G$ contains a finite index subgroup isomorphic to $\Z^d$ for some $d\ge 1$. Thus the conclusion follows from Proposition \ref{prop-finite-index}, by considering the action of $\Z^d$ on $\sqcup_{i=1}^d \Z$ as in Lemma \ref{l-product-actions} 
%\end{proof}

\subsection{Confined of subgroups}

We denote by $\sub(G)$ the space of subgroups of $G$, endowed with the topology inherited from $\left\lbrace 0,1\right\rbrace ^G$. The following simple lemma shows that the growth of actions is well-behaved with respect to the topology on $\sub(G)$.

\begin{lem} \label{lem-growth-closure}
	Let $X$ be a $G$-set, and let \[ \mathcal{S}(X) = \overline{ \left\lbrace G_x  \, : \, x \in X \right\rbrace } \subseteq \sub(G). \] Then for every $H \in \mathcal{S}(X)$, we have $\vol_{G, G/H}(n) \preccurlyeq \vol_{G, X}(n)$. 
\end{lem}

\begin{proof}
	Let $H\in \mathcal{S}(X)$ and let $(x_k)$  a sequence of points such that $(G_{x_k})$ converges to $H$ in $\sub(G)$. For every coset $gH\in G/H$ and every $n$, there exists $k_0$ such that for every $k\geq k_0$ the ball of radius $n$ centered at $gH$ in the graph $\Gamma(G, G/H)$ is isomorphic to  the ball of radius $n$ centered at $x_k$ in $\Gamma(G, X)$. Thus its cardinality does not exceed $\vol_{G, X}(n)$, showing that $\vol_{G, G/H}(n)\preccurlyeq\vol_{G, X}(n)$. \qedhere
\end{proof}

As a consequence of Lemma \ref{lem-growth-closure}, if the trivial subgroup $\{1\}$ belongs to the closure of the stabilisers $\{G_x, x\in X\}$, then $\vol_{G}(n)\preccurlyeq\vol_{G, X}(n)$ and thus $\vol_{G, X}(n)\simeq \vol_{G}(n)$.

Recall that a subgroup $H$ of $G$ is \textbf{confined} if the closure of the $G$-conjugacy class of $H$ in $\sub(G)$ does not contain the trivial subgroup $\left\lbrace 1\right\rbrace$. Explicitly, this means that there exists a finite subset $P$ of $G$ consisting of non-trivial elements such that $g Hg^{-1} \cap P \neq \emptyset$ for all $g \in G$. We will consider more generally the case where we only take into account the conjugation action of a given subgroup $L$ of $G$:

\begin{defin} \label{def-k-conf-subset}
	Let $G$ be a group, $L$ a subgroup of $G$. A subgroup $H$ of $G$ is \textbf{confined by $L$} if the closure of the $L$-conjugacy class of $H$ in $\sub(G)$ does not contain the trivial subgroup. Equivalently, if there exists a finite subset $P \subset G \setminus \left\lbrace 1\right\rbrace $ such that $g Hg^{-1} \cap P \neq \emptyset$ for all $g \in L$. We say that such a subset $P$ \textbf{is confining for $(H,L)$}.
\end{defin}

\begin{notation}
	We denote by $S_G(P,L)  \subset \sub(G)$ the set of subgroups $H$ of $G$ such that $P$ is confining for $(H,L)$.
\end{notation}

\subsection{Non-foldable subsets} \label{subsec-non-foldable}

We now introduce the notion of non-foldable subset of a group $G$. The motivation is that non-foldable subsets provide geometric information on the graphs of actions of $G$. In particular they provide lower bounds for certain asymptotic invariants of these graphs: see Lemma \ref{lem-exp-subset-growth} and Proposition \ref{p-asdim-non-foldable}.

\begin{defin} \label{d-non-foldable}
	Let $G$ be a group and $\mathcal{L}$ a subset of $G$. 
	\begin{enumerate}[label=\roman*)]
		\item Let $X$ be a $G$-set. We say that $\mathcal{L}$ is \textbf{non-folded in $X$} if for every finite subset $\Sigma$ of $\mathcal{L} $, there exists $x \in X$ such that the map $\Sigma \to X$, $g \mapsto g x$, is injective.
		\item We say  that $\mathcal{L}$ is \textbf{non-foldable} if $\mathcal{L} $ is non-folded in $X$ for every faithful $G$-set $X$. 
		\end{enumerate}
\end{defin}

The following lemma reformulates the condition that a given subset of $G$ is non-foldable in terms of confined of subgroups.

\begin{lem}[Non-foldable subsets and confined subgroups] \label{l-non-foldable-confined}
	Let $G$ be a group, and $\mathcal{L}$ be a subset of $G$. The following are equivalent:
	\begin{enumerate}[label=(\roman*)]
		\item \label{i-non-foldable} $\mathcal{L}$ is a non-foldable subset of $G$.
		\item  \label{i-conf} for every finite subset  $\Sigma \subset \mathcal{L} $, we have \[\bigcap_{H\in S_G(P, G)} \! \! H\neq \{1\}, \] where $P=\{g^{-1}h\colon g, h\in \Sigma, g\neq h\}$.
	\end{enumerate}
\end{lem}

\begin{proof}
	Assume that \ref{i-conf} holds, and let $X$ be a faithful $G$-set. Assume by contradiction that \ref{i-non-foldable}  does not hold. Let $\Sigma\subset \mathcal{L}$ be a subset such that for every $x\in X$, the map $g\mapsto gx$ is not injective on $\Sigma$, and let $P$ be the corresponding set of differences as in the statement. Then $G_x\cap P\neq \varnothing$ for  every $x\in X$, and thus $G_x\in S_G(P, G)$ for every $x$. Hence $\cap_x G_x$ contains $\bigcap_{S_G(P, G)} \! \! H$, and is therefore non-trivial. This contradicts the fact that $X$ is faithful.
	
	Assume now that \ref{i-non-foldable} holds. Fix a finite subset $\Sigma\subset \mathcal{L}$ and consider the $G$-set $X:=\sqcup_{H\in S_G(P, G)} G/H$,  for the subset  $P$ associated to $\Sigma$. By construction every point of $X$ is fixed by an element of $P$, so that no orbital map can be injective in restriction to $\Sigma$. If the action of $G$ on $X$ was faithful then by the assumption \ref{i-non-foldable} we would have a contradiction. So $X$ cannot be faithful, and since the kernel is precisely $\bigcap_{S_G(P, G)} H$, condition \ref{i-conf} holds. \qedhere
\end{proof}

\begin{lem} \label{lem-exp-subset-growth}
	Let $G$ be a group, and $X$ a $G$-set. If $\mathcal{L} $ is non-folded in $X$, then $\vol_{G,X} (n) \succcurlyeq\relvol_{(G, \mathcal{L})}(n)$.
\end{lem}

\begin{proof}
	Let $S$ be a symmetric generating subset of $G$. For $n \geq 1$, let $g_n$ be an element of $G$ such that $|B_{g_n}(n) \cap \mathcal{L}| =  \relvol_{(G, \mathcal{L}) }(n)$.  We apply the defining condition of $\mathcal{L} $ being non-folded in $X$ to the finite subset $\Sigma =B_{g_n}(n) \cap \mathcal{L}$. We obtain $x_n \in X$ such that $\Sigma \to X$, $g \mapsto g x_n$, is injective. If we write $y_n = g_n x_n$, this implies that $|S^n y_n | \geq |\Sigma|$. But by definition we have $\vol_{G,X}(n) \geq |S^n y_n|$, so $\vol_{G,X}(n) \geq |\Sigma| = \relvol_{(G, \mathcal{L}) }(n)$.
\end{proof}

\begin{example}
	An infinite cyclic subgroup $C$  is always a non-foldable subset of $G$. In particular for a finitely generated group $G$, the existence of a distorted cyclic subgroup $C$ provides a non-trivial Schreier growth gap $\relvol_{(G,C)}(n)$.
\end{example}

\subsection{Non-foldable $k$-tuples}

In Section \ref{sec-metab} we will exhibit non-foldable subsets of a particular form in metabelian groups. To this end, we introduce the following convenient terminology. 

\begin{defin} 
	Let $G$ be a group, $k \geq 1$ and $(g_1,\ldots,g_k) \in G^k$.
		\begin{enumerate}[label=\roman*)]
		\item Let $X$ be a $G$-set. We say that $(g_1,\ldots,g_k)$ is \textbf{non-folded in $X$} if for all $n \geq 1$, there exists $x_n \in X$ such that \[ \left[-n,n \right]^k \to X, \, \, (n_1, \ldots, n_k ) \mapsto g_k^{n_k}  \ldots g_1^{n_1} (x_n),\] is injective. 
		\item We say  that $(g_1,\ldots,g_k)$ is \textbf{non-foldable} if $(g_1,\ldots,g_k)$ is non-folded in $X$ for every faithful $G$-set $X$. 
	\end{enumerate}
\end{defin}

The following is a reformulation of the definition:

\begin{lem} \label{lem-mapZk-inj}
Let $X$ be a $G$-set. A $k$-tuple $(g_1,\ldots,g_k)$ is non-folded in $X$ if and only if the following two conditions hold: \begin{itemize}
		\item the map  $\varphi : \mathbb{Z}^k \to G, \, \, (n_1, \ldots, n_k ) \mapsto g_k^{n_k}  \ldots g_1^{n_1}$, is injective;
		\item the image of $\varphi$ is non-folded in $X$.
	\end{itemize}
	In particular if $(g_1,\ldots,g_k)$ is non-folded in $X$ then $g_1,\ldots,g_k$ all have infinite order. 
\end{lem}

In particular in this setting Lemma \ref{lem-exp-subset-growth} reads as follows: 

\begin{lem} \label{lem-exp-tuple-growth}
	Let $G$ be a group, and $X$ a $G$-set. If a $k$-tuple $(g_1,\ldots,g_k)$ is non-folded in  $X$ then $\vol_{G,X}(n)   \succcurlyeq  n^k$.
\end{lem}

\begin{proof}
	Consider the map $\varphi : \mathbb{Z}^k \to G, \, \, (n_1, \ldots, n_k ) \mapsto g_k^{n_k}  \ldots g_1^{n_1}$, and denote by $\mathcal{L}$ its image. Since $\varphi$ is injective, we have $\relvol_{(G, \mathcal{L}) }(n)  \succcurlyeq  n^k$. And since $\mathcal{L}$ is non-folded in $X$ the statement follows from Lemma \ref{lem-exp-subset-growth}.
\end{proof}

\subsection{Asymptotic dimension of graphs of actions}

Another geometric invariant associated to graphs of actions of a finitely generated group is the asymptotic dimension, introduced by Gromov \cite{Gro-asdim}. Let us recall its definition and some basic properties, in analogy with the previous discussion on growth from \S  \ref{subsec-growth} and \S  \ref{subsec-non-foldable}. For more information we refer to the survey \cite{BD-asdim}.

\begin{defin}
	Let $(X, d)$ be a metric space, and $n \geq 0$. We say that space $(X, d)$  has asymptotic dimension at most $n$ if for every $R>0$ there exists  a cover $\mathcal{U}$ of $X$ consisting of subsets of uniformly bounded diameter such that every $R$-ball in $X$ intersects at most $n+1$ sets in $\mathcal{U}$. 
	The asymptotic dimension $\asdim(X, d)$ of $X$ is the smallest $n\in \N\cup \{ \infty \}$ such that $(X, d)$ has asymptotic dimension at most $n$. 
\end{defin}

The following lemma follows easily from the definition of asymptotic dimension, see the argument in \cite[\S 6]{BTS-asdim} (or the proof of \cite[Proposition 2.5]{MB-graph-germs} for more details).

\begin{lem} \label{l-asdim-monotone}
	Let $(\Gamma, d_\Gamma)$ and $(\Delta, d_\Delta)$ be uniformly locally finite metric spaces. Assume that there exist $C, D>0$ such that for every $R>0$ and $x\in \Gamma$ there exists a $C$-Lipschitz map $f\colon B(x, R)\to \Delta$ such that $|f^{-1}(y)|\le D$ for every $y\in \Delta$. Then $\asdim(\Gamma)\le \asdim(\Delta)$. 
\end{lem}

\begin{defin}
	Let $G$ be a finitely generated group, and $X$ a $G$-set. The asymptotic dimension of the $G$-set $X$ is $\asdim(G, X)=\asdim(\Gamma(G, X))$.
\end{defin}

Note that $\asdim(G, X)$ does not depend on the generating subset $S$ used to define the graph $\Gamma(G, X)$, since asymptotic dimension is a quasi-isometry invariant \cite[Prop. 22]{BD-asdim}. 

The following lemma is a straightforward consequence of Lemma \ref{l-asdim-monotone}.

\begin{lem}
	Let $G$ be a finitely generated group and $X$ be a $G$-set.  Then for every  finitely generated subgroup $H$ of $G$ we have $\asdim(H, X)\le \asdim(G, X)$. 
\end{lem}

The following lemma is proven analogously to Proposition \ref{prop-finite-index}, using Lemma \ref{l-asdim-monotone} and the invariance of asymptotic dimension up to quasi-isometry. 

\begin{lem}
	Let $G$ be a finitely generated group and $H$ be a subgroup of finite index of $G$. Then for every faithful $G$-set $X$ we have $\asdim(G, X)=\asdim(H, X)$. Conversely for every $H$-set $X$, the induced $G$-set $Y:=\operatorname{Ind}^G_H(X)$ satisfies $\asdim(G, Y)=\asdim(H, X)$. 
\end{lem}

Lemma \ref{l-asdim-monotone} has the following immediate consequence concerning  non-foldable subsets.

\begin{prop}\label{p-asdim-non-foldable}
	Let $G$ be a finitely generated group, and $\mathcal{L}$ a non-foldable subset of $G$. Then for every faithful $G$-set $X$ we have $\asdim(G, X)\ge \asdim(\mathcal{L})$, where $\mathcal{L}$ is seen as a metric space endowed with the restriction to $\mathcal{L}$ of a word metric on $G$. 
\end{prop}

\section{Motivating and limiting examples}

In this section we give some examples of group actions which add  context  to  the main results of the article.

\subsection{Some actions of small growth} \label{s-small-growth}

The point of this paragraph is to recall some examples of finitely generated solvable groups of exponential growth that admit natural actions with small growth, such as the wreath products $C_p \wr \Z^d$ or $\Z\wr \Z^d$, and Baumslag metabelian groups.

\subsubsection{Wreath product actions} \label{s-wreath-actions} Consider a wreath product $G =  A \wr B$ of two groups $A,B$. Recall that $G$ is the semi-direct product $\oplus_B A \rtimes B$, where $\oplus_B A$ is the set of finitely supported functions $B \to A$, and $B$ acts on $\oplus_B A$ by $b \cdot \varphi : b' \mapsto  \varphi(b^{-1}b')$. The group $G$ naturally comes with an action on the set $B \times A$, called the \textbf{standard wreath product action}, defined by $(\varphi,b) \cdot (b_0, a_0) = (bb_0, \varphi(bb_0)a_0)$. In the sequel we denote by $X_\mathrm{st}$ the $G$-set $X_\mathrm{st} = B \times A$.

If $S_1,S_2$ are finite generating subsets of $A,B$ respectively, then $S = S_1 \cup S_2$ is a generating subset of $G$. Here and throughout the paper, we implicitly view $S_1$ and $S_2$ a subsets of $G$ by identifying $A$ and $B$ with their natural copies inside $G$, where the copy of $A$ consists of elements in $\oplus_B A$ that are trivial everywhere except at the identity position. The edges in the graph $\Gamma(G,X_\mathrm{st})$ are defined by saying that there is an edge between $(1_B,a)$ and $(1_B,s_1a)$ for every $a \in A$ and $s_1 \in S_1$, and between $(b,a)$ and $(s_2b,a)$ for every $b \in B,a \in A$ and $s_2 \in S_2$. Equivalently, the graph $\Gamma(G,X_\mathrm{st})$ is obtained by taking a copy of the Cayley graph of $A$ and attaching to each vertex a copy of the Cayley graph of $B$. For instance when $A$ is finite, $\Gamma(G,X_\mathrm{st})$ is just the union of $k=|A|$ copies of the Cayley graph of $B$ that are joined at the identity position, and $\vol_{G, X_\mathrm{st}}(n) \simeq \vol_{B}(n)$. So for $B = \Z$, we obtain a union of $k$ bi-infinite lines. When $A = B = \Z$ with standard generating subsets, the graph $\Gamma(\Z \wr \Z,X_\mathrm{st})$ is a comb, see Figure \ref{fig-comb}. In that case we have $\vol_{G, X_\mathrm{st}}(n) \simeq  n^2$. In general the growth of the $G$-action on $X_\mathrm{st}$ is given by the following:

\begin{figure}[ht]
	\includegraphics[scale=.5]{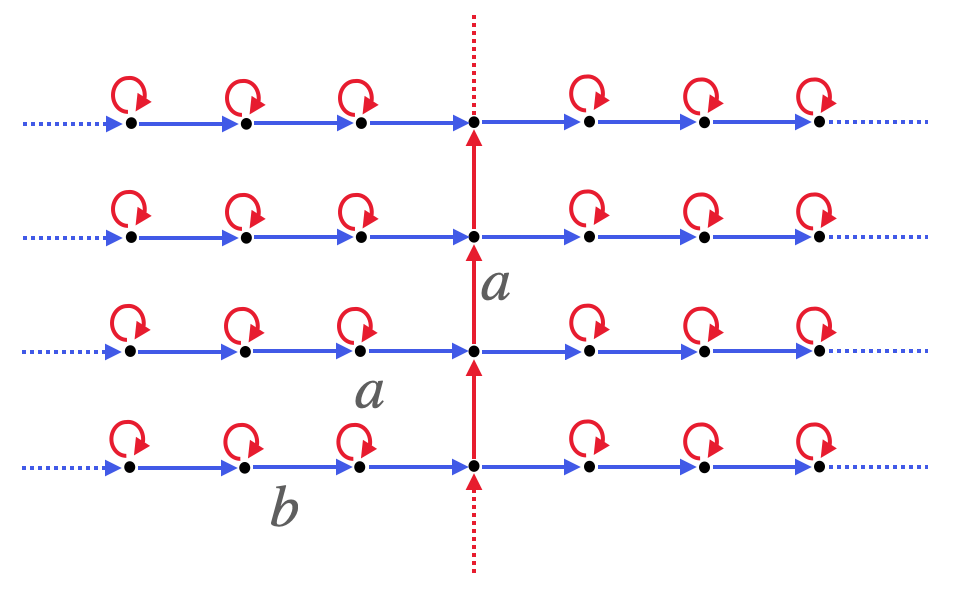}
	\caption{\small The graph of the standard action of $\Z\wr \Z$ on $\Z\times \Z$. Red arrows correspond to the generator $a$ of the lamp group, while blue arrows correspond to the generator $b$ of the base group. }\label{fig-comb}
\end{figure} 

\begin{lem}
	We have  $\vol_{G, X_\mathrm{st}}(n) \simeq \vol_{B}(n) \times  \vol_{A}(n)$. 
\end{lem}

\begin{proof}
	By the above description of $\Gamma(G,X_\mathrm{st})$, the identity map $\Gamma(G,X_\mathrm{st}) \to B \times A$ is a Lipschitz map, where the target space $B \times A$ is endowed with the $\ell^1$-metric. This implies $\vol_{G, X_\mathrm{st}}(n) \preccurlyeq \vol_{B}(n) \times  \vol_{A}(n)$. Conversely, every element $(b,a)$ such that $|a|_{S_1}, |b|_{S_2} \leq n$ is at distance at most $2n$ from $(1_B,1_A)$ in $\Gamma(G,X_\mathrm{st})$, so the ball of radius $2n$ around $(1_B,1_A)$ in $\Gamma(G,X_\mathrm{st})$ contains at least $\vol_{B}(n) \times  \vol_{A}(n)$ elements. The statement follows.
\end{proof}

\subsubsection{Baumslag's finitely presented examples} \label{subsec-Baumslag}

Consider the group given by the finite presentation
\[\Lambda_p = \left\langle u,t,s \,  | \, u^p = 1, \, [s,t] = 1, \, [u^{t},u] = 1, \,  u^{s} = u^{t} u  \right\rangle, \] 
where $p$ is  a prime number. These groups (actually their torsion-free counterpart) were introduced in \cite{Baumslag-group}. Baumslag showed that $\Lambda_p$ is isomorphic to the semi-direct product $\mathbb{F}_p[T, T^{-1},(T+1)^{-1}] \rtimes \Z^2$, where  the generators $t,s$ of $\Z^2$ act on the ring $R:=\mathbb{F}_p[T, T^{-1},(T+1)^{-1}]$ by multiplication by $T$ and $T+1$.

We claim that $\Lambda_p$ admits a faithful transitive action on a set $X$ such that $\vol_{\Lambda_p, X}(n)\simeq n^2$. To see this, observe that every element of $R$ can be uniquely written as \[ \sum_{n \in \Z} a_n T^n + \sum_{\ell \geq 0, q > 0} b_{\ell,q} T^{-\ell} (1+T)^{-q},\] where $a_n, b_{\ell,q}  \in \mathbb{F}_p$, and only finitely many of them are non-zero. Consider the subgroup $H$ of $R$ defined by the condition $a_0 = 0$. An easy computation shows that the intersection of all conjugates of $H$ in $\Lambda_p$ is trivial. Equivalently, the action of $\Lambda_p$ on $X = \Lambda_p/H$ is faithful. The subgroup $H$ has finite index equal to $p$ in $R$. We have a  $G$-equivariant surjective map $ G/H\to G/R$, whose fibers have cardinality $p$. Thus Proposition \ref{prop-monoton} implies that $\vol_{\Lambda_p, X}(n) \simeq  \vol_{\Lambda_p/R}(n)= n^2$.

The Baumslag group admits the following generalisation. For $d\ge 1$ let \[ \Lambda_{p, d}= \mathbb{F}_p[T_1,\ldots, T_d, T_1^{-1}, \ldots, T_d^{-1}, (1+T_1)^{1},\ldots (1+T_d)^{-1}] \rtimes \Z^{2d}.\] The action is defined by saying that if $t_1,\ldots, t_d, s_1,\ldots, s_d$ are generators of $\Z^{2d}$, then $t_i$ and $s_i$ act respectively by multiplication by $T_i$ and $T_i+1$. The (torsion-free analogues) of these groups were introduced by Erschler \cite{Ersch-Liouville} as generalisation of Baumslag's groups, to give examples of finitely presented  groups which are amenable but not Liouville  (for $d\ge 3$). They admit an analogous finite presentation \cite[Lemma 5.1]{Ersch-Liouville}. The same argument as above shows that each group $\Lambda_{p, d}$ admits a faithful action with $\vol_{\Lambda_{p, d}, X}(n)\simeq n^{2d}$.

\subsubsection{A general criterion} The following result, which is based on a result of Olshanskii \cite{Olshanskii-KK}, shows that  the existence of actions of polynomial growth of the wreath products $C_p \wr \Z^d$ and the above groups $\Lambda_{p, d}$ are not isolated phenomena.

\begin{prop} \label{prop-growth-ol}
	Let $G$ be a finitely generated metabelian group such that $G'$ is a torsion group and $G/G'$ has torsion-free rank $d$. Then $G$ admits a faithful transitive action on a set $X$ with $\vol_{G, X}(n) \simeq  n^d$.
\end{prop}

\begin{proof}
	By \cite[\S 3]{Olshanskii-KK}, if $G$ is a metabelian group such that  $G'$ is a torsion group, there always exists a finite index subgroup $H$ of $G'$ such that  the intersection of all $G$-conjugates of $H$ is trivial, so that the action of $G$ on $G/H$ is faithful. Since $H\le G'$   we have a  $G$-equivariant surjective map $ G/H\to G/G'$, whose fibers have cardinality equal to  the index $[G': H]$. Proposition \ref{prop-monoton} then implies $\vol_{G, G/H}(n)\simeq \vol_{G, G/G'}(n)\simeq n^d$. \qedhere
\end{proof}

\subsection{Actions of lamplighter groups with prescribed growth} \label{subsec-behavior}

It is natural to wonder if, as in the case of the growth of the group, the growth of an action of a solvable group is always polynomial of integer degree or exponential. It turns out that this is far from being the case. In fact, even  the growth of an action of the lamplighter group $G=C_2 \wr \Z$ can  be arranged to be equivalent to an arbitrary function  satisfying mild conditions:

\begin{prop}\label{p-lamplighter-given-growth}
	Let $G=C_2 \wr \Z$. Let $f\colon \N \to \R_+$ be a non-decreasing function such  that $f(n+1)/f(n)$ is non-increasing. Then there exists a faithful and transitive $G$-set $X$ such that $\vol_{G, X}(n)\simeq n f(n)$. In particular:
	\begin{itemize}
		\item for every real number $\alpha \ge 1$, there exists a faithful $G$-set $X$ with $\vol_{G, X}(n)\simeq n^\alpha$;
		\item for every $0<\beta \le 1$, there exists a faithful $G$-set $X$ with $\vol_{G, X}(n) \simeq \exp(n^\beta)$. 
	\end{itemize}
\end{prop}

%%\begin{prop}
%Let $G=C_2 \wr \Z$. Let $f\colon \N \to \R_+$ be an non-decreasing function such  that $f(n)/n$ is non-decreasing and $f(n+1)/f(n)$ is non-non-decreasing. Then there exists a faithful and transitive $G$-set $X$ such that $\vol_{G, X}(n)\simeq f(n)$. In particular:
%\begin{itemize}
%	\item for every $\alpha \ge 1$, there exists a faithful $G$-set $X$ with $\vol_{G, X}(n)\simeq n^\alpha$;
%	\item for every $0<\beta \le 1$, there exists a faithful $G$-set $X$ with $\vol_{G, X}(n) \simeq \exp(n^\beta)$. 
%\end{itemize}
%\end{prop}

\begin{proof} 
	Let us write $f(n)= 4^{g(n)}$, where $g\colon \N\to \R$ is non-decreasing and $g(n+1)-g(n)$ is non-increasing. If $g(n+1)-g(n)$ is bounded away from $0$, then $f(n)\simeq \exp(n)$, so that we can choose $X=G$ acting on itself. Thus in the following we suppose that $g(n+1)-g(n)$ tends to $0$. Upon replacing $g(n)$ by $g(n_0+n)$ for some $n_0>0$ we can suppose that $g(n+1)-g(n)\le 1$ for every $n$ (this does not affect $f(n)$ up to $\simeq$ since $f(n)\le 4^{g(n+n_0)} \le 4^{n_0(g(1)-g(0))} f(n)$). Similarly if $g(n)$ is bounded then we can choose $X=C_2\times \Z$ (the standard wreath product action; see \S \ref{s-wreath-actions}). So we assume that $g(n)$ tends to $+\infty$. 
	
	For $k\in \N$ let  $x_k=\min\{n\colon g(n)\ge k\}$ and let $\Omega=\{x_k, k\in \N\} \cup \{-x_k, k\in \N\}$.   
	Let $H=\{r\in \oplus_\Z C_2 \colon r|_\Omega=0\}$  and set $X=G/H$. Note that $X$ is a faithful $G$-set. Indeed it is easy to see that every non-trivial element $r\in \oplus_\Z C_2$ has a conjugate  outside of $H$.

	In what follows we write elements of $G$ as pairs $(r,m)$ with  $r\in \oplus_\Z C_2$ and $m\in \Z$, and let $S=\{s, t^{\pm 1}\}$ be the standard generating set given by $s=(\delta_0, 0)$ and $t=(0, 1)$.   The set $X$ can be naturally identified with $\Z \times (\oplus_{\Omega} C_2)$, by mapping each coset $(r, m)H$ to $(m, \sigma_mr|_{\Omega})$, where $\sigma_mr$ denotes the shifted configuration $\sigma_mr(x):=r(x-m)$. Under this identification, the generator $t$ acts by moving $(m, r)$ to $(m+1, r)$, and the generator $s$ acts on $(m, r)$ by flipping the value of $r(-m)$ if $-m\in \Omega$, while $s$ fixes $(m, r)$ if $-m\notin \Omega$. From this description it follows that any product of at most $n$ generators must move $(m, r)$ to some point $(m', r')$, where $m'\in   [m-n, m+n]$ and  $r$ and $r'$ coincide outside $[m-n, m+n]\cap \Omega$. On the other hand,  by applying a product of at most $2n$ generators to $(m, r)$ it is possible to reach any $(m', r')$ satisfying the same constraints. This implies that the ball  $B_n$ of radius $n$ around  $(m, r)$ in the Schreier graph of the action satisfies
	\begin{equation}\label{e-lamplighter-actions} |B_n|\le (2n+1)2^{|[m-n, m+n]\cap \Omega|}\le |B_{2n}|.\end{equation}
	Let us analyse the cardinality $|[m-n, m+n]\cap \Omega|$. First we note that the assumption that $g(n+1)-g(n)$ is non-increasing implies that $x_{k+1}-x_k$ is non-decreasing, so that when $n$ is fixed,  $|[m-n, m+n]\cap \Omega|$ is maximized for $m=0$.  On the other hand the assumption that $g(n+1)-g(n)\le 1$ implies that the map $k\mapsto x_k$ is injective. As a consequence,  we have $|[0, n]\cap \Omega|=\max\{k \colon x_k\le n\}$, which  is equal to $g(n)$ up to an additive error bounded by 1. We deduce that $\vol_{G, X}(n)= (2n+1)2^{|[-n, n]\cap \Omega|}\simeq n2^{2g(n)}= n f(n)$. \qedhere
\end{proof}

\begin{remark}
	For the   $d$-dimensional lamplighter $G=C_2\wr \Z^d$, the same argument (with minor modifications) can be applied to show that for every non-decreasing function $f\colon \N\to \R_{+}$ such that $f(n+1)/f(n)$ is non-decreasing, there exists a faithful $G$-set $X$ such that $\vol_{G, X}(n)\simeq n^d f(n)$. 
	Note that the smallest growth obtained through this construction is $n^d$, which is equal to the growth of the standard wreath product action of $C_2 \wr \Z^d$.
\end{remark}

\section{Polycyclic groups} \label{s-noetherian}

The goal of this section is to explain how the understanding of the growth of actions of nilpotent and polycyclic groups reduces to the classical results on growth of these groups due to Wolf \cite{Wolf}, Bass \cite{Bass-nilp} and Guivarc'h \cite{Guivarch-nilp}. The approach is elementary. It is based on compactness arguments in the space $\sub(G)$, and holds more generally for Noetherian groups. Recall that a group is \textbf{Noetherian} if all its subgroups are finitely generated. A solvable group is Noetherian if and only if it is polycyclic.

First we isolate the following remark for future reference:

\begin{lem} \label{l-product-actions}
	Let $G$ be a finitely generated group which is a direct product $G=Q_1\times \cdots \times Q_k$. For each $i$, let $X_i$ be a faithful $Q_i$-set. Then $X=\sqcup_i X_i$ is a faithful $G$-set, where $G$ acts on each $X_i$ through $Q_i$, and we have 
	\[\vol_{G, X}(n)=\max_{i=1}^k \vol_{Q_i, X_i}(n).\]
	In particular if   all $Q_i$s have polynomial growth with $\vol_{Q_i}(n)\simeq n^{d_i}$, then there is a faithful $G$-set $X$ such that  $\vol_{G, X}(n)\simeq n^d$, where $d=\max_i d_i$. 
\end{lem}

In particular we have the following for virtually abelian groups. 

\begin{prop}\label{prop-virtually-abelian}
	If $G$ is an infinite finitely generated virtually abelian group,  then there exists a faithful $G$-set $X$ such that $\vol_{G, X}(n)\simeq n$. 
\end{prop}

\begin{proof}
	The group $G$ contains a finite index subgroup isomorphic to $\Z^d$ for some $d\ge 1$. Thus the conclusion follows from Proposition \ref{prop-finite-index}, by considering the action of $\Z^d$ on $\sqcup_{i=1}^d \Z$ as in Lemma \ref{l-product-actions} 
\end{proof}

%Recall that it is easy to produce examples of groups $G$ with a faithful $G$-set $X$ such that $\vol_{G, X}(n)$  grows strictly slower  than $\vol_G(n)$:  for instance a direct product $G=Q_1\times Q_2$ of two infinite nilpotent  groups acting on $X =Q_1\sqcup  Q_2$  (see Lemma \ref{l-product-actions}). The following proposition is a partial converse to this observation for Noetherian groups.

The following proposition is a partial converse to Lemma \ref{l-product-actions} for Noetherian groups.

\begin{prop} \label{prop-noetherian}
Let $G$ be a finitely generated Noetherian group, and $X$ a faithful $G$-set. Then there exist a finite index subgroup $G^0\le G$ and normal subgroups $K_1,\dots,K_d\unlhd G^0$  such that, if we denote $Q_i=G^0/K_i$, then the following hold:
\begin{enumerate}[label=\roman*)]
\item \label{item-neoth-subdirect} The intersection $\bigcap_i K_i$ is trivial. Equivalently, $G^0$ embeds in $Q_1\times \cdots \times Q_d$.

\item  \label{item-neoth-fixed} If we denote $X_i:=\fix(K_i)$, then we have $X=\cup_i X_i$, and each $X_i$ is a $G^0$-invariant subset such that the action of $G^0$ on $X_i$ factors through an action of $Q_i$. 

\item \label{item-neoth-locally-embeds} $Q_i$ is non-folded in $X_i$ for all $i$. 
	
	\item \label{item-neoth-growth} We have $\vol_{G, X}(n)\simeq \max_{i=1}^n \vol_{Q_i}(n)$.
	\end{enumerate}

\end{prop}

\begin{proof}

Let $\X$ be a closed $G$-invariant subset of $\sub(G)$. Fix $H\in \X$. Let $\H$ be the closure of the $G$-orbit of $H$ in $\X$, and $\mathcal{L} \subseteq \H$ be a non-empty minimal closed $G$-invariant subset. Since $G$ is Noetherian, $\sub(G)$ is a countable compact space, and hence every closed subset of $\sub(G)$ admits isolated points. So $\mathcal{L}$ admits isolated points, and hence by minimality it follows that $\mathcal{L}$ is finite. So $\mathcal{L}$ is just the finite $G$-orbit of a subgroup $L \in \X$. Now since $L$ is finitely generated, the set $S_G(\geq,L)$ of subgroups $H$ of $G$ such that $H$ contains $L$ is an open neighbourhood of $L$ in $\sub(G)$. Since by definition $L$ belongs to the closure of the $G$-orbit of $H$, it follows that there exists a conjugate of $H$ that contains $L$, and as a consequence $H$ itself contains some $L' \in \mathcal{L}$. Now using again that $S_G(\geq,L)$ is open for every subgroup $L$ and compactness of $\X$, it follows that we can find a finite $G$-invariant subset  $\{L_1, \cdots, L_n\}\in \X$,  such that for every $H\in \X$ there is $i$ such that $L_i\le H$.

Now let $\mathcal{X}$ be the closure of $\{G_x\colon x\in X\}$ in $\sub(G)$.  Let $L_1, \cdots, L_n\in \X$ as in the conclusion of the previous paragraph. Then for every $x$ in $X$ there exists $i$ such that $L_i \leq G_x$. In particular $\bigcap_i L_i$ acts trivially on $X$, and hence is trivial by the assumption that the action is faithful. Since the normalizer $N_G(L_i)$ of $L_i$ has finite index in $G$ for all $i$, $G^0:=\bigcap_i N_G(L_i)$ is a finite index subgroup of $G$. Set $K_i=G^0\cap L_i$, which is a normal subgroup of $G^0$. The intersection $\bigcap K_i$ is clearly trivial since this is already true for the groups $L_i$, and every $x\in X$ is fixed by some $K_i$. As in the statement, set $X_i:=\fix(K_i)$, and note that all conclusions in part \ref{item-neoth-subdirect} and \ref{item-neoth-fixed} hold true by construction. Let us show \ref{item-neoth-locally-embeds}. Since $L_i$ belongs to $\X$, we can find a sequence of points $(x_n)\subset X$ such that $G_{x_n}$ converges to $L_i$. As a consequence the subgroups  $G^0_{x_n}=G^0\cap G_{x_n}$ converge to $G^0\cap L_i=K_i$. It follows in particular that $G^0_{x_n}$ contains $K_i$ for $n$ large enough, which means that $(x_n)\subset X_i$ for $n$ large enough. Thus the stabiliser of $x_n$ in $Q_i$ converges to the trivial subgroup in $\sub(Q_i)$, so that $Q_i$ is non-folded in $X_i$. 
In order to justify  \ref{item-neoth-growth}, first note  that $\vol_{G, X}(n)\simeq \vol_{G^0, X}(n)$ by Proposition \ref{prop-finite-index}. Second, note that since the action of $G_0$ on the orbit of every $x\in X$ factors through an action of some $Q_i$, we must have $\vol_{Q_i}(n)\preccurlyeq\vol_{G^0, X}(n)$ for all $i$, and hence $\max_i \vol_{Q_i}(n)\preccurlyeq\vol_{G^0, X}(n)$. Finally the converse inequality follows from part \ref{item-neoth-locally-embeds}.\qedhere
\end{proof}

Recall that a group $G$ is \textbf{subdirectly decomposable} if it admits non-trivial normal subgroups $M,N$ such that $M \cap N = 1$. We will also say that $G$ is \textbf{virtually subdirectly decomposable} if some finite index subgroup of $G$ is subdirectly decomposable.

\begin{cor} \label{cor-noeth-group-loc-embed}
Suppose $G$ is Noetherian and not virtually subdirectly decomposable. Then the whole group $G$ is non-foldable. In particular  $\vol_{G, X}(n)\simeq \vol_G(n)$ for every faithful $G$-set $X$. 
\end{cor}

\begin{proof}
We apply Proposition \ref{prop-noetherian}. It cannot be that all the $K_i$ are non-trivial, since otherwise the intersection would be non-trivial by assumption. Hence there is $i$ such that $K_i = 1$, and hence $Q_i=G$ is non-folded in  $X_i=X$. 
\end{proof}

\subsection{Nilpotent groups}

If $G$ is a nilpotent group $G$, then classical results of Bass and Guivarc'h \cite{Bass-nilp, Guivarch-nilp} assert that the growth of $G$ is $\vol_G(n)\simeq n^{\alpha_G}$, where $\alpha_G$ is a positive integer given by the formula 
\begin{equation}\label{e-Bass-Guivarch}\alpha_G:=\sum_{i\ge 0} i\dim_\Q\left((\gamma_i(G)/\gamma_{i+1}(G))\otimes \Q\right).\end{equation}

The following is a direct consequence of Proposition \ref{prop-noetherian} combined with the Bass-Guivarc'h formula. It implies in particular that the growth of the action of every nilpotent group is polynomial.

\begin{cor} \label{cor-polynomial-growth}
	Let $G$ be a finitely generated nilpotent group, and let $X$ be a faithful $G$-set. Then there exist a finite index subgroup $G^0\le G$ and normal subgroups $K_1,\dots,K_d\unlhd G^0$  with $\bigcap K_i=\{1\}$ such that if we write $Q_i=G^0/K_i$ and $\alpha_*=\max_i \alpha_{Q_i}$,  then $\vol_{G, X}(n) \simeq n^{\alpha_*}$. 
	\end{cor}

In the sequel we denote by $H_3(\Z)$ the Heisenberg group of $3 \times 3$-matrices over the integers. Recall that it has infinite cyclic center, and the associated quotient is free abelian of rank $2$. It has the presentation $H_3(\Z) = \langle a, b, c\colon [a,b]=c, [a,c]=[b, c]=1\rangle$. For later use we record the following well-known fact.

\begin{lem} \label{l-subgroup-H3}
Let $G$ be finitely generated  nilpotent group which is not virtually abelian. Then $G$ contains a subgroup $H$ isomorphic to $H_3(\Z)$. 
\end{lem}

\begin{proof}
Upon passing to a finite index subgroup we can suppose that $G$ is torsion-free. Let $\gamma_{k} (G)$ be the last non-trivial term in the lower central series. Since $G$ is not abelian, we have $k\ge 2$ and we can find  $g\in \gamma_{k-1} (G)$ and $h\in G$ such that $k:=[g, h]$ is not trivial. Note that $k$ is central in $G$, so that we have $[g, k]=[h,k]=1$. Hence   $H=\langle g, h, k\rangle$ is isomorphic to a quotient of $H_3(\Z)$. Since every proper quotient of $H_3(\Z)$ either has torsion or is abelian, it follows that $H$ is  isomorphic to $H_3(\Z)$.
\end{proof}

\begin{lem} \label{l-H3-indecomp}
The group $H_3(\Z)$ is not virtually subdirectly decomposable. 
\end{lem}

\begin{proof}
If $L$ is a finite index subgroup of $H = H_3(\Z)$, the center $Z(L)$ is infinite cyclic. Since any normal subgroup of $L$ intersects $Z(L)$ non-trivially  \cite[1.2.8 i)]{Lennox-Rob}, any finite family of non-trivial normal subgroups of $L$ contains a common non-trivial element of $Z(L)$.  \qedhere
\end{proof}

\begin{prop}\label{prop-n4-growth}
Let $G$ be finitely generated  virtually nilpotent group which is not virtually abelian. Then any subgroup $H$ isomorphic to $H_3(\Z)$ is a non-foldable subset of $G$. In particular $G$ has a Schreier growth gap $n^4$.  
\end{prop}

\begin{proof}
The first statement is a direct consequence of Corollary \ref{cor-noeth-group-loc-embed} and the previous lemmas, and the last statement follows since $H_3(\Z)$ has growth $\simeq n^4$.
\end{proof}

\subsection{Polycyclic groups}

Recall that Wolf theorem asserts that a  polycyclic group that is not virtually nilpotent has exponential growth \cite{Wolf}. Combined with Proposition \ref{prop-noetherian} this has the following consequence. 

\begin{cor} \label{cor-poly-growth}
	Suppose $G$ is a polycyclic group that is not virtually nilpotent. Then $G$ has a Schreier growth gap $\exp(n)$.
\end{cor}

\begin{proof}
The group is polycyclic, and hence Noetherian. If $X$ is a faithful $G$-set, we apply Proposition \ref{prop-noetherian}. Let $G^0$ and $K_1,\cdots K_d$ as in the conclusion. Being a finite index subgroup of $G$, the group $G^0$ has exponential growth. It follows that at least one of the groups $Q_i=G^0/K_i$ must have exponential growth, since $G^0$ embeds in their product.  Therefore $\max_i{\vol}_{Q_i}(n)\simeq \exp(n)$, and the conclusion follows from Proposition \ref{prop-noetherian}. \qedhere
\end{proof}

	The route taken above to prove Corollary \ref{cor-poly-growth} does not produce explicit non-foldable subsets. However this will be achieved in Section \ref{s-finite-rank}  in the more general setting of solvable groups of finite rank. Below we provide a simpler construction of non-foldable subsets in polycyclic groups, which already yield information on the asymptotic dimension of graphs of actions of polycyclic groups (Corollary  \ref{cor-poly-asdimX}). To this end, we introduce the following terminology (which is consistent with the terminology that we will use in \S  \ref{subsec-strong-irr}).

\begin{defin}
We say that a semi-direct product $G = \Z^k \rtimes \Z$ is \textbf{strongly irreducible} if every non-trivial subgroup of $\Z^k$ that is invariant under a finite index subgroup of $\Z$ has rank $k$.
\end{defin}

\begin{lem} \label{lem-irr-SDP-notdecomp}
Let $G = \Z^k \rtimes \Z$ be a strongly irreducible semi-direct product  with $k \geq 2$. Then $G$ is not virtually subdirectly decomposable.
\end{lem}

\begin{proof}
Observe that every finite index subgroup of $G$ is also a strongly irreducible semi-direct product. Hence it is enough to check that the intersection of two non-trivial normal subgroups of $G$  remains non-trivial. Observe also that the condition $k \geq 2$ ensures that $G$ is not virtually abelian.  If $N$ is a non-trivial normal subgroup of $G$, then $N \cap \Z^k$  is necessarily non-trivial, because otherwise $N$ would commute with $\Z^k$ and  $G$ would have a finite index abelian subgroup. So by the strongly irreducible assumption it follows that $N \cap \Z^k$ has finite index in $\Z^k$. In particular it follows immediately that if $M,N$ are two non-trivial normal subgroups of $G$,  then $M \cap N$ is non-trivial.
\end{proof}

\begin{lem} \label{lem-poly-metab-norm-subgroup}
Every polycyclic group $G$ that is not virtually abelian contains a (normal) subgroup that is metabelian and not virtually abelian.
\end{lem}

\begin{proof}
Since $G$ is infinite, $G$ admits an  infinite free abelian normal subgroup \cite[1.3.9]{Lennox-Rob}. Let $A$ be an infinite free abelian normal subgroup of maximal rank $n$. Again $G/A$ is infinite, so we may find an infinite free abelian normal subgroup $B$ in $G/A$ or rank $m$. Consider the preimage $N$ of $B$ in $G$, which is a metabelian normal subgroup of $G$. If $N$ was virtually abelian, it would admit a characteristic free abelian subgroup $C$ of rank $m+n$. In particular $C$ would be normal in $G$, contradicting the maximality of $n$. So this subgroup $N$ satisfies the conclusion.
\end{proof}

The following is an easy consequence of basic facts from linear algebra.

\begin{lem} \label{lem-poly-metab-nilpORirr}
	Let $G$ be a metabelian polycyclic group. Then either $G$ is virtually nilpotent; or contains a subgroup $H = \Z^k \rtimes \Z$ with $k \geq 2$ that is a strongly irreducible semi-direct product.
\end{lem}

\begin{proof}
Upon passing to a finite index subgroup, $G$ admits a normal subgroup $N$ isomorphic to $\Z^d$ such that $Q = G/N$ is isomorphic to $\Z^r$. Consider the associated representation $Q \to \GL(d, \Z)$. If every element of $Q$ has the property that all its eigenvalues are roots of unity then $Q$ has a finite index subgroup that is unipotent, and $G$ is virtually nilpotent, a contradiction. So we can find an element of $Q$ whose eigenvalues are not all roots of unity, and the conclusion follows by applying Lemma \ref{l-strongly-irreducible-power} below.
\end{proof}

\begin{lem}\label{lem-poly-special subgroup}
	Let $G$ be polycyclic group that is not virtually abelian. Then $G$ contains a subgroup isomorphic to the Heisenberg group $H_3(\Z)$; or a subgroup $H = \Z^k \rtimes \Z$ with $k \geq 2$ that is a strongly irreducible semi-direct product.
\end{lem}

\begin{proof}
By Lemma \ref{lem-poly-metab-norm-subgroup} it is enough to treat the case where $G$ is metabelian. In the virtually nilpotent case we use Lemma \ref{l-subgroup-H3}, and in the other case we invoke Lemma  \ref{lem-poly-metab-nilpORirr}.
\end{proof}

\begin{prop} \label{prop-expand-subgroup-poly}
	Let $G$ be polycyclic group that is not virtually abelian. Then any subgroup of $G$ as in Lemma \ref{lem-poly-special subgroup} is a non-foldable subset of $G$.
\end{prop}

\begin{proof}
	Let $H$ be a subgroup of $G$ that is isomorphic either to $H_3(\Z)$ or to a strongly irreducible semi-direct product $\Z^k \rtimes \Z$ with $k \geq 2$. In both cases $H$ is not virtually subdirectly decomposable (Lemma \ref{l-H3-indecomp} and Lemma \ref{lem-irr-SDP-notdecomp}), and hence $H$ is a non-foldable subset of $G$ by Corollary \ref{cor-noeth-group-loc-embed}.
\end{proof}

Recall that the \textbf{Hirsch length} $h(G)$ of a polycyclic group $G$ is the number of infinite cyclic factors appearing in a finite series of $G$ with cyclic factors. The following result was proven in \cite{Dranishnikov-Smith}.

\begin{thm} \label{thm-poly-asdim-hirsch}
If $G$ is a polycyclic group, then $\asdim(G) = h(G)$. 
\end{thm}

\begin{cor} \label{cor-poly-asdimX}
Let $G$ be polycyclic group that is not virtually abelian. Then $\asdim(G,X) \geq 3$  for every faithful $G$-set $X$.
\end{cor}

\begin{proof}
	Let $H$ be a subgroup of $G$ that is isomorphic either to $H_3(\Z)$ or to a strongly irreducible semi-direct product $\Z^k \rtimes \Z$ with $k \geq 2$. According to Theorem \ref{thm-poly-asdim-hirsch} we have $\asdim(H) = 3$ if $H$ is $H_3(\Z)$ and $\asdim(H) = k+1$ if $H$ is $\Z^k \rtimes \Z$. In both case $\asdim(H) \geq 3$. Let $X$ be a faithful $G$-set. The subgroup $H$ being non-foldable by Proposition \ref{prop-expand-subgroup-poly}, Proposition \ref{p-asdim-non-foldable} then implies $\asdim(G,X) \geq \asdim(H,d_G) = \asdim(H) \geq 3$. In the middle equality we have used that the restriction of the word metric $d_G$ to $H$ is coarsely equivalent to the word metric on $H$, and $\asdim$ is an invariant of coarse equivalence. 
\end{proof}

\section{Strongly irreducible extensions and groups of finite rank}  \label{s-finite-rank}
The goal of this section is to prove Theorem \ref{t-strongly-irreducible}, which gives a criterion to find  non-foldable subsets in a group $G$ that can be written as an extension $1\to N\to G\to Q\to 1$, where $N$ is a nilpotent group of finite rank and the action of $Q$ satisfies a certain irreducibility condition. As an application, we will prove Theorem \ref{thm-intro-prufer} from the introduction. 

\subsection{Preliminaries on nilpotent groups} Before stating the main results of the section we need to  recall some preliminaries on nilpotent groups, based on Malcev theory.

 A group $G$ is \textbf{divisible} if for every $g\in G$ and every integer $n>0$ there is $h\in G$ such that $h^n=g$.  Assume that $\mk$ is a nilpotent Lie algebra over $\Q$. Then $\mk$ can be turned into a group with group law given by the Baker-Campbell-Hausdorff formula. The resulting group is a divisible torsion-free nilpotent group. The following theorem of Malcev \cite{Malcev49} is a converse to this construction (see also \cite{Stewart} for a more algebraic treatment).

\begin{thm}[Malcev]\label{t-Malcev-divisible}
Let $\mk$ be a torsion-free divisible nilpotent group. Then $\mk$ can be endowed with a unique structure of Lie algebra over $\Q$ such that the group law on $\mk$ coincides with the group law determined by the Baker-Campbell-Hausdorff formula. Moreover, every homomorphism between divisible nilpotent groups is also a homomorphism of Lie algebras, and vice versa 
\end{thm} 

In the sequel, whenever $\mk$ is a torsion-free divisible nilpotent group, we will consider it also as a  Lie algebra  $\mk$ without mention. By the above theorem the Lie subalgebras of $\mk$ are precisely the divisible subgroups of $\mk$. The Lie bracket will be denoted $\llbracket\cdot, \cdot \rrbracket$, the notation $[\cdot, \cdot]$ being reserved for commutators in groups. Given a subset $S\subset \mk$, we will denote by $\Lie(S)$ the Lie subalgebra generated by $S$, so $\Lie(S)$ is the smallest divisible subgroup of $\mk$ containing $S$. We keep the notation $\langle S \rangle$ for the subgroup generated by $S$. The lower central series or $\mk$ as a Lie algebra is denoted $\gamma_i\mk$; this does not lead to confusion as it coincides with its lower central series of $\mk$ as a group. The abelianization of $\mk$ is denoted $\mk^{ab}$.

\begin{lem} \label{l-nilpotent-abelianisation}
Let $\mk$ be a nilpotent Lie algebra and $S\subset \mk$ be a set whose projection to $\mk^{ab}:=\mk/\llbracket\mk, \mk\rrbracket$ contains a linear basis of $\mk^{ab}$. Then $\Lie(S)=\mk$. 
\end{lem}
\begin{proof}
Denote $\pi_i\colon \gamma_{i} \mk\to \gamma_{i}\mk/\gamma_{i+1}\mk$ the quotient projection.  Set $\hk=\Lie(S)$. It is enough to show that $\pi_i(\hk\cap \gamma_i\mk)=\gamma_i\mk/\gamma_{i+1}\mk$ for every $i$. By assumption, this holds for $i=1$. Assume that it holds for all $j<i$. Fix $X\in \gamma_{i-1}\mk$ and $Y\in \mk$. By assumption we can find $\tilde{X}\in \hk\cap \mk_{i-1}, \tilde{Y}\in\hk$ such that $\pi_{i-1}(\tilde{X})=\pi_{i-1}(X)$ and $\pi_1(\tilde{Y})=\pi_1(Y)$. Then we have $\llbracket\tilde{X}, \tilde{Y}\rrbracket\in  \hk\cap \gamma_i\mk$ and $\pi_i(\llbracket\tilde{X}, \tilde{Y}\rrbracket)=\pi_i(\llbracket X, Y\rrbracket)$. Since $\gamma_i\mk$ is generated by  $\llbracket X, Y\rrbracket$ with $X\in \gamma_{i-1}\mk$ and $Y\in \mk$, the claim follows. \qedhere

\end{proof}

For a proof of the following result, we refer to \cite[2.1.1]{Lennox-Rob}

\begin{thm}[Rational Malcev completion] \label{t-Malcev-completion}
Let $N$ be a torsion-free nilpotent group. Then there exists a  divisible torsion-free nilpotent group $\mk_N$ and an injective group homomorphism  $\iota \colon N\to \mk_N$   such that  every element of $\mk_N$ has a power in $\iota(N)$. Moreover for if $(\mk', \iota')$ is another pair with this property, then there exists an isomorphism $\varphi\colon \mk'\to \mk_N$ such that $\varphi\circ \iota=\iota'$.  
 \end{thm}

The group $\mk_N$ from Theorem \ref{t-Malcev-completion} is called the \textbf{rational Malcev completion} of $N$. In the sequel we will drop the embedding $\iota$ from the notation and consider $N$ as a subgroup of $\mk_N$. When $N$ is abelian, $\mk_N$ is simply $N\otimes \Q$. Note that it follows from Theorem \ref{t-Malcev-completion} that every automorphism of $N$ extends to an automorphism of $\mk_N$.

If $H$ is a subgroup of a group $G$, we denote by $I_G(H)$ the set of elements $g \in G$ such that there exists $n \geq 1$ such that $g^n \in H$. $I_G(H)$ is called the \textbf{isolator} of $H$ in $G$. When working with the rational Malcev completion it is useful to keep in mind the following facts, see \cite[\S 2.3]{Lennox-Rob}.

\begin{lem}\label{lem-isolator-elementary}
Let $N$ be a nilpotent group and $H$ a subgroup of $N$. Then:
\begin{enumerate}[label=\roman*)]
\item $I_N(H)$ is a  subgroup of $N$.
\item If $N$ is divisible, then so is $I_N(H)$.
\item If $N$ is finitely generated, then $H$ has finite index in $I_N(H)$.  
\end{enumerate}
\end{lem}

\begin{prop}\label{p-isolator}
Let $\mk$ be a divisible torsion-free nilpotent group. 
\begin{enumerate}[label=\roman*)]
\item \label{i-lie-span} If $H$ is a subgroup of $\mk$, then  $\Lie(H) = \operatorname{Vect}(H) =  I_\mk(H)$.
\item If $H_1, H_2$ are finitely generated subgroups of $\mk$ such that $\Lie(H_1)=\Lie(H_2)$, then $H_1$ and $H_2$ are commensurable. 
\end{enumerate}
\end{prop}

\begin{proof}
$I_\mk(H)$ is a subgroup of $\mk$ by Lemma \ref{lem-isolator-elementary}, and we have the inclusions $H\le I_\mk(H) \le \operatorname{Vect}(H) \le \Lie(H)$. But since $I_\mk(H)$ is divisible, Theorem \ref{t-Malcev-divisible} implies that $I_\mk(H)$ is  a Lie subalgebra of $\mk$. Since $I_\mk(H)$ contains $H$, we must have $I_\mk(H) =\Lie(H)$.

For the second statement, since each $H_i$ is finitely generated, $H_1 \cap H_2$ has finite index in $I_{H_i}(H_1 \cap H_2)$ by Lemma \ref{lem-isolator-elementary}. Now since $\Lie(H_1)=\Lie(H_2)$ we have $I_{H_i}(H_1 \cap H_2) = H_i$, so  $H_1 \cap H_2$ is indeed of finite index in both $H_1$ and $H_2$.
\end{proof}

Assume that $N$ is a torsion-free nilpotent group. Then the derived subgroup $[N, N]$ is always contained in $\llbracket \mk_N, \mk_N\rrbracket$, but in general it might be strictly smaller than $N\cap \llbracket \mk_N, \mk_N\rrbracket$: indeed  $N^{ab}$ may have non-trivial torsion, while $\mk_N/\llbracket \mk_N, \mk_N\rrbracket$ is a torsion-free group. For later use we record the following lemma, which clarifies the difference.

\begin{lem}\label{l-malcev-ab}
Let $N$ be a torsion-free nilpotent group, with abelianization map $\pi_{ab}\colon N\to N^{ab}$.  Then $\pi_{ab}^{-1}(T(N^{ab})) =N\cap \llbracket \mk_N, \mk_N\rrbracket$. In particular we have $N^{ab}/T(N^{ab}) \simeq N/ (N\cap \llbracket \mk_N, \mk_N\rrbracket)$ and $N^{ab}\otimes \Q\simeq \mk_N/\llbracket \mk_N, \mk_N\rrbracket$.\end{lem}

\begin{proof}
Write $T=T(N^{ab})$ and $A = N^{ab}/T$. The canonical projection  $p\colon N\to A$ extends to an epimorphism  $ \mk_N \to \mk_A$ \cite[Cor.\ 2.41]{Baumslag-lecturenotes}, which must factor through $\mk_N/\llbracket\mk_N, \mk_N\rrbracket$. In particular $p$ descends to a map $p'\colon N/(N\cap \llbracket\mk, \mk\rrbracket)\to A$.  But since $N/(N\cap \llbracket\mk, \mk\rrbracket)$ is a torsion-free abelian quotient of $N$ and $A$ is the largest such quotient, the map $p'$ is an isomorphism. Moreover since $\mk_N/\llbracket \mk_N, \mk_n\rrbracket$ is the rational Malcev completion of $N/(N\cap \llbracket\mk, \mk\rrbracket)$, $p'$ extends to an isomorphism $\mk_N/\llbracket \mk_N, \mk_n\rrbracket \to \mk_A$. Since $\mk_A \simeq A \otimes \Q \simeq N^{ab}\otimes \Q$, the statement follows. \qedhere
\end{proof}

\subsection{Non-foldable subsets via strong irreducibility} \label{subsec-strong-irr}

Let $\corps$ be a field. Recall that a subgroup $G \subset \GL(n, \corps)$ is \textbf{irreducible} if $G$ does not preserve any proper non-trivial subspace of $\corps^n$, and \textbf{strongly irreducible} if $G$ does not preserve any finite union of non-trivial proper subspaces. Equivalently $G$ is strongly irreducible if every finite index subgroup of $G$ is irreducible. By extension we will say that a linear representation $\rho\colon G\to \GL(n,\corps)$ is (strongly) irreducible if its image is so. 

If $G$ is a subgroup of $\GL(n, \corps)$, we denote by $\overline{G}$ its Zariski closure, and by $\overline{G}\,^0$ the connected component of the identity in $\overline{G}$ with respect to the Zariski topology. Recall that  $\overline{G}\,^0$  can be equivalently defined as the unique Zariski-closed subset of $\overline{G}$ which contains the identity and which cannot be written as a union of closed strict subsets (not necessarily disjoint). The group $\overline{G}\,^0$ is a closed finite index subgroup of $\overline{G}$ and is contained in every closed finite index subgroup of $\overline{G}$.

We will need the following lemma, which may be compared with \cite[Lemma 3.13]{Glas-IRS}. 

\begin{lem}\label{l-strongly-irreducible}
	Suppose that $G\le \GL(n, \corps)$ is strongly irreducible. Then for every finite subset $\Sigma\subset \corps^n\setminus{0}$ of non zero vectors, and every finite collection $V_1, \cdots, V_m$ of strict subspaces of $\corps^n$, there is $g\in G$ such that $g(\Sigma)\cap V_i=\varnothing$ for every $i=1,\cdots, m$. 
\end{lem}

\begin{proof}
Note that $\overline{G}$ is also strongly irreducible as it contains $G$, and hence $\overline{G}\,^0$ is irreducible. We argue by contradiction and assume that the conclusion fails for  $\Sigma=\{w_1,\cdots, w_\ell\}$ and subspaces $V_1,\cdots, V_m$. Write $Y_{i, j}=\{g\in \GL(n, \corps)\colon g(v_i)\in V_j\}$ for $i=1,\cdots, \ell$ and $j=1,\cdots, m$, so that  $G\subset \cup_{i, j} Y_{i, j}$. Since every set $Y_{i, j}$ is Zariski-closed, we have $\overline{G}\subset \cup_{i, j} Y_{i, j}$, and thus there exists $i, j$ such that $\overline{G}\,^0\subset Y_{i, j}$. In particular the linear span of $\overline{G}\,^0v_i$ is contained in $V_j$. So we have found a strict non-trivial $\overline{G}\,^0$-invariant subspace, contradicting that $\overline{G}\,^0$ is irreducible. \qedhere
\end{proof}

\begin{defin}
Assume  that $G$ is a group  which can be written as an extension
\begin{equation}\label{e-extension} 1\to N\to G\to Q\to 1.\end{equation}
The conjugation action of  $Q$ on $N^{ab}$  gives rise to a linear representation of $Q$ on the vector space $H_1(N, \Q)=N^{ab}\otimes \Q$ over $\Q$.  We say that the extension \eqref{e-extension} is \textbf{strongly irreducible}  if $\dim_\Q(N^{ab}\otimes \Q)<\infty$ and the associated linear representation $\rho\colon Q\to \GL(N^{ab}\otimes \Q)$ is strongly irreducible. 
\end{defin}

We are now ready to state the main result of this section. 

\begin{thm}\label{t-strongly-irreducible} Let $G$ be a group that can be written as a strongly irreducible extension \[1\to N\to G\to Q\to 1,\]
where $N$ is a nilpotent group.  Let $\mathcal{L}\subset N$ be a subset whose projection to $N^{ab}/T(N^{ab})$ is injective. Then $\mathcal{L}$ is a non-foldable subset of $G$. \end{thm}

\begin{proof}
In the sequel we fix a finite subset $P$ of $N$ such that the image of every element of $P$ in $N^{ab}/T(N^{ab})$ is non-trivial, and we aim to show $\bigcap_{H\in S_G(P, G)} H\neq \{1\}$. By Lemma \ref{l-non-foldable-confined} this is equivalent to the conclusion of the theorem. 

Set $M:=N/T(N)$,  and let $\mk:=\mk_{M}$  be the rational Malcev completion of $M$. By Lemma \ref{l-malcev-ab} we have an isomorphism  $\left(N^{ab}/T(N^{ab})\right)\otimes \Q \simeq \mk^{ab}$, so that the present assumption tells us that the representation $\rho\colon Q\to \GL(\mk^{ab})$ is strongly irreducible. Denote by $\pi\colon N\to \mk$  the composition of the quotient map from $N$ to $M$ and of the inclusion $M\hookrightarrow \mk$, and $\pi^{ab}$ be the post-composition of $\pi$ with $\mk \to \mk^{ab}$. Write $d=\dim_\Q (\mk^{ab})$.

We construct inductively a family of subsets $P_1, P_2, \cdots, P_k$ with $k \leq d$ with the following properties: \begin{itemize}
	\item for each $k$ there exists $g_k \in G$ such that $P_k = g_k P g_k^{-1}$;
	\item  for every  $\sigma = (h_1, \cdots, h_k) \in P_1 \times \cdots \times P_k$, the subspace $W_\sigma = \operatorname{Vect} (\pi^{ab}(h_1), \cdots, \pi^{ab}(h_k))$ is a $k$-dimensional subspace  of $\mk^{ab}$. 
\end{itemize}

For $k = 1$ we set $P_1=P$. Since $\pi^{ab}(h)$ is non-zero for every $h \in P$, the second condition above holds. Assume that for $k < d$ we have constructed $P_1, \cdots, P_k$ with the desired properties. When $\sigma$ ranges over $P_1 \times \cdots \times P_k$, the subspaces $W_\sigma$ form a finite collection of proper subspaces of $\mk^{ab}$. Since $\rho\colon Q\to \GL(\mk^{ab})$ is strongly irreducible and since the subset $P$ is finite, according to Lemma \ref{l-strongly-irreducible} we can find an element $q\in Q$ such that $\rho(q)(\pi^{ab}(P)) \cap W_\sigma=\varnothing$ for all $\sigma \in P_1\cdots \times \cdots P_k$.  Let $g_{k+1}\in G$ be a preimage of $q$ in $G$, and set $P_{k+1}= g_{k+1} P g_{k+1}^{-1}$, so that $\pi^{ab}(P_{k+1})=\rho(q)(\pi^{ab}(P))$. Then by construction the $\pi^{ab}$-image of  every $(k+1)$-tuple $(h_1,\cdots, h_{k+1})\in P_1\times \cdots \times P_{i+1}$ generates an $(k+1)$-dimensional subspace of $\mk^{ab}$. Continuing in this way up to $d$, we arrive at finite subsets $P_1, \cdots, P_d$ such that for any choice of $h_1\in P_1,\cdots, h_d\in P_d$, the elements $\pi^{ab}(h_1), \ldots, \pi^{ab}(h_d)$ form  a basis of $\mk^{ab}$. Hence by Lemma \ref{l-nilpotent-abelianisation} the elements  $\pi(h_1), \ldots, \pi(h_d)$ generate $\mk$ as a Lie algebra.

We write $\Sigma = P_1\times \cdots \times P_d$, and for $\sigma=(h_1,\ldots, h_d)\in \Sigma$, we denote by $H_\sigma$ the subgroup of $N$ generated by $h_1,\cdots, h_d$. Let $H\in S_G(P, G)$. Since $H$ intersects all conjugates of $P$, in particular $H$ intersects all the sets $P_i$. Hence for every $H\in S_G(P, G)$ there exists $\sigma \in \Sigma$ such that $H$ contains $H_\sigma$. Hence in order to terminate the proof it is enough to justify that $\bigcap_\sigma H_\sigma$ is non-trivial. Since all the subgroups $H_\sigma$ are finitely generated and verify $\Lie(\pi(H_\sigma))=\mk$, all the subgroups $\pi(H_\sigma)$ are commensurable by Proposition \ref{p-isolator}. So $\bigcap_\sigma \pi(H_\sigma)$ has finite index in each of them, and hence in particular is infinite. Let $\Delta=\langle P_1, \cdots, P_d\rangle$. Being a finitely generated nilpotent group, $\Delta$ is virtually torsion-free. Let $\Delta_0$ be a finite index subgroup of $\Delta$ such that $\Delta_0\cap T(N)=\{1\}$. So the restriction of $\pi$ to $\Delta_0$ is injective. For $\sigma\in \Sigma$ set $K_\sigma:=H_\sigma \cap \Delta_0$. Since $K_\sigma$ has finite index in $H_\sigma$, we still have that $\Lie(\pi(K_\sigma))=\Lie(\pi(H_\sigma))=\mk$ by Proposition \ref{p-isolator}. Thus the same argument as before tells us that $\bigcap_\sigma \pi(K_\sigma)$ is non-trivial. Since $\pi$ is injective on $\Delta_0$, this shows that $\bigcap_\sigma K_\sigma$ is non-trivial, and the proof is complete.   \qedhere
\end{proof}

\subsection{Solvable groups of finite rank}

The goal of this part is to  prove Theorem \ref{t-finite-rank-exp}. For this we need some preliminaries on solvable groups of finite rank, which will allow to reduce the proof of this theorem to the case of a group $G$ falling into the setting of Theorem \ref{t-strongly-irreducible}.

Given a group $G$, we denote $\Res(G)$ the \textbf{finite residual} of $G$, which is defined as the intersection of all  subgroups of finite index of $G$. Note that $\Res(G)$ is a normal subgroup of $G$,  and $G/\Res(G)$ is the largest residually finite quotient of $G$. We call a group \textbf{quasi-cyclic} if it is isomorphic to $\Z[\frac{1}{p}]/\Z$ for some prime $p$. 

The following proposition summarizes the main structural properties of solvable groups of finite rank. Proofs can be found in \cite[\S 5]{Lennox-Rob}.

\begin{thm}\label{thm-preliminaries-prufer}
Let $G$ be a finitely generated solvable group of finite rank. Set $R:=\Res(G)$, $N:=\Fit(G)$. Then the following hold:
\begin{enumerate}[label=(\roman*)]
\item $N$ is nilpotent and $Q = G/N$ is virtually abelian. 
\item $R$ is a direct product of finitely many quasi-cyclic groups.
\item  $G/R$ is virtually torsion-free. Moreover it is linear over $\Q$. 
\end{enumerate}
\end{thm}

   The following is a classical consequence of a criterion of Hall. We include a proof for completeness. 
   
   \begin{prop} \label{p-finite-rank-non-virt-nilp}
   Let $G$ be a finitely generated solvable group of finite rank, and assume that $N$ is a nilpotent normal subgroup of $G$ such that $Q=G/N$ is virtually abelian.  Set $V=N^{ab}\otimes \Q$ and let  $\rho\colon Q\to \GL(V)$ be the associated linear representation. Then the following are equivalent. 
   \begin{enumerate}[label=\roman*)]
   \item \label{i-virt-nilp} The group $G$ is virtually nilpotent.
   \item \label{i-roots-unity} There exists $n\ge 0$ such that the complex eigenvalues of $\rho(q)$ are $n$th roots of unity for every  $q\in Q$.    \end{enumerate}
   
\end{prop}

\begin{proof} 
\ref{i-virt-nilp} $\Rightarrow$ \ref{i-roots-unity}. Let $G_0$ be a finite index subgroup of $G$ containing $N$ that is nilpotent, and let $Q_0$ be its image in $Q$. Write $T=T(N^{ab})$  and $A=N^{ab}/T$, so that $V\simeq A \otimes \Q$. Let $\tilde{T}$ be the preimage of $T$ in $N$. Then $\tilde{T}$  is  normal in $G_0$ and  $H=G_0/\tilde{T}$ satisfies $A\unlhd H$ and $H/A=Q_0$. Then there exists $m$ such that for every $h\in H$ and $a\in A$ the $m$-fold iterated commutator $[h, a]_m:=[h, [h,\cdots, [h,a]\cdots]]$ is trivial. But in $V= A\otimes \Q$ we have the equality $[h, a]_m \otimes 1=(\rho(q)-1)^m v$ where $ v=a\otimes 1$ and $q$ is the projection of $h$ to $Q$. It follows that $\rho(Q_0)$ consists of unipotent elements, and since every element of $Q$ has a power in $Q_0$, the claim follows.

For the converse, suppose first that \ref{i-roots-unity} $\Rightarrow$  \ref{i-virt-nilp}  holds true in the case where $N$ is abelian. Since assumption  \ref{i-roots-unity} depends on $N^{ab}$ rather than $N$,  the group $G/N'$ still satisfies \ref{i-roots-unity}. Hence by the current assumption we infer that $G/N'$ is virtually nilpotent. Since $N$ is nilpotent, we deduce from a result of Hall \cite[1.2.17]{Lennox-Rob} that $G$ is also virtually nilpotent. Hence it is enough to see that \ref{i-roots-unity} $\Rightarrow$  \ref{i-virt-nilp}  is true when $N$ is abelian. In that case $G$ is virtually metabelian, and hence residually finite.  So by Theorem \ref{thm-preliminaries-prufer} $G$ is virtually torsion-free. Upon modding out by the torsion subgroup of $N$, we can assume that $N$ is torsion-free. Also upon passing to a finite index subgroup, \ref{i-roots-unity}  tells us that every element of $Q$ is unipotent. By a standard argument the group $Q$ acts unipotently on $V = N\otimes \Q$, and it follows that $G$ is nilpotent. \qedhere
 \end{proof}

In the sequel we say that a linear automorphisms $f\in \GL(V)$  is strongly irreducible if $\langle f\rangle$ is a strongly irreducible subgroup of $\GL(V)$, i.e. if $f^n$ does not preserve any non-zero proper subspace of $V$ for all $n\neq 0$. 

The following is a basic lemma from linear algebra.

\begin{lem}\label{l-strongly-irreducible-power}
	Consider a finite dimensional vector space $V$ over $\Q$, and $f \in \GL(V)$ such that the complex eigenvalues of $f$ are not all roots of unity. Then there exists $m>0$ and a non-zero subspace $W\subset V$ such that $W$ is invariant by $f^m$ and the restriction of $f^m$ to $W$ is strongly irreducible and has no root of unity as a complex eigenvalue. 
\end{lem}

\begin{proof}
	Let $p\in \Q[T]$ be the minimal polynomial of $f$, and $p=p_1^{r_1}\cdots p_k^{r_k}$ the factorisation of $p$ in $\Q[T]$, where the $p_i$s are irreducible and coprime. Then one of the irreducible factors of $p$, say $p_1$, does not admit any root of unity among its complex roots. The subspace $V'=\ker p_1(f)$  is $f$-invariant and non-zero, and the complex eigenvalues of $f|_{V'}$ are precisely the roots of $p_1$, thus none of them is a root of unity. Among pairs $(W, m)$, where $m\ge 1$ and $W\subset V'$ is a non-zero $f^m$-invariant subspace, choose one where $W$ has minimal dimension. Then by construction $f^m|_W$ is strongly irreducible, and since  $f|_{V'}$ does not admit any root of unity as a complex eigenvalue, the same holds for $f^m|_W$. \qedhere
\end{proof}

The following result, which might be of independent interest, will reduce the proof of Theorem \ref{t-finite-rank-exp} to finitely generated solvable groups of finite rank of a particular form.

\begin{prop}\label{prop-non-virt-nilp-subgroup}
Let $G$ be a finitely generated solvable group of finite rank which is not virtually nilpotent. Then $G$ admits a finitely generated subgroup $H$ such that: 
\begin{enumerate}[label=\roman*)]
	\item $H$ is not virtually nilpotent;
\item  $H$ splits as a semi-direct product $H=M\rtimes \langle t \rangle$, where $M$ is nilpotent;
\item the action of $t$ on  $M$ induces a strongly irreducible automorphism of  $M^{ab}\otimes \Q$.
\end{enumerate}
\end{prop}

\begin{proof} 
Set $N=\Fit (G)$, $Q=G/N$ and $\rho\colon G\to \GL(N^{ab}\otimes \Q)$ be the associated representation. Recall from Theorem \ref{thm-preliminaries-prufer} that $N$ is nilpotent and $Q$ is virtually abelian. Suppose for a moment that $N$ is torsion-free. Let $\nk$ be the rational Malcev completion of $N$, so that we can  identify $N^{ab}\otimes \Q$ with $\nk^{ab}$ as $\Q Q$-modules (see Lemma \ref{l-malcev-ab}). Since $G$ is not virtually nilpotent, there exists $q\in Q$ such that $\rho(q)$ has an eigenvalue which is not a root of unity (Proposition \ref{p-finite-rank-non-virt-nilp}). Let $g\in G$ be an element that projects to $q$. The automorphism induced by the conjugation action of $g$ on $N$ extends to a Lie algebra automorphism, denoted $g_*\in \aut(\nk)$.  

Let $\mathcal{E}$ be the set of pairs $(m, \hk)$, where $m>0$ is a integer and $\hk\subset \nk$ is a $g_*^m$-invariant Lie subalgebra with the property that the linear automorphism of $\hk/\llbracket \hk, \hk\rrbracket$ induced by $g_*^m$ has an eigenvalue which is not a root of unity. Note that $\mathcal{E}\neq \varnothing$ as it contains $(1, \nk)$. Choose a pair $(m, \hk)\in \mathcal{E}$ such that $\dim_\Q \hk$ is minimal.  We claim that $g_*^m$ induces a strongly irreducible element of $\GL(\hk/\llbracket \hk, \hk\rrbracket)$. By Lemma \ref{l-strongly-irreducible-power} applied to $g_*^m$, there exists  $\ell>0$ and a $g_*^{m\ell}$-invariant  subspace $V\subset \hk/\llbracket \hk, \hk\rrbracket$ such that the restriction of $g_*^{m\ell}$ to $V$ is strongly irreducible and has no root of unity as an eigenvalue. Then the preimage $\hk'\subset \hk$ of $V$ is a Lie subalgebra $\hk'$ of  $\hk$ such that $(\hk', m\ell)\in \mathcal{E}$, so the minimality of $\hk$ implies $\hk'=\hk$. Therefore $V$ is equal to $\hk/\llbracket\hk, \hk\rrbracket$, and $g_*^m$ is strongly irreducible on $\hk/\llbracket \hk, \hk\rrbracket$.  We have $\hk = \operatorname{Vect}(\hk \cap N)$ by Proposition \ref{p-isolator}, so there exist $n_1,\ldots,n_k \in N$ that form a basis of the $\Q$-vector space $\hk$. Let $t=g^m$. and $H = \langle n_1,\ldots,n_k, t\rangle = M \rtimes \langle t\rangle$, where $M = \langle t^i n_j t^{-i,} : i \in \Z, j = 1,\ldots,k\rangle $ is contained in $N$. Then $\hk$ is isomorphic to the rational Malcev completion of $M$, so by Lemma \ref{l-malcev-ab} we have $M^{ab}\otimes \Q\simeq \hk/\llbracket \hk, \hk\rrbracket$.  By Proposition \ref{p-finite-rank-non-virt-nilp}) the subgroup $H$ is not virtually nilpotent, and hence $H$ satisfies all the conclusions. 

Now for the general case, we can apply the previous argument to $G / T(N)$, and deduce the existence of a subgroup $H = M \rtimes \langle t \rangle$ of $G / T(N)$ with $M \leq \Fit(G / T(N)) =  \Fit (G)/T(N)$, that satisfies the conclusion. If $\tilde{H}$ is a finitely generated subgroup $G$  of the form $\tilde{H} = \tilde{M} \rtimes \langle \tilde{t}  \rangle$ with $\tilde{M}/T(N) = M$, then we have $\tilde{M}^{ab} \otimes \Q \simeq M^{ab} \otimes \Q$ since $T(N)$ is a torsion group, and it follows that $\tilde{H}$ satisfies the desired properties.
\end{proof}

\begin{thm} \label{t-finite-rank-exp}
Let $G$ be a finitely generated solvable group of finite rank, and suppose that $G$ is not virtually nilpotent. Then $G$ has a Schreier growth gap  $\exp(n)$. 
\end{thm}

\begin{proof}
	Since the group $G$ is not virtually nilpotent, we can apply Proposition \ref{prop-non-virt-nilp-subgroup}. Let $H=M\rtimes \langle t\rangle$ be a subgroup of $G$ as in the conclusion of Proposition \ref{prop-non-virt-nilp-subgroup}. Theorem \ref{t-strongly-irreducible} applies to $H$, and implies that if $\mathcal{L}$ is a subset of $M$ that projects injectively to $A := M^{ab}/T(M^{ab})$, then $\mathcal{L}$ is a non-foldable subset of $H$. A fortiori $\mathcal{L}$ is a non-foldable subset of $G$.
	
	 Set $\overline{H}:=A\rtimes \langle t\rangle$. Since $A\otimes \Q=M^{ab}\otimes \Q$ as $\Q[\langle t\rangle]$-modules, Proposition \ref{p-finite-rank-non-virt-nilp} implies that the group $\overline{H}$ remains not virtually nilpotent. Thus $\overline{H}$ has exponential growth. Since $\overline{H}/A\simeq \Z$, it follows that the relative growth of $A$  in $\overline{H}$ is exponential. Choose a finite generating subset $S$ of $H$ and let $\overline{S}$ be its projection to $\overline{H}$. Let $\mathcal{L}\subset M$ be a set theoretic section of $A$ such that for every $a \in A$, the unique lift of $a$ in $\mathcal{L}$ has minimal length with respect to the word metric defined by $S$ among the elements of $M$ that project to $a$. Since the relative growth of $A$  in $\overline{H}$ is exponential, we have $\relvol_{(G,\mathcal{L})}(n) \simeq \exp(n)$. Since $\mathcal{L}$ is non-foldable according to the first paragraph, Lemma \ref{lem-exp-subset-growth} yields the conclusion.  \qedhere
\end{proof}

 \section{Confined subgroups and abelian normal subgroups}

  The goal of this section is to prove Proposition \ref{p-Neumann-abelien}, a key result which provides a tool to study confined subgroups of a group $G$ admitting an abelian normal subgroup $M$. It will be used later to show that under suitable conditions, a confined subgroup of $G$ must contain a large subgroup of $M$. 
 
 \begin{notation}
In all this section we fix once for all the following notation:\begin{itemize}
	\item $G$ is a group that can be written as an extension $1\rightarrow M \rightarrow G\rightarrow Q\rightarrow 1$, where $M$ is abelian. We denote by $\pi_Q$ the projection from $G$ to $Q$.
	\item If $H$ is a subgroup of $G$ and $q\in Q$, we set
	\[M_{q, H}:=\{m\in M \colon (1-q)m\in H\}.\]
\end{itemize}
 \end{notation}

 Recall that since $M$ is abelian, $M$ is naturally a $\mathbb{Z}Q$-module. As explained in \S \ref{subsec-notation}, in this setting we use additive notation for $M$.

  \begin{lem} \label{lem-MqH} 
 	Let $M,G,Q$ as above. For every subgroup $H$ of $G$ and $q\in Q$, the following hold:
 	\begin{enumerate}[label=\roman*)]
 		\item \label{i-Neumann-contained} $M_{q, H}$  is a subgroup of $M$, and $(1-q)M_{q, H}\le H$. 
 		\item \label{i-Neumann-submodule}  $M_{q, H}$ is a $\Z K$-submodule of $M$, where $K:=C_{\pi_Q(H)}(q)$. In particular if $q\in \pi_{Q}(H)$, then $M_{q, H}$ is a $\Z \langle q\rangle$-submodule of $M$.
 	\end{enumerate}
 \end{lem}
 
 \begin{proof}
 \ref{i-Neumann-contained} is clear from the definitions. To prove \ref{i-Neumann-submodule}, set $K=C_{\pi_Q(H)}(q)$. Note that $M \cap H$ is normalized by $H$ and thus it is a $\Z \pi_Q(H)$-submodule of $M$. In particular it is a $\Z K$-submodule of $M$. Thus for $f\in \Z K$ and $m\in M_{q, H}$, since $f$ commutes with $1-q$ and $(1-q)m\in M \cap H$ we have $(1-q) fm=f(1-q)m\in M\cap H$, showing that $fm\in M_{q, H}$. Thus $M_{q, H}$ is a $\Z K$-submodule. If $q\in \pi_Q(H)$, then clearly $q\in K$, so that $M_{q, H}$ is a $\Z \langle q \rangle $-submodule. 
 \end{proof}

 The proof of the following proposition uses a well-known lemma of B.H.\ Neumann \cite{Neum54}, stating that if a group $G=\cup_{i=1}^r g_iH_i$ is the union of $r$ cosets of subgroups, then at least  one of the subgroups $H_i$ has finite index at most $r$. An argument of similar spirit appears in the proof of Proposition 3.8 in \cite{LBMB-sub-dyn}.
 
 \begin{prop} \label{p-Neumann-abelien} 
Let $M,G,Q$ as above. Let $P$ be a finite subset of $G\setminus \{1\}$ , and set $r:=|P|$.	Then for every $H\in S_G(P, M)$, there exists $q\in \pi_Q(P) \cap \pi_Q(H)$ such that $M_{q, H}$ is a finite index subgroup of $M$, and $(M : M_{q, H}) \leq r$.
 \end{prop}

 \begin{proof}
Fix $g\in P$. Since all conjugates of $g$ by elements  of $M$ have the same projection to $Q$,  if $g$ has at least one such conjugate inside $H$, then $\pi_Q(g)\in \pi_q(H)$. This observation ensures that since $P$ is confining for $(H, M)$, the subset $P_H:=\{g\in P\colon \pi
 	_Q(g)\in \pi_Q(H)\}$ is already confining for $(H, M)$. 
 	
 	Now for $g\in P_H$, we write $Y_g=\{m\in M \colon mgm^{-1}\in H\}$. That $P_H$ is confining for $(H, M)$ means that $M=\bigcup_{P_H} Y_g$. Let $L_g$ be the subgroup of $M$ generated by the set of differences by elements in $Y_g$.  Then $Y_g$ is contained in a coset of $L_g$ and thus, by Neumann's Lemma, there exists $g_0\in P_H$ such that  $L_{g_0}$ has index at most $|P_H| \le r$ in $M$. So if we denote $q=\pi_Q(g_0)$, then in order to terminate the proof it is enough to check that $L_{g_0}\le M_{q, H}$. 
 	
 	Let $m, n\in Y_{g_0}$. Since $g_0^{-1}m^{-1}ng_0$ belongs to $M$ and $M$ is abelian, we have the equality
 	\begin{equation} \label{e-a-delta-gamma} m(g_0^{-1}m^{-1}ng_0)n^{-1}=mn^{-1}(g_0^{-1}m^{-1}ng_0)=(1-q)(m-n),\end{equation}
 	where in the first two terms we use multiplicative notation within the group $G$, while in the right-most term  we see $M$ as a $\Z Q$-module and  use additive notation. Since the first term in \eqref{e-a-delta-gamma} is the product of  $mg_0^{-1}m^{-1}$ and $ng_0n^{-1}$, which both belong to $H$,  we see that $m-n\in M_{q, H}$ for $m, n\in Y_{g_0}$. Thus $L_{g_0}\le M_{q, H}$, and the proof is complete. \qedhere
 	 \end{proof}

\section{Metabelian groups} \label{sec-metab}

\subsection{Non-foldable subsets in metabelian groups} \label{subsec-metab-non-foldable-tuples} The goal of this part is to prove Theorem \ref{thm-explicit-metab}, which will be our main tool to study growth of actions of metabelian groups. We need some preliminaries.

		If $Q$ is a finitely generated abelian group and $N$ is a $\Z Q$-module,  we denote by $C_Q(N)$ the centralizer of $N$ in $Q$, which is the set of $q \in Q$ such that $qn = n$ for all $n \in N$. By definition we have $C_Q(N) = Q \cap (1 + \ann(N))$ (the intersection is taken in $\Z Q$), where $\ann(N)$ is the annihilator of $N$ in $\Z Q$. In particular $\ann(N) = \left\lbrace 0 \right\rbrace$ implies $C_Q(N) = \left\lbrace 1 \right\rbrace$.

\begin{lem} \label{lem-virt-ab-FC}
	Let $Q = \langle q \rangle$ be a cyclic group, and  $M$ a finitely generated $\Z Q$-module. Assume that $q$ centralizes a subgroup $N$ of finite index in $M$. Then there is $n \geq 1$ such that $q^n$ centralizes $M$.
\end{lem}

\begin{proof}
	$N$ is finitely generated as a $\Z Q$-module, and $Q$ centralizes $N$, so $N$ is finitely generated as an abelian group. Since $N$ has finite index in $M$, $M$ is also a finitely generated abelian group. The automorphism induced by $q$ on $M/T(M)$ centralizes the finite index subgroup $NT(M)/T(M)$ of $M/T(M)$, and hence is trivial since $M/T(M)$ is a finitely generated free abelian group. So the conclusion holds with $n$ equals to the order of the automorphism induced by $q$ on the finite group $T(M)$. \qedhere
\end{proof}

\begin{lem} \label{lem-finite-number-submodules}
	Let $Q$ be a finitely generated group, and $M$ a finitely generated $\Z Q$-module.	Fix $r \geq 1$. Let $S_{r}$  be the set of $\Z Q$-submodules $L$ of $M$ such that $L$ has index at most $r$ in $M$. Then $S_{r}$ is finite.
\end{lem}

\begin{proof}
	The semi-direct product $G = M \rtimes Q $ is a finitely generated group, and for every $L \in S_{r}$ we have that $L \rtimes Q $ is a subgroup of $G$ of finite index at most $r$. Since $G$ is finitely generated it has only finitely many subgroups of index at most $r$. It follows that $S_{r}$ is finite.
\end{proof}

\begin{remark}
When $N$ is a $\Z Q$-module, the terminology \enquote{$N$ is torsion-free} might be ambiguous, as this might refer to the structure of abelian group or the module structure. In all the article when use this terminology we always mean that $N$ is torsion-free as an abelian group.
\end{remark}

\begin{lem} \label{lem-trick-non-zero-module}
	Let $Q$ be an abelian group, and $N$ a finitely generated $\Z Q$-module. Let $q_0 \in Q$ such that at least one of the following conditions hold:
	
	\begin{enumerate}[label=\roman*)]
		\item the abelian group $N$ is torsion-free and  $q_0\notin C_Q(N)$;
		
		\item $q_0^n\notin C_Q(N)$ for all $n \geq 1$.

	\end{enumerate}

	For $r\ge 1$, let $S_{r,q_0}$ be the set of $\Z \langle q_0 \rangle $-submodules $L$ of $N$ such that $L$ has index at most $r$ in $N$, and let $N_{r,q_0} := \bigcap_{L \in S_{r,q_0}} L$. Then $(q_0-1) N_{r, q_0}\neq \left\lbrace 0 \right\rbrace$. 
\end{lem}

\begin{proof}
	We first consider the case where $N$ is torsion-free. Since $k N \subset L$ for every $L \in S_{r,i}$, where $k = r!$, we have $k N \subset N_{r,i}$. Therefore $(q_0-1) N_{r,i} $ contains $(q_0-1) kN = k (q_0-1) N$, which is non-zero since $(q_0-1) N$ is non-zero and $N$ is torsion-free. So in that case the conclusion holds.
	
	We shall now assume that $q_0^n\notin C_Q(N)$ for all $n \geq 1$. We denote by $Q_0$ the subgroup of $Q$ generated by $q_0$. Since $N$ is a finitely generated $\Z Q$-module, there exist $x_1,\ldots,x_k \in N$ such that $N = \Z Q x_1 + \ldots + \Z Q  x_k$. For $1 \leq p \leq k$, let $J_p$ be the $\Z Q_0$-submodule generated by $x_p$. We shall first prove that there exists $p$  such that $(q_0^n-1) J_p \neq \left\lbrace 0 \right\rbrace$ for all $n \geq 1$. Suppose that this is not the case, i.e.\ for all $p$ there is $n_p \geq 1$ with $(q_0^{n_p}-1) J_p = \left\lbrace 0 \right\rbrace$. Then we have $(q_0^{n}-1) J_p = \left\lbrace 0 \right\rbrace$ for all $p$, where $n$ is the least common multiple of $(n_1,\ldots,n_k)$. Hence, since $Q$ is abelian, we deduce that we have $(q_0^{n}-1)  \Z Q x_p = \Z Q (q_0^{n}-1)   x_p = \left\lbrace 0 \right\rbrace$ for all $p$, and hence $(q_0^{n}-1) N = \left\lbrace 0 \right\rbrace$. This contradicts the assumption. So in the sequel we fix $p$ such that $(q_0^n-1) J_p \neq \left\lbrace 0 \right\rbrace$ for all $n \geq 1$.
	
	Let $S_{r,q_0}(J_p)$ be the set of $\Z Q_0$-submodules $L$ of $J_p$ such that $L$ has index at most $r$ in $J_p$, and let $N_{r,q_0}(J_p) := \bigcap_{L \in S_{r,q_0}(J_p)} L$. Since any $\Z Q_0$-submodule of $N$ of index at most $r$ in $N$ intersects $J_p$ along a $\Z Q_0$-submodule of $J_p$ of index at most $r$ in $J_p$, it follows that $N_{r,q_0}(J_p) \leq N_{r,q_0}$. Hence to prove the desired result it it is enough to prove that $(q_0-1) N_{r,q_0}(J_p) \neq \left\lbrace 0 \right\rbrace$. Now Lemma \ref{lem-finite-number-submodules} ensures that the set $S_{r,q_0}(J_p)$  is finite, and hence $N_{r,q_0}(J_p)$ has finite index in $J_p$. Therefore Lemma \ref{lem-virt-ab-FC} implies that if $(q_0-1) N_{r,q_0}(J_p) = \left\lbrace 0 \right\rbrace$, then there must exist $n \geq 1$ such that $(q_0^n-1)J_p = \left\lbrace 0 \right\rbrace$. By the definition of $p$ this cannot happen. So $(q_0-1) N_{r,q_0}(J_p) \neq \left\lbrace 0 \right\rbrace$, and the conclusion also holds in that case.
\end{proof}

We recall the following terminology from module theory. 
\begin{defin}
	A	$\Z Q$-module  $N$ is \textbf{uniform} if for every non-zero submodules $N_1,N_2$ of $N$, the submodule $N_1 \cap N_2$ is non-zero. 
\end{defin}

\begin{prop} \label{prop-construct-expand-metab}
	Let $G$ be a group that lies in a short exact sequence $1\to M \to G \to Q\to 1$, where $M,Q$ are abelian. Let $N$ be a finitely generated uniform $\Z Q$-submodule of $M$.  Let  $P=\{g_1,\cdots, g_r\}\subset G$ be a finite subset of $G$, and set $q_i=\pi_Q(g_i)$. Suppose that for every $i=1,\cdots, r$, at least one of the following conditions hold:
	
	\begin{enumerate}[label=\roman*)]
		\item \label{item-mod-ss-tor} the abelian group $N$ is torsion-free and  $q_i\notin C_Q(N)$;
		
		\item \label{item-mod-inf-cc} $q_i^n\notin C_Q(N)$ for all $n \geq 1$.
		
	\end{enumerate}

	Let $S_{r, i}$ be the set of $\Z \langle q_i\rangle $-submodules $L$ of $N$ of index at most $r$ in $N$, and set $N_{r, i}= \bigcap _{L \in S_{r,i}} L$. Then $J := \bigcap_{i=1}^r (q_i-1)N_{r,i}$ is a non-zero $\Z Q$-submodule of $M$, and $J$ is contained in $H$ for every $H \in S_G(P,N)$.
\end{prop}

\begin{proof}
	We observe that when $i$ is fixed, the set $S_{r,i}$ is globally invariant by $Q$ since $Q$ is abelian. Therefore $N_{r,i}$ is a $\Z Q$-submodule of $N$, and hence so is $(q_i-1)N_{r,i}$. By the assumptions we may apply Lemma \ref{lem-trick-non-zero-module}, which ensures that $(q_i-1)N_{r,i}$ is non-zero for each $i$. The module $N$ being uniform by assumption, it follows that $J$ is non-trivial. 
	
	Now given $H \in S_G(P,N)$, Proposition \ref{p-Neumann-abelien} ensures that there exist $i$ and $L \in S_{r, i}$ such that $(q_i-1)L \le H$. So in particular we have $(q_i-1)N_{r,i} \leq H$, and hence $J \leq H$ for every $H \in S_G(P,N)$. 
\end{proof}

\begin{remark}
	If $M$ is a non-zero finitely generated $\Z Q$-module, $M$ always admits a (non-zero) uniform submodule. This follows from the existence of associated primes (see \S \ref{s-metabelian-torsion-free}).
\end{remark}

\begin{defin}
	Let $\pi : G \to Q$ be a group homomorphism. We say that a subset $\mathcal{L}$ of $G$ is a \textbf{lift} of a subset $\mathcal{L}'$ of $Q$ if for every $q \in \mathcal{L}'$ there is a unique $g \in \mathcal{L}$ such that $\pi(g)=q$.  
\end{defin}

\begin{notation}
	If $A,B$ are subsets of a group $G$, we write $A \cdot B = \left\lbrace a b \, : \, a \in A, b \in B \right\rbrace$.
\end{notation}

\begin{thm} \label{thm-explicit-metab}
	Let $G$ be a finitely generated metabelian group that is an extension $1 \to M \to G \to Q \to 1$, where $M$ is abelian and $Q$ is free abelian, and suppose that there exists a submodule $N$ of $M$ such that $N$ is uniform and $C_Q(N) = \left\lbrace 1 \right\rbrace$. Let  $\mathcal{L}$ be a lift of $Q$. Then: 
	
	\begin{enumerate}[label=\roman*)]
		\item \label{item-non-foldable-gene} $\mathcal{L}$ is a non-foldable subset of $G$.
		\item \label{item-non-foldable-tf} If moreover $N$ is torsion-free, then for every non-zero element $m$ of $N$, the subset $\mathcal{L} \cdot \mathcal{Z}$ is a non-foldable subset of $G$, where $\mathcal{Z}$ is the cyclic subgroup generated by $m$.
	\end{enumerate}
\end{thm}

\begin{proof}
	In case where the group $N$ is torsion-free, we fix a  non-zero element $m$ of $N$, and denote by $\mathcal{Z}$ the subgroup generated by $m$. In that case we write $\mathcal{J} = \mathcal{L} \cdot \mathcal{Z}$. In case $N$ admits torsion, we set $\mathcal{J} = \mathcal{L}$. We shall prove that $\mathcal{J}$ is always a non-foldable subset of $G$.

	Fix a finite subset $\Sigma\subset \mathcal{J}$, and let $P=\{g^{-1}h\colon g, h\in \Sigma, g\neq h\}$. We will  prove that there is a non-zero $\Z Q$-submodule $L$ of $N$ such that $L$ is contained in $H$ for every $H\in S_G(P, G)$. This is enough to conclude according to  Lemma \ref{l-non-foldable-confined}. Note that since $N$ is a subgroup of $G$, we have the inclusion $S_G(P,G) \subset S_G(P,N)$, so it is enough to find a submodule $L$ contained in $H$ for every $H\in S_G(P, N)$.

	Suppose first that we make no torsion-free assumption for $N$, so that $\mathcal{J} = \mathcal{L}$. Then all the elements of $P$ have a non-trivial projection to $Q$ since $\mathcal{L}$ is a lift of $Q$. Since in addition $N$ is uniform, we are in position to apply Proposition  \ref{prop-construct-expand-metab} to the subset $P$ and the $\Z Q$-submodule $N$. Note that $G$ being finitely generated, $M$ is a finitely generated $\Z Q$-module \cite[11.1.1]{Lennox-Rob}. Since $\Z Q$ is a Noetherian ring \cite[Th.\ 1]{Hall}, it follows that $N$ is a finitely generated $\Z Q$-submodule of $M$. Condition \ref{item-mod-inf-cc} of the proposition is satisfied here because by assumption $C_Q(N) = \left\lbrace 1 \right\rbrace$ and $Q$ is free abelian. The conclusion of Proposition  \ref{prop-construct-expand-metab} therefore provides a non-zero $\Z Q$-submodule $L$ of $N$ that is contained in $H$ for  every $H \in S_G(P,N)$, as desired.

	We now deal with the case where $N$ is torsion-free. So here $\mathcal{J} = \mathcal{L} \cdot \mathcal{Z}$. Let $P'$ denote the elements of $P$ that have non-trivial projection to $Q$, and $P''$ the complement of $P'$ in $P$. The elements of $P''$ are precisely the elements of $P$ that belong to $\mathcal{Z}$.  Since $P''$ is finite, we may find an integer $s$ such that, if we set $m' = m^s$, then every subgroup of $G$ that intersects $P''$ contains $m'$. Since $H \cap P \neq \emptyset$ for all $H \in S_G(P, N)$, in particular for all such $H$ we have the following alternative: $m'  \in H$, or $H \cap P' \neq \emptyset$. Let $Y$ be the set of $H\in S_G(P, N)$ such that $m'  \notin H$. Since the group $N$ is abelian and $m' \in N$, the subset $Y$ in $N$-invariant. Therefore we deduce that we have $H \in S_G(P',N)$ for every $H\in Y$. As in the previous paragraph we can apply Proposition  \ref{prop-construct-expand-metab}, and find a non-zero $\Z Q$-submodule $N_1$ of $N$ which is contained in all  $H\in Y$. Let now $Z$ be the set of subgroups $H\in S_G(P, N)$ such that $N_1\not\le H$. By definition we have $Y \cap Z = \emptyset$. Then for every $H\in Z$, we have $m' \in H$ according to the above alternative. Since $N_1$  is a $\Z Q$-submodule, hence a normal subgroup of $G$,  the set $Z$ is $G$-invariant and $m'\in H$ for all $H\in Z$, and thus  the $\Z Q$-submodule $N_2$ generated by $m'$  satisfies $N_2\le H$  for every $H\in Z$. So  every subgroup in $S_G(P, N)$ contains  $N_1$ or $N_2$, and thus in all cases contains $N_1 \cap N_2$. Since $N_2$ is non-zero because $m' \neq 0$, we have that  $N_1 \cap N_2$ is non-zero since $N$ is uniform. Therefore all subgroups in $S_G(P, N)$ intersect non-trivially, as desired.
\end{proof}

\begin{cor} \label{cor-explicit-metab}
	Let $G$ be a finitely generated metabelian group that is an extension $1 \to M \to G \to Q \to 1$, where $M$ is abelian and $Q$ is free abelian of rank $d \geq 1$. Suppose that there exists a submodule $N$ of $M$ such that $N$ is uniform and $C_Q(N) = \left\lbrace 1 \right\rbrace$. Then whenever $\gamma_1,\ldots,\gamma_d \in G$ are lifts of generators of $Q$, $(\gamma_1,\ldots,\gamma_d)$ is non-foldable. If moreover $N$ is torsion-free, then $(m,\gamma_1,\ldots,\gamma_d)$ is non-foldable for every non-zero element $m$ of $N$. 
\end{cor}

\begin{proof}
	This is a direct consequence of Lemma \ref{lem-mapZk-inj} and Theorem \ref{thm-explicit-metab}.
\end{proof}

In the case of split extensions Theorem \ref{thm-explicit-metab} also implies the following:

\begin{cor}\label{c-metab-asdim}
Let $G=M\rtimes Q$ be a finitely generated group, where $M$ is abelian and $Q$ is free abelian of rank $d\ge 1$. Suppose that there exists a $\Z Q$-submodule $N$ of $M$ such that $N$ is uniform and $C_Q(N)=1$. Then for every faithful $G$-set $X$ we have $\asdim(G, X)\ge d$.
\end{cor}
\begin{proof}
Theorem \ref{thm-explicit-metab} implies that $Q$ is a non-foldable subset of $G$. Thus the conclusion follows from Proposition \ref{p-asdim-non-foldable}, since $\asdim(Q)=d$. 
\end{proof}

\subsection{First applications} \label{s-metabelian-first-examples}We now proceed to  apply Theorem \ref{thm-explicit-metab} to some explicit families of finitely generated metabelian groups. We will appeal to the following trivial lemma.

\begin{lem} \label{lem-prime-SI}
	Let $Q$ be a finitely generated abelian group, and let $\p$ be a prime ideal of $\Z Q$. Consider the $\Z Q$-module $N = \Z Q / \p$. Then the following hold:\begin{enumerate}[label=\roman*)]
		\item $C_Q(N) = Q \cap (1 + \p)$.
		\item $N$ is uniform.
		\item the abelian group $N$ is torsion-free if and only if $\p \cap \Z =  \left\lbrace 0\right\rbrace $.
	\end{enumerate}
\end{lem}

\begin{proof}
	The annihilator of $N$ is $\p$, so the first statement is clear. The submodules of $N$ are the ideals of the ring $\Z Q / \p$. If $I,J$ are non-zero ideals, then $I J \leq I \cap J$ and $IJ$ is non-zero since $\Z Q / \p$ is a domain. So $N$ is uniform. The last statement is also clear.
\end{proof}

\subsubsection{Wreath products}

\begin{thm} \label{thm-wreath}
Consider the wreath product $G=A\wr \Z^d$, where $A$ is a non-trivial finitely generated abelian group. Then $\Z^d$ is a non-foldable subset of $G$, and if $t\in A$ is an element of infinite order then $\Z^d \cdot \langle t\rangle$ is a non-foldable subset of $G$. In particular:
\begin{enumerate}[label=\roman*)]
\item \label{i-wreath-i} If $A$ is finite, then $G$ has a Schreier growth gap $n^d$. 
\item \label{i-wreath-ii} If  $A$ is infinite, then $G$ has a Schreier growth gap $n^{d+1}$. 
\item \label{i-wreath-asdim}  every faithful $G$-set $X$ satisfies $\asdim(G, X)\ge d$.
\end{enumerate}
Moreover these bounds are sharp.
\end{thm}

\begin{proof} Set $M=\oplus_{\Z^d} A$ and $Q=\Z^d$. 
Assume first that $A$ is finite, and choose an element $s\in A$ of prime order $p$. The $\Z Q$-submodule  $N$ generated by $S$ is  $N = \oplus_{\Z^d}  C_p \simeq \Z Q / (p)$. The module $N$ is indeed uniform by Lemma \ref{lem-prime-SI}, and clearly $C_Q(N)=\{1\}$. Thus we may apply Theorem \ref{thm-explicit-metab} conclusion of the theorem says that $\mathcal{L} = \Z^d$ is a non-foldable subset of $G$. The bound $\vol_{G, X}(n) \succcurlyeq  n^d$ follows from Lemma \ref{lem-exp-subset-growth} since $\relvol_{(G, \mathcal{L}) }(n) \simeq n^{d}$. This proves part \eqref{i-wreath-i}.

Now assume that $A$ is infinite and choose an element $s\in A$ of infinite order. Then the module generated by $s$ is $N=\oplus_Q\Z$, a free $\Z Q$-module of rank 1. It also follows from Theorem \ref{thm-explicit-metab} that the subset $\mathcal{L}=\Z^d \cdot \langle s \rangle$ is non-foldable (in particular, so is the subset $\Z^d$), and the claim on the growth follows again from Lemma \ref{lem-exp-subset-growth} since $\relvol_{(G, \mathcal{L})}(n)\simeq n^{d+1}$. The claim on asymptotic dimension follows similarly from Proposition \ref{p-asdim-non-foldable}, since $\asdim(\Z^d)=d$. 

Finally we justify that these bounds are sharp. If the group $A$ is finite, then the standard wreath product action of $G$  on $X=\Z^d\times A$ satisfies $\vol_{G, X}(n)\simeq n^d$ (\S \ref{s-wreath-actions}). If $A$ is infinite, then we may choose a faithful $A$-set $Y$ with $\vol_{A, Y}(n)\simeq n$ (see Proposition \ref{prop-virtually-abelian}). Then the action of $G$ on $X = B \times Y$, given by $((a_b),b_1) \cdot (b_2,y) = (b_1+b_2, a_{b_1b_2}y)$ is faithful, and $\vol_{G, X}(n)\simeq n^{d+1}$. Moreover in both cases the asymptotic dimension of the resulting action is $\asdim(G, X)=d$.  \qedhere
\end{proof}

\begin{remark}
 Theorem  \ref{thm-wreath} implies an analogous result when $A$ is an arbitrary finitely generated group,  because any non-trivial wreath product $G=A\wr \Z^d$ contains a subgroup isomorphic to either $C_p \wr \Z^d$ or $\Z \wr \Z^d$. However in this case the lower bounds obtained on the growth might not be sharp. 
\end{remark}

\subsubsection{Baumslag finitely presented groups} 
Here we consider the extended family of Baumslag groups $\Lambda_{p, d}$ from \S  \ref{subsec-Baumslag}, where $p$ is a prime and $d\ge 1$. Recall that $\Lambda_{p, d}=R_{p, d}\rtimes Q$,
	where   $Q=\Z^{2d}$ and $R_{p, d}:=\mathbb{F}_p[T_1,\ldots, T_d, T_1^{-1}, \ldots, T_d^{-1}, (1+T_1)^{1},\ldots (1+T_d)^{-1}]$. If we denote $t_1,\ldots, t_d, s_1,\ldots, s_d$ a basis of $\Z^{2d}$, then $t_i$ and $s_i$ act on $R_{p, d}$ respectively by multiplication by $T_i$ and $T_i+1$.

The ring $R_{p, d}$ is quotient of the ring $\Z Q\simeq \Z[T_1^{\pm 1},\cdots, T_{2d}^{\pm 1}]$ by the ideal generated by $p$ and by the polynomials $T_i-1 - T_{d+i}$ for $i=1,\ldots, d$, which is a prime ideal. Hence $R_{p, d}$ is a uniform $\Z Q$-module. Since moreover $C_Q(R_{p, d})=\{1\}$, all the assumption of Theorem \ref{thm-explicit-metab} are satisfied. We obtain:

\begin{thm}
For every   prime $p$ and every $d\ge 1$, 	$\Z^{2d}  $ is a non-foldable subset of $G = \Lambda_{p, d}$. In particular $G$ has a Schreier growth gap $n^{2d}$.
\end{thm}

We note that this lower bound is sharp because $\Lambda_{p, d}$ admits a faithful transitive $G$-set $X$ with $\vol_{G, X}(n) \simeq n^{2d}$ (Proposition  \ref{prop-growth-ol}).

\subsubsection{Free metabelian groups} Let $d \geq 2$, and let $\mathbb{FM}_d$ be the fee metabelian group of rank $d$, i.e.\ the quotient of the free group $F_d$ by its second derived subgroup. If  $x_1,\ldots,x_d$ are free generators of $F_d$, then for simplicity we still denote  $x_1,\ldots,x_d$ their images in $\mathbb{FM}_d$. The group $Q = \mathbb{FM}_d / \mathbb{FM}_d'$ is free abelian of rank $d$.

\begin{thm} \label{thm-freemetab-growth}
	Let $G = \mathbb{FM}_d$. For every non-trivial $m \in G'$, the $(d+1)$-tuple $(m,x_1,\ldots,x_d)$ is non-foldable. In particular $G$ has a Schreier growth gap $n^{d+1}$. Moreover this bound is sharp.
\end{thm}

\begin{proof}
Recall that the Magnus embedding is an injective homomorphism from  $\mathbb{FM}_d$ to the wreath product $\mathbb{Z}^d \wr Q $ \cite{Magnus-embed}. From this it is easy to see that  every non-trivial $m \in \mathbb{FM}_d'$ has trivial annihilator in $\Z Q$, so that the $\Z Q$-module $N$ generated by $m$ is free of rank one. In particular $C_Q(N) = \left\lbrace 1 \right\rbrace$ and $N$ is uniform. Therefore Corollary \ref{cor-explicit-metab} applies and yields the conclusion. The only thing that remains to be justified is the last claim. Since $\Z^d \wr \Z^d$ is isomorphic to a (finite index) subgroup of $ \Z \wr \Z^d$, the Magnus embedding implies in particular that $G = \mathbb{FM}_d$ embeds in $\Z \wr \Z^d$. Since the standard wreath product action of $ \Z \wr \Z^d$ has growth $n^{d+1}$, by restricting to $G$ we obtain a faithful $G$-set $X$ with $\vol_{G, X}(n) \simeq n^{d+1}$. So the lower bound $n^{d+1}$ from the statement is sharp.
\end{proof}

\subsection{Growth of actions and Krull dimension} \label{s-Krull}

We recall the following definitions. 

\begin{defin}
	Let $A$ be a commutative ring with unit. The Krull dimension of $A$, written $\dim(A)$, is the supremum of the lengths of all chains of prime ideals of $A$, where the length of the chain $\p_0 \subsetneq \cdots \subsetneq \p_{n}$ is $n$. The Krull dimension of a non-zero module $M$ over $A$ is defined as $\dim(M) = \dim(A/\ann(M))$. 
\end{defin}

We will need the following proposition proven by Jacoboni \cite[Prop.\ 4.2]{Lison}.

\begin{prop} \label{prop-lison}
	Let $Q$ be a finitely generated free abelian group, and $M$ a finitely generated $\Z Q$-module of Krull dimension $k$. Then one can find a subgroup $Q_0 \leq Q$ and $m \in M$ such that at least one of the following holds:
	\begin{enumerate}[label=\roman*)]
		\item $Q_0$ has rank $k-1$ and the $\Z Q_0$-module generated by $m$ is a free module;
		\item $Q_0$ has rank $k$ and the $\Z Q_0$-module generated by $m$ is isomorphic to $\mathbb{F}_p Q_0$ for some prime number $p$.
	\end{enumerate}
\end{prop}

For a proof of the following, see \cite[\S 2.2.3]{Lison}. 

\begin{prop-def}
	Let $G$ be a finitely generated metabelian group, and suppose $G$ is not virtually abelian. If $1 \to M \to G \to Q \to 1$ is a short exact sequence of groups with $M,Q$ abelian, then the Krull dimension of $M$ as a $\Z Q$-module is a positive integer that does not depend on the choice of $M,Q$. This integer is the \textbf{Krull dimension} of $G$.
\end{prop-def}

\begin{thm} \label{t-krull}
	Let $G$ be a finitely generated metabelian group, and suppose that $G$ is not virtually abelian. If $G$ has Krull dimension $k$, then there exists a non-foldable $k$-tuple $(g_1,\ldots,g_k)$ in $G$. In particular $G$ has a Schreier growth gap $n^{k}$.
\end{thm}

\begin{proof}
	Note that if $k=1$ then there is nothing to prove because every element of $G$ of infinite order satisfies the conclusion. So we assume $k \geq 2$. Write $M=G'$ and $Q=G/G'$, and apply Proposition \ref{prop-lison}. Let $Q_0$ and $m$ as in the conclusion of the proposition, and let $N$ be the $\Z Q_0$-submodule generated by $m$. We choose lifts $\gamma_1,\ldots,\gamma_d \in G$ of generators of $Q_0$, where $d \geq 1$ is the rank of $Q_0$, and we denote by $G_0$ the subgroup of $G$ generated by $m$ and $\gamma_1,\ldots,\gamma_d$. Then we have a short exact sequence $1 \to N \to G_0 \to Q_0 \to 1$. Since the module $N$ is uniform and satisfies $C_{Q_0}(N) = \left\lbrace 1 \right\rbrace$  since $N$ is either a free module or isomorphic to $\mathbb{F}_p Q_0$ for some prime number $p$, we can apply Corollary \ref{cor-explicit-metab}. In case $N \simeq \mathbb{F}_p Q_0$ we have $d=k$ and by the corollary $(\gamma_1,\ldots,\gamma_k)$ is non-foldable; and in case $N$ is free we have $d=k-1$ and again by the corollary $(m,\gamma_1,\ldots,\gamma_{k-1})$ is non-foldable. So in both cases we have found a non-foldable $k$-tuple, and hence the proof is complete.
\end{proof}

\subsection{Torsion-free metabelian groups} \label{s-metabelian-torsion-free}

In the sequel $A$ is a commutative ring with unit. The radical of an ideal $I$ of $A$ is denoted $\rad(I)$. Let $M$ be a module over $A$. Given $x \in M$, the annihilator of $x$ is denote $\ann(x)$. The annihilator of $M$ is denoted $\ann(M)$.  A prime ideal $\p$ of $A$ is \textbf{associated} with $M$ if there exists $x \in M$ such that $\p = \ann(x)$. We denote by $\ass(M)$ the set of associated prime ideals of $M$. We will use the following basic facts, a proof of which can be found for instance in \cite{Bourbaki-alg-comm-3-4}:

\begin{enumerate}
	\item If $A$ is Noetherian and $M$ is non-zero, then $\ass(M)$ is not empty. 
	\item \label{item-ass-ext}	If $N$ is a submodule of a module $M$, every associated prime of $N$ is associated with $M$, and every associated prime of $M$ is associated with $N$ or $M/N$.
	\item \label{item-radical} If $A$ is Noetherian and $M$ is a Noetherian $A$-module, then $\rad(\ann(M)) = \bigcap_{\p \in \ass(M)} \p$.
	\item \label{item-series} If $A$ is Noetherian and $M$ is a Noetherian $A$-module, then there exists a series of submodules $0 = M_0 \subsetneq M_1 \subsetneq \ldots \subsetneq M_n = M$ such that each $M_{i+1}/M_i$ is isomorphic to $A / \p_{i+1}$ for some prime ideal $\p_{i+1}$, and $\ass(M) \subseteq \left\lbrace \p_{1}, \ldots, \p_{n} \right\rbrace $.
\end{enumerate}

\begin{lem} \label{lem-associatedprime}
	Let $Q$ be a finitely generated abelian group. Let $\P$ be a property of $\Z Q$-modules such that the zero module has $\P$, and $\P$ is stable under taking finite direct product of modules, submodules, quotient modules, and extensions.
	
	Let $M$ be a finitely generated $\Z Q$-module that does not have $\P$. Then $M$ admits a submodule $N$ that is isomorphic to $\Z Q / \p$ for some prime ideal $\p$ of $\Z Q$ such that $N$ does not have $\P$ either.
\end{lem}

\begin{proof}
	Consider a series $0 = M_0 \subsetneq M_1 \subsetneq \ldots \subsetneq M_n = M$ as in (\ref{item-series}) above. So $M_{i+1}/M_i$ is isomorphic to $\Z Q / \p_{i+1}$, and $\ass(M) \subseteq \left\lbrace \p_{1}, \ldots, \p_{n} \right\rbrace $. Let $i$ be the least integer such that $M_i$ does not have $\P$. By the assumption on $M$ such an integer exists, and we have $ 1 \leq i \leq n$.

	Since $M_i$ admits $\Z Q / \p_{i}$ as a quotient module, we have $\ann(M_i) \subset \p_i$. hence $\rad(\ann(M_i))$ is contained in $\p_i$. We claim that $\p_i \in \ass(M_i)$. Suppose for a contradiction that this is not the case. Then $M_{i-1}$ and $M_i$ have the same associated primes (because every prime associated with $M_i$ but not with $M_{i-1}$ would be associated with $M_{i} / M_{i-1}$ by (\ref{item-ass-ext}), and $M_{i} / M_{i-1}$ has $\p_{i}$ as only associated prime since $M_{i} / M_{i-1}$ is isomorphic to $\Z Q / \p_{i}$). So by property (\ref{item-radical})  we deduce that $\rad(\ann(M_{i-1})) \subset \p_i$, and in particular $I := \bigcap_{j=1}^{i-1} \p_i \subset \p_i$. Now the module $\Z Q/I$ embeds in $\prod_{j=1}^{i-1} \Z Q/\p_j$. By the definition of the integer $i$, each module  $\Z Q/\p_j$ has $\P$, and since $\P$ is stable under taking products and submodules, $\Z Q/I$ also has $\P$. Therefore, being a quotient of $\Z Q/I$, the module $\Z Q/\p_i$ also has  $\P$, and finally so does $M_i$ as an extension of two modules $M_{i-1}$ and $\Z Q/\p_i$ that have $\P$. This is a contradiction. So we indeed have $\p_i \in \ass(M_i)$, and the statement holds with $N = \Z Q x$, where $x \in M_i$ is such that $\ann(x) = \p_i$. 
\end{proof}

\begin{lem} \label{prop-metab-fp}
	Let $G$ be a finitely generated metabelian group, and assume that $G$ is not polycyclic.  Choose abelian groups $M,Q$ such that $G$ is  an extension $1 \to M \to G \to Q \to 1$. Then there exists a submodule $N$ of  $M$ whose underlying abelian group is infinitely generated, and such that $N$ is isomorphic to $\Z Q / \p$ for some prime ideal $\p$ of $\Z Q$.
\end{lem}

\begin{proof}
	$M$ is a finitely generated $\Z Q$-module, and the assumption that $G$ is not polycyclic is equivalent to saying that $M$ is infinitely generated as an abelian group. The statement follows by applying  Lemma \ref{lem-associatedprime} with the property $\P$ equals \enquote{$M$ is finitely generated as an abelian group}.   
\end{proof}

\begin{thm} \label{thm-metab-notorsion}
	Let $G$ be a finitely generated torsion-free metabelian group, and suppose that $G$ is not virtually abelian. Then there exists a pair $(g_1,g_2)$ that is non-foldable.
\end{thm}

\begin{proof}
In case where $G$ is polycyclic, Proposition \ref{prop-expand-subgroup-poly} provides a subgroup $H$ of $G$ isomorphic either to $H_3(\Z)$ or $\Z^k \rtimes \Z$ with $k \geq 2$, such that $H$ is a non-foldable subset of $G$. It immediately follows that if $g_1,g_2$ are two non-trivial elements of $H$ such that the subgroups generated by $g_1$ and $g_2$ intersect trivially, then the pair $(g_1,g_2)$ is non-foldable.

	Now suppose that $G$ is not polycyclic. Choose abelian groups $M,Q$ such that $G$ lies in an extension $1 \to M \to G \to Q \to 1$, and choose $N$ as in Lemma \ref{prop-metab-fp}. So $N$ is isomorphic to $\Z Q / \p$ for some prime ideal $\p$ of $\Z Q$, and $N$ is infinitely generated as an abelian group. This last property ensures that the rank $d$ of $Q/ C_Q(N)$ is at least one. Let $Q_1$ be a free abelian subgroup of rank $d$ of $Q$ such that $Q_1$ intersects $C_Q(N)$ trivially. If $m \in N$ is a generator of $N$ as a $\Z Q$-module, that $Q_1$ intersects $C_Q(N)$ trivially is equivalent to saying that no non-trivial element of $Q_1$ centralizes $m$. Let $g_1,\ldots,g_d \in G$ be lifts of generators of $Q_1$, and let $G_1$ be the subgroup of $G$ generated by $m$ and $g_1,\ldots,g_d $. The group $G_1$ lies in a short exact sequence $1 \to L \to G_1 \to Q_1 \to 1$, and we denote by $N_1 \subseteq L$ the $\Z Q_1$-module generated by $m$. The annihilator $\p_1$ of $m$ in  $\Z Q_1$ equals $\p \cap \Z Q_1$, and hence is a prime ideal of $\Z Q_1$. Moreover $Q_1$ acts faithfully on $N_1$ because no non-trivial element of $Q_1$ centralizes $m$. So the $\Z Q_1$-module $N_1$ satisfies $C_{Q_1}(N_1) = \left\lbrace 1 \right\rbrace$ and $N_1$ is uniform by Lemma \ref{lem-prime-SI}. Therefore we may apply Corollary \ref{cor-explicit-metab}  to the group $G_1$. Since $N$ is torsion-free, the conclusion says that $(m,g_1,\ldots,g_d)$ is non-foldable in $G_1$. A fortiori $(m,g_1,\ldots,g_d)$ is also non-foldable in $G$. In particular the pair $(m,g_i)$ is non-foldable for all $i$, and we have proved the statement.  
\end{proof}

\subsection{Finitely presented metabelian groups}

In this section we show how our results from \S \ref{subsec-metab-non-foldable-tuples}, combined with results of Bieri and Strebel, provide the existence of expansive pairs in finitely presented metabelian groups. 

Let $Q$ be a finitely generated abelian group, and $M$ a finitely generated $\Z Q$-module. Following Bieri--Strebel \cite{BS78,BS80}, for $v \in \mathrm{Hom}(Q,\R)$ we denote by $Q_v$ the set of $q \in Q$ such that $v(q) \geq 0$, and by $\Sigma_M$ the set of $v$ such that $M$ is a finitely generated $\Z Q_v$-module. The module $M$ is called \textbf{tame} if $\mathrm{Hom}(Q,\R) = \Sigma_M \cup - \Sigma_M$. 
The main result of \cite{BS80} states that if $G$ is a metabelian group and $1\to M\to G\to Q\to 1$ is a short exact sequence where $M, Q$ are abelian, then $G$ is finitely presented if and only if $M$ is a tame $\Z Q$-module.

We record the following properties, proven in \cite[Proposition 2.5]{BS80}. 
\begin{enumerate}

\item \label{BS1} A submodule of a tame $\Z Q$-module is tame. 

\item \label{BS2} If $G$ is a tame $\Z Q$-module and $R\le Q$ is a finite index subgroup of $Q$, then $M$ is a tame $\Z R$-module.

\end{enumerate}

\begin{thm} \label{thm-metab-presfin}
	Let $G$ be a finitely presented metabelian group that is not virtually abelian. Then there exists a pair $(g_1,g_2)$ that is non-foldable.
\end{thm}

\begin{proof}
	The polycyclic case has already been treated in Theorem  \ref{thm-metab-notorsion}, so in the sequel we assume that $G$ is not polycyclic. Choose abelian groups $M,Q$ such that $G$ is  an extension $1 \to M \to G \to Q \to 1$, and choose $N$ as in Proposition \ref{prop-metab-fp}. 	Since the group $G$ is finitely presented, $M$ is a tame $\Z Q$-module according to \cite{BS80}. Therefore $N$ is also tame by\eqref{BS1} above. In the sequel we write $\overline{Q} = Q/ C_Q(N)$, which is an infinite abelian group. Note $N$ is naturally a tame $\Z \overline{Q}$-module.
	
	We first consider the case when $\overline{Q}$ is virtually cyclic. We choose an infinite cyclic subgroup $Q_1 \leq Q$ of finite index such that no non-trivial element of $Q_1$ centralizes $N$. By \eqref{BS2} above, $N$ is a finitely generated tame $\Z Q_1 $-module. Again by the main result of \cite{BS80}  the subgroup $G_1:=N\rtimes Q_1$ is finitely presented. According to Theorem A in \cite{BS78} this implies that $G_1$ splits as an  HNN-extension over some finitely generated subgroup of $N$. Since $G_1$ has no free subgroups, this HNN-extension is necessarily ascending. In this situation by Proposition 3.3 in \cite{BS78}  the torsion subgroup of $N$ is finite, and hence trivial here because $N$ is isomorphic to $\Z Q / \p$ for some prime ideal $\p$. Therefore in this situation the group $G_1  $  is torsion-free. Since $G_1$ is  not virtually abelian, $G_1$ falls under the scope of Theorem  \ref{thm-metab-notorsion}, and the existence of a non-foldable pair in is thus guaranteed.

	Now  we consider the case where the torsion-free rank $d$ of $\overline{Q}$ is at least $2$. As in the proof of Theorem  \ref{thm-metab-notorsion}, we choose a torsion-free subgroup $Q_1$ inside $Q$ of rank $d$ such that no non-trivial element of $Q_1$ centralizes $N$. We choose lifts $g_1,\ldots,g_d \in G$ of generators of $Q_1$ to $G$; and consider the subgroup $G_1$ of $G$ generated by $m$ and $g_1,\ldots,g_d $, where $m$ is a generator of $N$ as a $\Z Q$-module. Repeating the argument from the end of the proof of Theorem \ref{thm-metab-notorsion}, we apply Corollary \ref{cor-explicit-metab} to $G_1$, and deduce that $(g_1,\ldots,g_d)$ is non-foldable. This completes the proof.
\end{proof}

\begin{cor} \label{cor-metab-presfinie-growth}
	Let $G$ be a finitely presented metabelian group that is not virtually abelian. Then $G$ has a Schreier growth gap $n^2$. If moreover $G$ is torsion-free, then $G$ has a Schreier growth gap  $n^3$.
\end{cor}

\begin{proof}
	The first statement directly follows from Theorem \ref{thm-metab-presfin}. In order to prove the second statement, we assume that $G$ is torsion-free and follow the proof of Theorem \ref{thm-metab-presfin}. In the polycyclic case, any faithful $G$-set verifies  $\vol_{G, X}(n)  \succcurlyeq  n^4$ according to Proposition \ref{prop-n4-growth} and Corollary  \ref{cor-poly-growth}. Suppose now $G$ is not polycyclic. We retain the notation $N,Q,m, \ldots$ as above. In case $\overline{Q}$ is virtually cyclic, as in the previous proof we find a cyclic subgroup $Q_1$ of $Q$ such that $G_1 =  N \rtimes Q_1 $ is a finitely presented subgroup of $G$ which splits as an HNN-extension over a torsion-free finitely generated subgroup of $N$. It follows that  $N$ is an ascending union of finitely generated torsion-free abelian groups whose rank is bounded, and that $G_1$ is a group of finite rank. Hence in that case the conclusion is provided by Theorem \ref{t-finite-rank-exp}. In case $\overline{Q}$ has torsion-free rank $d \geq 2$, we choose $g_1,\ldots,g_d \in G$ whose projections generate a free abelian subgroup $Q_1$ of $Q$ in which $N$ has trivial centralizer, and we obtain from Corollary \ref{cor-explicit-metab} that $(m,g_1,\ldots,g_d)$ is non-foldable.  
\end{proof}

\begin{remark}
	When $G$ is finitely presented and torsion-free (and not virtually abelian), although faithful $G$-actions have growth $\succcurlyeq n^3$ by Corollary  \ref{cor-metab-presfinie-growth}, $G$ does not always admit a non-foldable triple $(g_1,g_2,g_3)$. For example for the Baumslag-Solitar groups $\mathrm{BS}(1,n) \simeq \Z [1/n] \rtimes \Z$, $n\geq 2$, it is not hard to see that for every triple $(g_1,g_2,g_3)$, the map $\Z^3 \to \mathrm{BS}(1,n), \, \, (n_1, n_2, n_3 ) \mapsto g_3^{n_3}  g_2^{n_2} g_1^{n_1}$, is non-injective. In particular no triple can be non-foldable (Lemma \ref{lem-mapZk-inj}). 
\end{remark}

\section{Torsion-free solvable groups} \label{sec-torsionfree}

Based on the results from the previous sections, we make the following conjecture:

\begin{conjecturebis} \label{conj-torsion-free}
	Let $G$ be a finitely generated solvable group which is virtually torsion-free, and not virtually abelian. Then $G$ has a Schreier growth gap  $n^2$. 
\end{conjecturebis}

The purpose of this section is to establish the following results:

\begin{itemize}
	\item Conjecture \ref{conj-torsion-free} is true  if $G$ admits a nilpotent normal subgroup $N$ such that $G/N$ is virtually abelian. This includes in particular solvable linear groups.
	\item Conjecture \ref{conj-torsion-free} is true under a strengthening of the torsion-free assumption on $G$, see Theorem \ref{thm-fitting-torsionfree}.
	\item Conjecture \ref{conj-torsion-free} is true if we restrict to actions with finitely many orbits, see Theorem \ref{t-torsion-free-quasitransitive}.
\end{itemize}

\subsection{Nilpotent-by-abelian groups}

 First recall that we have proven in Theorem \ref{thm-metab-notorsion} that the above conjecture is true when $G$ is a metabelian group. This has the following straightforward consequence:
 
 \begin{cor} \label{c-metabelian-subgroup}
 Let $G$ be a torsion-free solvable group, and assume that $G$ admits a finitely generated metabelian subgroup which is not virtually abelian. Then $G$ contains an expansive pair. In particular Conjecture \ref{conj-torsion-free} is true for $G$. 
 \end{cor}

We recall the following:

\begin{lem} \label{lem-nilp-virtab-tf}
Let $N$ be a torsion-free nilpotent group. If $N$ is virtually abelian, then $N$ is abelian.
\end{lem}

\begin{proof}
Let $g,h \in N$. Since $N$ is virtually abelian, there are $m,n \geq 1$ such that $g^m$ and $h^n$ commute. Then $g^m h g^{-m}$ and $h$ have the same $n$-th power, and hence $g^m h g^{-m} = h$ since $N$ is torsion-free \cite[2.1.2]{Lennox-Rob}.  So $h^{-1} g h$ and $g$ have the same $m$-th power, and again $h^{-1} g h= g$.
\end{proof}

Recall that a group is nilpotent-by-abelian if $G$ admits a nilpotent normal subgroup $N$ such that $G/N$ is abelian.  Corollary \ref{c-metabelian-subgroup} readily implies the following. 
 
 \begin{thm} \label{t-torsion-free-nilp-by-abelian}
Let $G$ be a finitely generated torsion-free group that is nilpotent-by-abelian, and not virtually abelian. Then $G$ contains a non-foldable pair. In particular Conjecture \ref{conj-torsion-free} is true for $G$.
\end{thm}

\begin{proof} 
By Corollary \ref{c-metabelian-subgroup} it is enough to show that $G$ contains a finitely generated subgroup $H$ which is metabelian and not virtually abelian. Let $N$ be a nilpotent normal subgroup of  $G$ such that $G/N$ is abelian. If $N$ is not abelian, then $N$ is not virtually abelian by Lemma \ref{lem-nilp-virtab-tf}, and we can choose $H$ to be a copy of the Heisenberg group in $N$  (Lemma \ref{l-subgroup-H3}). If $N$ is abelian, then $G$ is metabelian, and we take $H=G$.  \qedhere 
\end{proof}

Since solvable linear groups are virtually nilpotent-by-abelian by a theorem of Malcev \cite[3.1.8]{Lennox-Rob},  Theorem \ref{t-torsion-free-nilp-by-abelian} implies:

\begin{cor} \label{cor-linear-case-quadratic}
Let $G$ be a finitely generated solvable linear group which is virtually torsion-free. If $G$ is not virtually abelian, then $G$ contains a non-foldable pair. In particular Conjecture \ref{conj-torsion-free} is true for $G$. 
\end{cor}

\begin{remark}
The above results are entirely based on the case of metabelian groups, proven in the previous section. However in general the reduction to metabelian subgroups is not enough to prove Conjecture \ref{conj-torsion-free}. As we show in a separate work in preparation, there exist finitely generated torsion-free solvable groups that are not virtually abelian and with the property that all finitely generated metabelian subgroups are virtually abelian.
\end{remark}

\subsection{Lifting non-foldable subsets} \label{subsec-lift-exp}

The goal of this paragraph is to prove Proposition \ref{prop-lift-nice-subsets}. This result enables, under suitable conditions, to lift a non-foldable subset from a quotient to a non-foldable subset of the ambient group. For solvable groups, this mechanism can be applied inductively, and  will be used in \S \ref{subsec-fittingseries}.

We fix the following notation:  $G$ is a group that is an extension $1 \to A \to G \to Q \to 1$, where $A$ is abelian and torsion-free.  For $r \geq 1$, we denote by $A_r$ the intersection of all subgroups of $A$ of index at most $r$. The subgroup $A_r$ contains $(r!)$-powers of all elements of $A$, and hence is non-trivial. Note also that $A_r$ is a characteristic subgroups of $A$. In particular $A_r$ is $Q$-invariant.

\begin{lem} \label{lem-QHr}
	Retain the above notation.
		\begin{enumerate}[label=\roman*)]
		\item If $H$ is a subgroup of $G$, then \[Q_{H,r} = \left\lbrace q \in Q : (q-1) A_r \leq H \right\rbrace \] is a subgroup of $Q$.
		\item The map $\sub(G) \to \sub(Q)$, $H \mapsto Q_{H,r}$, is equivariant, where $G$ acts on  $\sub(Q)$ by conjugation via the projection $G \to Q$. 
	\end{enumerate}

\end{lem}

\begin{proof}
The fact that $Q_{H,r}$ is a subgroup follows from the fact that $A_r$ is $Q$-invariant and the equality \[ (q_1 q_2^{-1} - 1) a = (q_1  - 1) q_2^{-1} a - (q_2  - 1) q_2^{-1} a, \] which holds for all $q_1, q_2 \in Q$ and $a \in A$. 

To see that $H \mapsto Q_{H,r}$ is equivariant, fix $g \in G$ and take $q \in Q_{gHg^{-1},r}$. Using module notation, that $(q-1) A_r$ lies in $gHg^{-1}$ can be rewritten as $\pi(g)^{-1} (q-1) A_r \leq H$. Now we have \[ \pi(g)^{-1} (q-1) A_r  = (\pi(g)^{-1} q \pi(g)-1) A_r \] since $A_r$ is $Q$-invariant, so we deduce $\pi(g)^{-1} q \pi(g) \in Q_{H,r}$, or equivalently $q \in \pi(g) Q_{H,r} \pi(g)^{-1}$. This shows $Q_{gHg^{-1},r} \leq \pi(g) Q_{H,r} \pi(g)^{-1}$, and the same argument also shows that equality holds.
\end{proof}

\begin{prop} \label{prop-lift-confine} Retain the notation from Lemma \ref{lem-QHr}.
Suppose that the group $G$ is an extension $1 \to A \to G \to Q \to 1$, where $A$ is abelian, torsion-free, and such that $A = C_G(A)$. Suppose $P$ is a finite subset of $Q \setminus \left\lbrace 1\right\rbrace $, and $P_G$ is a finite subset of $G$ of cardinality $r$ such that $\pi(P_G) = P$, where $\pi$ is the projection $G \to Q$. Then: \begin{enumerate}[label=\roman*)]
	\item $Q_{H,r} \in S_Q(P,Q)$ for all $H \in S_G(P_G,G)$.
	\item If \[ \bigcap_{L \in S_Q(P,Q)} \!  \!  \!  \!  \!  L \neq \left\lbrace 1 \right\rbrace, \] then \[ \bigcap_{H \in S_G(P_G,G)} \!  \!  \!  \!  \!  H \neq \left\lbrace 1 \right\rbrace . \] 
\end{enumerate}
\end{prop}

\begin{proof}
By Proposition \ref{p-Neumann-abelien}, for every $H$ in $S_G(P_G,G)$ there exists $g \in P_G$ such that $(\pi(g)-1)A_r \leq H$. This means that $\pi(g) \in Q_{H,r}$. So we have $Q_{H,r} \cap P \neq \emptyset$ for all $H \in S_G(P_G,G)$. Since $H \mapsto Q_{H,r}$ is equivariant, this implies that  $Q_{H,r} \in S_Q(P,Q)$ for all $H \in S_G(P_G,G)$. This proves the first statement. For the second, the assumption implies that there exists a non-trivial element $q$ such that $q \in Q_{H,r}$ for all $H \in S_G(P_G,G)$. This means that $(q-1)A_r \leq H$ for all $H \in S_G(P_G,G)$. Since $q$ is non-trivial, $(q-1)A$ is non-zero according to the assumption that $A = C_G(A)$. Since $A$ has no torsion by assumption, $(q-1)A_r$ is also non-trivial. Since $(q-1)A_r$ is contained in $\bigcap_{S_G(P_G,G)} H$, we obtain the conclusion.
\end{proof}

\begin{prop} \label{prop-lift-nice-subsets}
	Suppose that the group $G$ is an extension $1 \to A \to G \to Q \to 1$, where $A$ is abelian, torsion-free, and such that $A = C_G(A)$. Suppose that  $\mathcal{L} \subset Q$ is a non-foldable subset of $Q$. Let $\mathcal{L}_G \subset G$ be a lift of $\mathcal{L}$. Then $\mathcal{L}_G $ is a non-foldable subset of $G$.
\end{prop}

\begin{proof}
This follows by combining Proposition \ref{prop-lift-confine} with Lemma \ref{l-non-foldable-confined}. \qedhere
\end{proof}

\subsection{Torsion-free Fitting series} \label{subsec-fittingseries}

Observe that if the successive quotients $N_{i+1}/N_i$ in a series $\{1\}=N_0\unlhd N_1\unlhd \cdots \unlhd N_k=G$ are all torsion-free, then the group $G$ is evidently torsion-free. In general the converse does not hold, in the sense that some torsion-free solvable groups do not admit a series with torsion-free abelian factors. The goal of this paragraph is to prove Conjecture \ref{conj-torsion-free} under the assumption that the factors in the Fitting series of $G$ are torsion-free. Recall that the Fitting series of a solvable group $G$ is defined inductively by $F_0(G) = \left\lbrace 1\right\rbrace $ and $F_{i+1}(G) / F_i(G) = \Fit(G/F_i(G))$. So $F_1(G) = \Fit(G)$, and the series eventually reaches $G$ since $G$ is solvable. The Fitting length of $G$ is the least $\ell$ such that $F_\ell(G) = G$. 

%We will show the following. 

\begin{thm} \label{thm-fitting-torsionfree}
Let $G$ be a finitely generated solvable group, and assume that all successive quotients $F_{i+1}(G) / F_i(G)$ in the Fitting series of $G$ are torsion-free. If $G$ is not virtually abelian, then $G$ contains a non-foldable pair. In particular Conjecture \ref{conj-torsion-free} is true for $G$.
\end{thm}

To prove Theorem \ref{thm-fitting-torsionfree} we need some further preliminaries. We will use  the following result of Lennox (see \cite[2.3.14]{Lennox-Rob}).

\begin{prop} \label{prop-Lennox-isolator}
	Let $G$ be a finitely generated solvable group, and $H$ a subgroup of $G$ such that for every $g \in G$ there is $n \geq 1$ such that $g^n \in H$. Then $H$ has finite index in $G$.
\end{prop}

 Recall that by Corollary \ref{c-metabelian-subgroup}, in order to prove Conjecture \ref{conj-torsion-free}, there is no loss of generality in restricting to solvable groups whose finitely generated metabelian subgroups are all virtually abelian. It is convenient to introduce the following ad-hoc terminology.

 \begin{defin}
 We say that a  solvable group $G$ is \textbf{restricted} if every finitely generated metabelian subgroup of $G$ is virtually abelian.  
 \end{defin} 

\begin{prop} \label{prop-no-metab-sbgp-restriction}
Let $G$ be a finitely generated restricted solvable group. Suppose  $\Fit(G)$ is torsion-free. Then the following hold: \begin{enumerate}[label=\roman*)]
\item  $\Fit(G)$ is abelian and contained in the FC-center of $G$. 
\item If $\Fit(G)$ is finitely generated, then $G$ is virtually abelian. 
\end{enumerate}
\end{prop}

\begin{proof}
By Lemma \ref{l-subgroup-H3}, every nilpotent subgroup of $\Fit(G)$ is virtually abelian, and hence abelian by Lemma \ref{lem-nilp-virtab-tf}. Now if $N_1,N_2$ are two abelian normal subgroups of $G$, then $N_1N_2$ is nilpotent by Fitting's theorem, and hence abelian. So all abelian normal subgroups of $G$ commute, and hence $\Fit(G)$ is abelian.

Now let $a \in \Fit(G)$, and let $H$ be the centralizer of $a$ in $G$. If there is $g \in G$ such that $g^n \notin H$ for all $n \geq 1$, then $a$ and $g$ generate a non-virtually abelian subgroup of $G$, that is metabelian because $\Fit(G)$ is abelian. This is impossible by our assumption. So for all $g \in G$ there is $n \geq 1$ such that $g^n \in H$, and hence by Proposition \ref{prop-Lennox-isolator} the subgroup $H$ has finite index in $G$. So $a$ has a finite conjugacy class, and $\Fit(G)$ is indeed contained in the FC-center of $G$. 

Finally if $\Fit(G)$ is finitely generated, then it follows from the previous paragraph that $C_G(\Fit(G))$ is a finite index subgroup of $G$. But for $G$ is solvable, $C_G(\Fit(G))$ is always contained in $\Fit(G)$  \cite[1.2.10]{Lennox-Rob}. So $\Fit(G)$ has finite index in $G$, and $G$ is virtually abelian.
\end{proof}

\begin{prop} \label{prop-lift-fitting}
	Let $G$ be a finitely generated solvable group such that $\Fit(G)$ is torsion-free. If $G / \Fit(G)$ admits a non-foldable pair, then so does $G$.
\end{prop}

\begin{proof}
	Note that the assumption that $G / \Fit(G)$ admits a non-foldable pair implies that $G$ is not virtually abelian. If $G$ is not restricted, then we can invoke Theorem \ref{thm-metab-notorsion}. If $G$ is restricted, then $A = \Fit(G)$ is abelian by Proposition \ref{prop-no-metab-sbgp-restriction}, and we have $A = C_G(A)$ \cite[1.2.10]{Lennox-Rob}. Hence we may apply Proposition \ref{prop-lift-nice-subsets}, which implies that any pair $(g_1,g_2)$ that is a lift of a non-foldable pair $(q_1,q_2)$ of $Q$, is a non-foldable pair of $G$.
\end{proof}

\begin{proof}[Proof of Theorem \ref{thm-fitting-torsionfree}]
Recall that when $G$ is not restricted, the conclusion follows from Theorem \ref{thm-metab-notorsion}. We argue by induction on the Fitting length $\ell$. Since $G$ is finitely generated, the case $\ell=1$ corresponds to the case where $G$ is nilpotent. Since $G$ is not virtually abelian, Proposition \ref{prop-no-metab-sbgp-restriction} says that in that case $G$ cannot be restricted, and hence the statement is true. Suppose now that $G$ has length $\ell+1$, $\ell \geq 1$.  According to Proposition \ref{prop-no-metab-sbgp-restriction}, we may assume that $\Fit(G)$ is abelian. If $Q = G/\Fit(G)$ is virtually abelian, then $G$ is virtually metabelian, and again the conclusion holds. Now if $Q$ is not virtually abelian, then $Q$ satisfies all the assumptions of the theorem. Hence in that case we can apply the induction hypothesis to $Q$, and the conclusion then follows from  Proposition  \ref{prop-lift-fitting}. 
\end{proof}

\subsection{Quasi-transitive actions}

Here we prove Conjecture \ref{conj-torsion-free} in the special case of transitive actions. Actually the same argument extends without much effort to quasi-transitive actions. In the sequel we call a $G$-set $X$ \textbf{quasi-transitive} if the action of $G$ on $X$ has finitely many orbits. 

 \begin{thm} \label{t-torsion-free-quasitransitive}
	Let $G$ be a finitely generated torsion-free solvable groups, not virtually abelian. Let $X$ be a faithful quasi-transitive $G$-set. Then $\vol_{G, X}(n) \succcurlyeq n^2$. 
\end{thm}

The proof requires some additional ingredients. 

\begin{defin}
A subgroup $H$ of a group $G$ is \textbf{absorbing} if for every $g \in G$, there is $n \geq 1$ such that $g^n \in H$.
\end{defin}

In the sequel, given  $f, g\colon \N\to \N$ we write $f(n)\nsucccurlyeq g(n)$ for the negation of $f(n) \succcurlyeq  g(n)$. Thus $f(n)\nsucccurlyeq g(n)$ if and only if  $\liminf_{n\to \infty} \frac{f(n)}{g(n)} =0$.

\begin{prop} \label{prop-subquadratic-robinson}
Let $G$ be a finitely generated group, and $N$ a normal subgroup of $G$ such that $Q = G/N$ is solvable, and denote $p_Q$ the projection of $G$ to $Q$. Let $X$ a $G$-set such that $\vol_{G, X}(n)\nsucccurlyeq n^2$. Then for every $x \in X$, at least one of the following hold:
\begin{enumerate}[label=\roman*)]
	\item $N \cap G_x$ is absorbing in $N$.
	\item $p_Q(G_x)$ has finite index in $Q$;

\end{enumerate}
\end{prop}

\begin{proof}
Fix $x \in X$, and suppose that $N \cap G_x$ is not absorbing in $N$. By definition this means that there exists $g_1 \in N$ such that $g_1^n \notin G_x$ for every non-zero integer $n$. We want to show that $p_Q(G_x)$ has finite index in $Q$. According to Proposition \ref{prop-Lennox-isolator}, it suffices to prove that $p_Q(G_x)$ is absorbing in $Q$ since $Q$ is a finitely generated solvable group. So we shall prove that for every $q \in Q$, there are $n \geq 1$ and $g \in G_x$ such that $p_Q(g) = q^n$. Clearly it suffices to treat the case where $q$ has infinite order. Let $g_2 \in G$ such that $p_Q(g_2)=q$. The subset $\mathcal{L} = \left\lbrace g_2^i g_1^j \, : \, i,j \in \Z \right\rbrace$ satisfies $\relvol_{(G, \mathcal{L})}(n) \succcurlyeq n^2$ since $g_1 \in N$ and $p_Q(g_2)=q$ has infinite order in $Q$. Since $X$ satisfies $\vol_{G, X}(n)\nsucccurlyeq n^2$ by assumption, by Lemma \ref{lem-exp-subset-growth} it cannot be the case that $\mathcal{L}$ is non-folded in $X$. Hence there exist $(i,j) \neq (k,\ell)$ such that $G_x$ contains $g = (g_2^i g_1^j)^{-1} g_2^k g_1^\ell = g_1^{-j} g_2^{-i} g_2^k g_1^\ell$. If $i = k$, then necessarily $j \neq \ell$, and hence $G_x$ contains a non-trivial power of $g_1$, which is impossible. So $n = i - k$ is non-zero. Since $g \in G_x$ and $p_Q(g) = p_Q(g_2^n)$, we have obtained the desired conclusion. 
\end{proof}

\begin{lem} \label{lem-chab-closure-FC-central}
Let $G$ be a group, and $N$ be a normal subgroup of $G$ contained in the FC-center of $G$. Let $H,K \in \sub(G)$ such that there is a sequence of conjugates of $K$ that converges to $H$ in $\sub(G)$. If $N \leq H$, then $N \leq K$. 
\end{lem}

\begin{proof}
Let $g \in N$, and let $g^G$ be the conjugacy class of $g$. Since  $N$ is normal and $N \leq H$, $g^G$ is contained in $H$. Since $g^G$ is finite, it follows that the set of subgroups of $G$ containing $g^G$ forms an open neighbourhood of $H$ in $\sub(G)$. Hence it follows from the assumption that there is a conjugate of $K$ that contains $g^G$, which is equivalent to saying that $K$ contains $g^G$. Since $g$ was arbitrary, we have  $N \leq K$. 
\end{proof}

\begin{prop} \label{prop-restricted-transitive}
	Suppose $G$ is a finitely generated solvable group, and $A$ is an abelian normal subgroup of $G$ contained in the FC-center of $G$.  Let $Q = G/A$. Let $X$ be a transitive $G$-set such that $\vol_{G, X}(n)\nsucccurlyeq n^2$. Then at least one of the following holds:\begin{enumerate}[label=\roman*)]
		\item there exists a normal subgroup $N$ of $G$ such that $N \leq A$, $A/N$ is a torsion group, and $N$ acts trivially on $X$;
		\item there exists a finite index subgroup $L$ of $Q$ such that for all $q \in L$, there is $r \geq 1$ such that $r (q-1) A \leq G_x$ for all $x \in X$.
	\end{enumerate}
\end{prop}

\begin{proof}
Note that since $A$ is abelian, a subgroup $B$ of $A$ is absorbing if and only if $A/B$ is a torsion group. Let \[ \mathcal{S}(X) = \overline{ \left\lbrace G_x  \, : \, x \in X \right\rbrace } \subseteq \sub(G). \] By Lemma \ref{lem-growth-closure}, we have that $\vol_{G, G/H}(n)\nsucccurlyeq n^2$ for every $H \in \mathcal{S}(X) $. Hence we are in position to apply Proposition  \ref{prop-subquadratic-robinson} to the action of $G$ on $G/H$, with $N = A$. Suppose that there exists $H \in \mathcal{S}(X) $ such that the first alternative of Proposition  \ref{prop-subquadratic-robinson} holds, that is $A \cap H$ is absorbing in $A$. Then for every finitely generated subgroup $M$ of $A$ that is normal in $G$, there exists a finite index subgroup $M'$ of $M$ that is normal in $G$ and such that $M' \leq H$. Now since $A$ is contained in the FC-center of $G$, the finitely generated subgroups $M$ of $A$ that are normal in $G$ exhaust $A$. Hence it follows that there is a normal subgroup $N$ of $G$ such that $N$ is absorbing in $A$ and $N \leq H$. Fix $x \in X$. Since the $G$-action on $X$ is transitive, $\mathcal{S}(X)$ is equal to the closure of the $G$-orbit of $G_x$ in $\sub(G)$. Since $N$ is contained in the FC-center of $G$, $H \in \mathcal{S}(X)$ and $N \leq H$, Lemma \ref{lem-chab-closure-FC-central} says that $N$ must be contained in $G_x$. It follows that $N$ acts trivially on $X$, and hence the first conclusion holds in this case. 

Hence in the sequel we assume that for every $H \in \mathcal{S}(X) $, the second alternative of Proposition  \ref{prop-subquadratic-robinson} holds, that is $p_Q(H)$ has finite index in $Q$, where $Q = G/A$. Observe that for $H \in \mathcal{S}(X) $, since  $p_Q(H)$ is finitely generated, there exists an open neighbourhood $\mathcal{U}$ of $H$ in $\mathcal{S}(X) $ such that $p_Q(K) \geq p_Q(H)$ for every $K \in \mathcal{U}$. Since $\mathcal{S}(X) $ is a compact space, $\mathcal{S}(X) $ can be covered by finitely many of these open sets, and hence that there is a finite index subgroup $L$ of $Q$ such that $p_Q(H)$ contains $L$ for every $H \in \mathcal{S}(X) $. Fix a non-trivial element $q \in L$. We have that for every $H\in \mathcal{S}(X)$ there exists $h \in H$ such that $p_Q(h) = q$. Since the condition of containing $h$ is an open condition, using compactness again we obtain a finite subset $P = \left\lbrace h_1,\ldots,h_r\right\rbrace $ such that $p_Q(h_i) = q$ for all $i$ and every $H \in \mathcal{S}(X) $ intersects $P$. So $P$ is confining for $(H,G)$ for all $H \in \mathcal{S}(X) $. So the conclusion follows from Proposition \ref{p-Neumann-abelien}. 
\end{proof}

\begin{proof}[Proof of Theorem \ref{t-torsion-free-quasitransitive}]
The case where the group is metabelian has already been treated in Theorem \ref{thm-metab-notorsion}, regardless of the number of orbits. This implies that if $G$ is not restricted, then the statement holds true. So in the sequel we assume that $G$ is restricted. By Proposition \ref{prop-no-metab-sbgp-restriction} the subgroup $A = \Fit(G)$ is abelian, and $A$ is contained in the FC-center of $G$. Let $Q = G/A$.
	
Suppose for a contradiction that $\vol_{G, X}(n)\nsucccurlyeq n^2$. Let $X_1,\ldots,X_m$ be the $G$-orbits in $X$. We have $\vol_{G, X_i}(n)\nsucccurlyeq n^2$ for all $i$ (Proposition \ref{prop-monoton}), so we can apply Proposition \ref{prop-restricted-transitive} to each $X_i$. Let $X_1,\ldots,X_\ell$ be the components that satisfy the first conclusion of the proposition.  So for all $i \in \left\lbrace 1,\ldots,\ell \right\rbrace $ there exists a normal subgroup $N_i$ of $G$ that is an absorbing subgroup of $A$ and such that $N_i$ acts trivially on $X_i$. For $j \in \left\lbrace \ell + 1,\ldots, m \right\rbrace $, there is a finite index subgroup $L_i$ of $Q$ such that for every $q \in L_i$, there is $r_j \geq 1$ such that $r_j(q-1) A$ acts trivially on $X_j$. Note that $M_1 := \bigcap N_i$ is also an absorbing subgroup of $A$, and $M$ acts trivially on $\cup_{i \leq \ell} X_i$. Moreover since $G$ is not virtually abelian, the group $Q$ is infinite, and hence $L := \bigcap L_j$ is non-trivial. Fix a non-trivial element $q$ of $L$. Then we can find $r \geq 1$ such that $r (q-1) A$ acts trivially on $\cup_{j > \ell} X_j$. Since $A$ is equal to its own centralizer in $G$ \cite[1.2.10]{Lennox-Rob} and $q$ is non-trivial, we have that $(q-1) A$ is non-trivial. Since $A$ is torsion-free, it follows that $r (q-1) A$  is also non-trivial. Hence if $M_2$ is the normal subgroup of $G$ generated by $r (q-1) A$, then $M_2$ is non-trivial and $M_2$ acts trivially on $\cup_{j > \ell} X_j$. The subgroup $M_1 \cap M_2$ is an absorbing subgroup of $M_2$, and hence is non-trivial. Since $M_1 \cap M_2$ acts trivially on $X$, we have reached the desired contradiction. 
\end{proof}

\subsection{Higman-type extensions}

%\subsubsection{Free solvable groups: orbital growth}

In this paragraph we focus on the following specific class of groups:

\begin{defin}
We call a group $G$ a \textbf{Higman type extension} if there is a finitely generated free group $F_d$ and a normal subgroup $N$ of $F_d$ such that $G = F_d/N'$. 
\end{defin}

%In this paragraph we let $F$ be a finitely generated free group, $N$ a normal subgroup of $F$, and consider groups of the form $G = F/N'$. 

%Equivalently, $G$ is the largest quotient of $F_d$ in which the image of $N$ is abelian, and 

So $G$ is an extension of the group $Q=F_d/N$ by the abelian normal subgroup $N/N'$. Our choice of terminology comes from the article \cite{Higman-torsionfree}, where Higman showed that the group $G$ is always torsion-free. 

%\cite[Th. 2]{Higman-torsionfree}

%	Let $F$ be a finitely generated free group, and let $N$ be a normal subgroup of $F$. Let $Q = F/N$ and 

\begin{prop} \label{prop-F/N'}
Let $G = F_d/N'$ be a Higman type extension. Let $L$ be a  torsion-free abelian subgroup of group $Q=F_d/N$, and let $\mathcal{L}$ be a lift of $L$ to $G$. Then the following hold:
	\begin{enumerate} [label=\roman*)]
		\item $\mathcal{L}$ is a non-foldable subset of $G$;
		\item $\vol_{G, X}(n) \succcurlyeq f_{Q,L}(n)$ for every faithful $G$-set $X$.
	\end{enumerate}
\end{prop}

\begin{proof}
Let $\Sigma$ be a finite subset of $\mathcal{L}$, and $P =\{g^{-1}h\colon g, h\in P, g\neq h\}$. Note that ell elements of $P$ have a non-trivial projection to $Q$. We shall prove that there is a non-trivial normal subgroup of $G$ that is contained in $H$ for every $H \in S_G(P,N/N')$. By Lemma \ref{l-non-foldable-confined}, this will prove that $\mathcal{L}$ is a non-foldable subset of $G$. 

The subgroup $M = N/N'$ is a torsion-free abelian normal subgroup of $G$. Let $r$ be the cardinality of $P$, and let $M_r = (r!) M$. According to Proposition \ref{p-Neumann-abelien}, for every $H \in S_G(P,M)$ there exists $b \in P$ such that $(\pi_Q(b)-1) M_r \leq H$. Since $\pi_Q(b)$ belongs to $L$ and $L$ is abelian, it follows that if we set $q =  \prod_P (\pi_Q(b)-1)$, then $q M_r$ is contained in $H$ for every $H \in S_G(P,M)$. Now since $L$ is torsion-free abelian, the group ring $\Z L$ has no zero divisors. Since $\pi_Q(b)$ is non-trivial for every $b \in P$, it follows that $q$ is non-zero in $\Z L \subseteq \Z Q$. Moreover according to \cite{Passi-annihilators}, $M = N/N'$ is a faithful $\Z Q$-module. Hence $qM \neq \left\lbrace  0\right\rbrace $, and since $M$ is torsion-free we also have $qM_r \neq \left\lbrace  0\right\rbrace $. Since $q M_r \leq H$ for every $H \in S_G(P,M)$, this terminates the proof of the first statement. 

By considering for every element of $L$ a preimage in $G$ of minimal word length, we see that one can choose a lift $\mathcal{L}$ of $L$ such that $f_{Q,L}(n) \simeq \relvol_{(G, \mathcal{L})}(n)$. Hence the second statement follows from the first together with Lemma  \ref{lem-exp-subset-growth}.
\end{proof}

\begin{remark}
The result from \cite{Passi-annihilators} used in the above proof shows that if $Q$ is not a torsion group, then $G$ always contains a copy of $\Z \wr \Z$. In particular solvable Higman-type extensions satisfy Conjecture \ref{conj-torsion-free}. 
\end{remark}

The previous result is most useful when the group $Q$ is not too small. We illustrate this for free solvable groups (compare with Theorem \ref{thm-freemetab-growth}):

\begin{thm}\label{thm-free-solvable}
	Let $G = \mathbb{FS}_{d,\ell}$ be the free solvable group  of rank $d \geq 2$ and length $\ell \geq 3$. Then $G$ has a Schreier growth gap $\exp(n)$.
\end{thm}

In the proof we will use the following easy lemma.

%which we believe is well-known. We include a proof for completeness. 

\begin{lem} \label{lem-Free-sol-exp-growth}
Let $Q = \mathbb{FS}_{d,\ell}$ be the free solvable group  of rank $d \geq 2$ and length $\ell \geq 2$, and let $L = Q^{(\ell-1)}$. Then $\relvol_{(Q,L)}(n) \simeq \exp(n)$.
\end{lem}

\begin{proof}
Let $x_1,\ldots, x_d$ be the images in $Q$ of free generators of $F_d$, and set $c = w_{\ell-1} \in L$, where $w_{i}$ is defined inductively by $w_{1}=[x_1,x_2]$ and $w_{i+1} = [x_1,w_i]$. For $i\geq 0$ we also set $c_i = x_1^i c x_1^{-i}$. The map from $\left\lbrace 0,1\right\rbrace ^n$ to $L$ that maps $(\varepsilon_1,\ldots,\varepsilon_n)$ to $w(\varepsilon_1,\ldots,\varepsilon_n) = c_1^{\varepsilon_1} \cdots c_n^{\varepsilon_n}$ is injective, and $w(\varepsilon_1,\ldots,\varepsilon_n)$ has word  length at most $(\left|c \right|  +2)n$ with respect to $x_1,\ldots, x_d$. Hence the ball of radius $(\left|c \right|  +2)n$ contains at least $2^n$ elements of $L$, and it follows that $L$ indeed has relative exponential growth in $Q$.
\end{proof}

\begin{proof}[Proof of Theorem \ref{thm-free-solvable}]
We have $G = F_d / N'$, where $N = F_d^{(\ell-1)}$, and $Q = F_d / N \simeq \mathbb{FS}_{d,\ell-1}$. The subgroup $L = Q^{(\ell-2)}$ is a torsion-free abelian subgroup of $Q$, and $f_{Q,L}(n) \simeq \exp(n)$ by	Lemma \ref{lem-Free-sol-exp-growth}.  Therefore the conclusion follows from Proposition \ref{prop-F/N'}. 
\end{proof}

\section{Additional comments and examples}

\subsection{Pointed growth of transitive actions}

When $X$ is a transitive $G$-set, there is another natural notion of growth, called the pointed growth of the action. It is defined by fixing a point $x \in X$, and considering $\vol_{G, X,x}(n)= |S^n \cdot x|$. By transitivity of the action, $\vol_{G, X,x}(n)$ is easily seen to be independent of the choice of $x$ up to $\simeq$. Clearly $\vol_{G, X}(n) \succcurlyeq \vol_{G, X,x}(n)$, and in general $\vol_{G, X}(n)$ is strictly larger than $\vol_{G, X,x}(n)$ (see the examples below). Here we would like to emphasize that in our setting the growth function $\vol_{G, X}(n)$ is more natural and better suited for our purposes than the pointed growth. The first advantage of considering $\vol_{G, X}(n)$ rather than the pointed growth is that the latter only makes sense for transitive actions. Restricting ourselves to transitive actions is not desirable here, for instance because it is not stable when passing to a subgroup. On the contrary $\vol_{G, X}(n)$ remains defined when passing to a subgroup and is monotone under this operation (Propositions  \ref{prop-monoton} and \ref{prop-finite-index}). Another advantage of $\vol_{G, X}(n)$ is the fact that it behaves nicely with respect to the topology on $\sub(G)$. One illustration of this is that our notion of non-foldable subsets, which by definition provides lower bounds for $\vol_{G, X}(n)$ (Lemma  \ref{lem-exp-subset-growth}), has a simple reinterpretation in terms of confined subgroups in $\sub(G)$ (Lemma \ref{l-non-foldable-confined}). Another illustration is given by Lemma \ref{lem-growth-closure}, which says that $\vol_{G, X}(n)$ decreases when taking a limit point in $\sub(G)$. Here we also note that this lemma  fails if one replaces $\vol_{G, X}(n)$ by the pointed growth (see Example \ref{ex-point-heis}).

The following simple examples also illustrate that the lower bounds obtained for $\vol_{G, X}(n)$ fail in general for the pointed growth. 

\begin{example} 
	Let $G = \Z \wr \Z  $. Recall that $\vol_{G, X}(n) \succcurlyeq  n^2$ for every faithful $G$-set $X$ by Theorem \ref{thm-wreath}. We claim that this bound fails for the pointed growth. Consider a pair of generators $s,t$, where $t$ is a generator of the base group $\Z$, and $s$ is a generator of lamp group at position $0$, and for $i\in \Z$ set $s_i=t^ist^{-i}$.  Fix an increasing sequence of integers $d_k \geq 1$, and let $H$ be the subgroup of $G$ contained in  $\oplus \Z$ consisting of all configurations $f \in \oplus \Z$ such that $f(d_k)\in 2^k\Z$.   The action of $G$ on $X:=G/H$ is easily seen to be faithful. In particular $\vol_{G, X}(n) \succcurlyeq  n^2$. By contrast, we claim that if $(d_k)$ grows sufficiently fast, then the pointed growth of the action of $G$ on $X$, henceforth denoted $v(n)$, can be made strictly subquadratic. To see this, let $k(n):=\max\{k \colon d_k\le n\}$. Let $g\in G$ be an element of word length $|g|\le n$. Then  $g$  can be expressed as  $g=t^\ell \prod_{i=-n}^ns_i^{m_i}$, with $\ell \leq n$ and $m_i\leq n$. Note that for each $i$, we have $s_i^{m_i}H=H$ if $i\neq d_k, k\ge 1$, and for every $i$ of the form $i=d_k$ for some $k$ we have $s_{d_k}^{m_{d_k}}H=s_{d_k}^{r_k}H$, where $r_i$ is the remainder of the euclidean division of $m_{d_k}$ by $2^k$.  Since moreover the $s_i$ commute,  we have $gH = t^\ell s_{d_1}^{r_1}\cdots s_{d_{k(n)}}^{r_{k(n)}}H$. Counting possibilities for $\ell$ and $r_i$ with $ i\le k(n)$, we obtain the upper bound \[v(n)\preccurlyeq n\prod_{i=1}^{k(n)} 2^i \preccurlyeq  n2^{k(n)^2}.\]  For a fixed $\alpha >1$, we choose  $d_k=2^{\alpha k^2}$ for all $k$. The above computation yields $v(n) \preccurlyeq n^{1+\frac{1}{\alpha}}$, and hence $v(n) \nsucccurlyeq n^2$. \end{example}

	The following example was suggested by Yves Cornulier.

\begin{example} \label{ex-point-heis}
 Consider the Heisenberg group $G=H_3(\Z) = \langle a, b, c\colon [a,b]=c, [a,c]=[b, c]=1\rangle$. Recall that every $g\in G$ admits a unique decomposition $g=c^zb^ya^x$, with $x, y, z\in \Z$. The word metric is given up to constants by $|g|\simeq \max (|x|, |y|, \sqrt{|z|})$, and as a consequence the volume growth of $G$ is $\vol_G(n)\simeq n^4$ (see e.g.  \cite[\S 14.1.1]{Drutu-Kapovich}). Recall from Proposition \ref{prop-n4-growth} that the group $G$  is non-foldable, and in particular every faithful $G$-set $X$ satisfies $\vol_{G, X}(n)\simeq \vol_G(n)\simeq n^4$. Now let $H$ be the cyclic subgroup generated by $a$.  For every $g=c^zb^ya^x$, the coset  $gH$ coincides with $c^zb^yH$. If $|g|\leq n$ we have $|y|\leq Cn, |z|\leq Cn^2$ for some $C>0$, so that there are $\simeq n^3$ possibilities for the pair $(y, z)$. Hence the pointed growth of the $G$-action on $G/H$ is $\simeq n^3$. Note that since $H$ is not confined, this example also shows that Lemma \ref{lem-growth-closure} fails if one replaces $\vol_{G, X}$ by the pointed growth. 
 	\end{example}

\subsection{Connection with groups of dynamical origin} \label{s-non-embeddings}

As mentioned in the introduction,  Schreier growth gaps provide an obstruction to the existence of embeddings between groups, which has applications to the study of topological full groups. Recall that given an action $G \acts \mathfrak{C}$ by homeomorphisms  on the Cantor set, the topological full group of the action is the group $F(G, \mathfrak{C})$ of all homeomorphisms of $\mathfrak{C}$ that locally coincide with elements of $G$. This notion was introduced (for $G=\Z$) by Giordano, Putnam and Skau, and further developed by Matui, and has been studied extensively in recent years \cite{Ju-Mo,Nek-simple-dyn, Nek-frag}.  This construction also encompasses many more specific groups studied in the litterature, such as Thompson's groups $V$ and some generalisations, and subgroups of the group $\operatorname{IET}$ of interval exchanges. 

Despite various advances, it remains quite mysterious how the properties of the $G$-action on $ \mathfrak{C}$ constraint the group $F(G, \mathfrak{C})$ and its possible subgroups. Non-foldable subsets and graphs of actions naturally fit  in this setting due to the following basic and well-known remark.

\begin{prop}
Let $G$ be a finitely generated group and $G\acts \mathfrak{C}$ be an action on the Cantor set. Then for every finitely generated subgroup $H\le F(G, \mathfrak{C})$, the identity map  $\mathfrak{C} \to \mathfrak{C}$ defines a Lipschitz embedding of the graph $\Gamma(H, \mathfrak{C})$ inside $\Gamma(G, \mathfrak{C})$. In particular, we have $\vol_{H, \mathfrak{C}}(n)\preceq \vol_{G, \mathfrak{C}}(n)$ and $\asdim(H, \mathfrak{C})\le \asdim (G, \mathfrak{C})$,
\end{prop}

Hence our results immediately provide restrictions on the nature of the solvable subgroups of  the topological full groups of various actions. Without being exhaustive, we point out some simple illustrations.

%\begin{example}

\subsubsection{Topological full groups of $\Z^d$-actions} \label{subsubsection-full-gp}
Matui showed that for every minimal action of $\Z$ on  $\mathfrak{C}$ that is not an odometer action, the topological full group $F(\Z, \mathfrak{C})$ contains a copy of the lamplighter group $C_2 \wr \Z$ \cite{Mat-exp}. By contrast, Theorem \ref{thm-wreath} shows that neither $\Z \wr \Z$ nor $C_p \wr \Z^{2}$ embed in $F(\Z, \mathfrak{C})$. The case of $C_p \wr \Z^{2}$ solves Conjecture 2 from \cite{Salo}. More generally for every action of $\Z^d$ on $\mathfrak{C}$, the groups $\Z \wr \Z^d$ and $C_p \wr \Z^{d+1}$ do not embed in $F(\Z^d, \mathfrak{C})$.
%\end{example}

\subsubsection{ Brin--Thompson groups}
Consider the generalizations $nV$ of Thompson group $V$ from \cite{Brin-nV}. It is not difficult to see that for the natural action of $nV$ on the Cantor set $\mathfrak{C}$, the graph of the action on each orbit is quasi-isometric to a product of $n$ trees (see e.g.  \cite[Lemma 11.10]{MB-graph-germs}). Thus $\Gamma(nV,\mathfrak{C})$ has asymptotic dimension $n$. Hence Theorem \ref{thm-wreath} implies:

\begin{cor}
The group $nV$ does contain a wreath product $A \wr \Z^{n+1}$, with $A\neq \{1\}$. 
\end{cor} 

By contrast it is not difficult to see that $C_2 \wr \Z^n$ embeds in $nV$. Hence this gives another proof of the result  \cite[Corollary 11.20]{MB-graph-germs} that $nV$ embeds in $mV$ only if $n\le m$.

In the case $n=2$, Corollary  \ref{cor-poly-asdimX} also immediately implies the following (that problem was notably raised in \cite{MO-zarem}):

\begin{cor} \label{cor-nV-poly}
Every polycyclic subgroup of $2V$ is virtually abelian.
\end{cor}

%\end{example}

\subsubsection{Interval exchanges}

%\begin{example}
Dahmani, Fujiwara and Guirardel studied solvable subgroups of the group  $\operatorname{IET}$ of interval exchanges in \cite{DFG-sol}. It is well known that the natural action of a finitely generated subgroup of $\operatorname{IET}$ has polynomial growth. One result of  \cite{DFG-sol} is that every finitely generated torsion-free solvable subgroup of $\operatorname{IET}$ is virtually abelian. They deduce in particular that this holds true for polycyclic subgroups \cite[Cor.\ 3.2]{DFG-sol}. The arguments in \S \ref{s-noetherian} provide a soft proof of this last result, since Corollary \ref{cor-poly-growth} implies that every polycyclic subgroup of $\operatorname{IET}$ is virtually nilpotent, and hence virtually abelian by \cite{Novak-disc-IET}.  However we stress that the absence of torsion-free solvable subgroups in $\operatorname{IET}$  \cite{DFG-sol} cannot be proven relying only on considerations on growth (as torsion-free solvable groups can admit actions of polynomial growth).
%\end{example}

\bibliographystyle{amsalpha}
\bibliography{bib-growth}

\providecommand{\bysame}{\leavevmode\hbox to3em{\hrulefill}\thinspace}
\providecommand{\MR}{\relax\ifhmode\unskip\space\fi MR }
% \MRhref is called by the amsart/book/proc definition of \MR.
\providecommand{\MRhref}[2]{%
  \href{http://www.ams.org/mathscinet-getitem?mr=#1}{#2}
}
\providecommand{\href}[2]{#2}
\begin{thebibliography}{LBMB20b}

\bibitem[Bas72]{Bass-nilp}
H.~Bass, \emph{The degree of polynomial growth of finitely generated nilpotent
  groups}, Proc. London Math. Soc. (3) \textbf{25} (1972), 603--614.

\bibitem[Bau71]{Baumslag-lecturenotes}
G.~Baumslag, \emph{Lecture notes on nilpotent groups}, Regional Conference
  Series in Mathematics, No. 2, American Mathematical Society, Providence,
  R.I., 1971.

\bibitem[Bau72]{Baumslag-group}
\bysame, \emph{A finitely presented metabelian group with a free abelian
  derived group of infinite rank}, Proc. Amer. Math. Soc. \textbf{35} (1972),
  61--62.

\bibitem[BD08]{BD-asdim}
G.~Bell and A.~Dranishnikov, \emph{Asymptotic dimension}, Topology Appl.
  \textbf{155} (2008), no.~12, 1265--1296.

\bibitem[BG00]{Bar-Gri-Hecke}
L.~Bartholdi and R.~I. Grigorchuk, \emph{On the spectrum of {H}ecke type
  operators related to some fractal groups}, Tr. Mat. Inst. Steklova
  \textbf{231} (2000), no.~Din. Sist., Avtom. i Beskon. Gruppy, 5--45.

\bibitem[BH21]{Bou-Houd}
R.~Boutonnet and C.~Houdayer, \emph{Stationary characters on lattices of
  semisimple {L}ie groups}, Publ. Math. Inst. Hautes \'{E}tudes Sci.
  \textbf{133} (2021), 1--46.

\bibitem[BNZ22]{BNZ}
L.~Bartholdi, V.~Nekrashevych, and T.~Zheng, \emph{Growth of groups with linear
  schreier graphs}, arXiv:2205.01792.

\bibitem[Bou61]{Bourbaki-alg-comm-3-4}
N.~Bourbaki, \emph{\'{E}l\'{e}ments de math\'{e}matique. {F}ascicule {XXVIII}.
  {A}lg\`ebre commutative. {C}hapitre 3: {G}raduations, filtrations et
  topologies. {C}hapitre 4: {I}d\'{e}aux premiers associ\'{e}s et
  d\'{e}composition primaire}, Actualit\'{e}s Scientifiques et Industrielles
  [Current Scientific and Industrial Topics], No. 1293, Hermann, Paris, 1961.

\bibitem[Bri04]{Brin-nV}
M.~Brin, \emph{Higher dimensional {T}hompson groups}, Geom. Dedicata
  \textbf{108} (2004), 163--192.

\bibitem[BS78]{BS78}
R.~Bieri and R.~Strebel, \emph{Almost finitely presented soluble groups},
  Comment. Math. Helv. \textbf{53} (1978), no.~2, 258--278.

\bibitem[BS80]{BS80}
\bysame, \emph{Valuations and finitely presented metabelian groups}, Proc.
  London Math. Soc. (3) \textbf{41} (1980), no.~3, 439--464.

\bibitem[BST12]{BTS-asdim}
I.~Benjamini, O.~Schramm, and A.~Tim\'{a}r, \emph{On the separation profile of
  infinite graphs}, Groups Geom. Dyn. \textbf{6} (2012), no.~4, 639--658.

\bibitem[CLB20]{CLB-commens}
P.-E. Caprace and A.~Le~Boudec, \emph{Commensurated subgroups and
  micro-supported actions}, to appear in JEMS (2020), With an appendix by D.
  Francoeur.

\bibitem[Cor11]{Cornulier-CBrank}
Y.~Cornulier, \emph{On the {Cantor-Bendixson} rank of metabelian groups},
  Annales de l'Institut Fourier \textbf{61} (2011), no.~2, 593--618.

\bibitem[Cor15]{Cor-MathZ}
\bysame, \emph{Irreducible lattices, invariant means, and commensurating
  actions}, Math. Z. \textbf{279} (2015), no.~1-2, 1--26.

\bibitem[CT20]{CT-Banach}
Y.~Cornulier and R.~Tessera, \emph{On the vanishing of reduced 1-cohomology for
  {Banach} representations}, Annales de l'Institut Fourier \textbf{70} (2020),
  no.~5, 1951--2003.

\bibitem[DFG20]{DFG-sol}
F.~Dahmani, K.~Fujiwara, and V.~Guirardel, \emph{Solvable groups of interval
  exchange transformations}, Ann. Fac. Sci. Toulouse Math. (6) \textbf{29}
  (2020), no.~3, 595--618.

\bibitem[DK18]{Drutu-Kapovich}
C.~Dru\c{t}u and M.~Kapovich, \emph{Geometric group theory}, American
  Mathematical Society Colloquium Publications, vol.~63, American Mathematical
  Society, Providence, RI, 2018, With an appendix by Bogdan Nica.

\bibitem[DS06]{Dranishnikov-Smith}
A.~Dranishnikov and J.~Smith, \emph{Asymptotic dimension of discrete groups},
  Fund. Math. \textbf{189} (2006), no.~1, 27--34.

\bibitem[Ele18]{Elek-simple-alg}
G.~Elek, \emph{Uniformly recurrent subgroups and simple {$C^*$}-algebras}, J.
  Funct. Anal. \textbf{274} (2018), no.~6, 1657--1689.

\bibitem[Ers04]{Ersch-Liouville}
A.~Erschler, \emph{Liouville property for groups and manifolds}, Invent. Math.
  \textbf{155} (2004), no.~1, 55--80.

\bibitem[Fra20]{Fraczyk-urs}
M.~Fraczyk, \emph{Kesten's theorem for uniformly recurrent subgroups}, Ergodic
  Theory Dynam. Systems \textbf{40} (2020), no.~10, 2778--2787.

\bibitem[Gla17]{Glas-IRS}
Y.~Glasner, \emph{Invariant random subgroups of linear groups}, Israel J. Math.
  \textbf{219} (2017), no.~1, 215--270, With an appendix by T. Gelander and Y.
  Glasner.

\bibitem[Gri84]{Gri-growth}
R.~I. Grigorchuk, \emph{Degrees of growth of finitely generated groups and the
  theory of invariant means}, Izv. Akad. Nauk SSSR Ser. Mat. \textbf{48}
  (1984), no.~5, 939--985.

\bibitem[Gro81]{Gromov-poly}
M.~Gromov, \emph{Groups of polynomial growth and expanding maps}, Inst. Hautes
  \'{E}tudes Sci. Publ. Math. (1981), no.~53, 53--73.

\bibitem[Gro93]{Gro-asdim}
\bysame, \emph{Asymptotic invariants of infinite groups}, Geometric group
  theory, {V}ol. 2 ({S}ussex, 1991), London Math. Soc. Lecture Note Ser., vol.
  182, Cambridge Univ. Press, Cambridge, 1993, pp.~1--295.

\bibitem[Gui73]{Guivarch-nilp}
Y.~Guivarc'h, \emph{Croissance polynomiale et p\'{e}riodes des fonctions
  harmoniques}, Bull. Soc. Math. France \textbf{101} (1973), 333--379.

\bibitem[GW15]{GW-urs}
E.~Glasner and B.~Weiss, \emph{Uniformly recurrent subgroups}, Recent trends in
  ergodic theory and dynamical systems, Contemp. Math., vol. 631, Amer. Math.
  Soc., Providence, RI, 2015, pp.~63--75.

\bibitem[Hal54]{Hall}
P.~Hall, \emph{Finiteness conditions for soluble groups}, Proc. London Math.
  Soc. (3) \textbf{4} (1954), 419--436.

\bibitem[Hig55]{Higman-torsionfree}
G.~Higman, \emph{Finite groups having isomorphic images in every finite group
  of which they are homomorphic images}, Quart. J. Math. Oxford Ser. (2)
  \textbf{6} (1955), 250--254.

\bibitem[Jac19]{Lison}
L.~Jacoboni, \emph{Metabelian groups with large return probability}, Ann. Inst.
  Fourier (Grenoble) \textbf{69} (2019), no.~5, 2121--2167.

\bibitem[JdlS15]{Ju-dlS}
K.~Juschenko and M.~de~la Salle, \emph{Invariant means for the wobbling group},
  Bull. Belg. Math. Soc. Simon Stevin \textbf{22} (2015), no.~2, 281--290.

\bibitem[JM13]{Ju-Mo}
K.~Juschenko and N.~Monod, \emph{Cantor systems, piecewise translations and
  simple amenable groups}, Ann. of Math. (2) \textbf{178} (2013), no.~2,
  775--787.

\bibitem[Ken20]{Kenn-urs}
M.~Kennedy, \emph{An intrinsic characterization of {$C^*$}-simplicity}, Ann.
  Sci. \'{E}c. Norm. Sup\'{e}r. (4) \textbf{53} (2020), no.~5, 1105--1119.

\bibitem[Kro84]{Kropholler84}
P.~H. Kropholler, \emph{On finitely generated soluble groups with no large
  wreath product sections}, Proc. London Math. Soc. (3) \textbf{49} (1984),
  no.~1, 155--169.

\bibitem[LB21]{LB-lattices}
A.~Le~Boudec, \emph{Amenable uniformly recurrent subgroups and lattice
  embeddings}, Ergodic Theory Dynam. Systems \textbf{41} (2021), no.~5,
  1464--1501.

\bibitem[LBMB18]{LBMB-sub-dyn}
A.~Le~Boudec and N.~Matte~Bon, \emph{Subgroup dynamics and {$C^*$}-simplicity
  of groups of homeomorphisms}, Ann. Sci. \'{E}c. Norm. Sup\'{e}r. (4)
  \textbf{51} (2018), no.~3, 557--602.

\bibitem[LBMB20a]{LB-MB-comm-lemm}
\bysame, \emph{A commutator lemma for confined subgroups and applications to
  groups acting on rooted trees}, arXiv:2006.08677.

\bibitem[LBMB20b]{LB-MB-confined-ht}
\bysame, \emph{Confined subgroups and high transitivity}, arXiv:2012.03997.

\bibitem[LR04]{Lennox-Rob}
J.~Lennox and D.~Robinson, \emph{The theory of infinite soluble groups}, Oxford
  Mathematical Monographs, The Clarendon Press, Oxford University Press,
  Oxford, 2004.

\bibitem[Mag39]{Magnus-embed}
W.~Magnus, \emph{On a theorem of {M}arshall {H}all}, Ann. of Math. (2)
  \textbf{40} (1939), 764--768.

\bibitem[Mal49]{Malcev49}
A.~I. Malcev, \emph{On a class of homogeneous spaces}, Izvestiya Akad. Nauk.
  SSSR. Ser. Mat. \textbf{13} (1949), 9--32, English translation, Amer. Math.
  Soc. Transl., no. 39 (1951).

\bibitem[Mat13]{Mat-exp}
H.~Matui, \emph{Some remarks on topological full groups of {C}antor minimal
  systems {II}}, Ergodic Theory Dynam. Systems \textbf{33} (2013), no.~5,
  1542--1549.

\bibitem[MB18]{MB-graph-germs}
N.~Matte~Bon, \emph{Rigidity properties of topological full groups of
  pseudogroups over the cantor set}.

\bibitem[MBT20]{MB-Tsan}
N.~Matte~Bon and T.~Tsankov, \emph{Realizing uniformly recurrent subgroups},
  Ergodic Theory Dynam. Systems \textbf{40} (2020), no.~2, 478--489.

\bibitem[Mil68]{Mil-sol}
J.~Milnor, \emph{Growth of finitely generated solvable groups}, J. Differential
  Geometry \textbf{2} (1968), 447--449.

\bibitem[Nek18]{Nek-frag}
V.~Nekrashevych, \emph{Palindromic subshifts and simple periodic groups of
  intermediate growth}, Ann. of Math. (2) \textbf{187} (2018), no.~3, 667--719.

\bibitem[Nek19]{Nek-simple-dyn}
\bysame, \emph{Simple groups of dynamical origin}, Ergodic Theory Dynam.
  Systems \textbf{39} (2019), no.~3, 707--732.

\bibitem[Neu37]{Neumann-manygroups}
B.~H. Neumann, \emph{Some remarks on infinite groups}, J. London Math. Soc.
  \textbf{12} (1937), 120--127.

\bibitem[Neu54]{Neum54}
\bysame, \emph{Groups covered by permutable subsets}, J. London Math. Soc.
  \textbf{29} (1954), 236--248.

\bibitem[Nov09]{Novak-disc-IET}
C.~Novak, \emph{Discontinuity-growth of interval-exchange maps}, J. Mod. Dyn.
  \textbf{3} (2009), no.~3, 379--405.

\bibitem[Ols15]{Olshanskii-KK}
A.~Yu. Olshanskii, \emph{On {K}aluzhnin-{K}rasner's embedding of groups},
  Algebra Discrete Math. \textbf{19} (2015), no.~1, 77--86.

\bibitem[Pas75]{Passi-annihilators}
I.~B.~S. Passi, \emph{Annihilators of relation modules. {II}}, J. Pure Appl.
  Algebra \textbf{6} (1975), no.~3, 235--237.

\bibitem[PSC03]{Pittet-Saloff-Coste}
Ch. Pittet and L.~Saloff-Coste, \emph{Random walks on finite rank solvable
  groups}, J. Eur. Math. Soc. (JEMS) \textbf{5} (2003), no.~4, 313--342.

\bibitem[Sal21]{Salo}
V.~Salo, \emph{Graph and wreath products in topological full groups of full
  shifts}, arXiv:2103.06663.

\bibitem[Sch27]{Schreier}
O.~Schreier, \emph{Die {U}ntergruppen der freien {G}ruppen}, Abh. Math. Sem.
  Univ. Hamburg \textbf{5} (1927), no.~1, 161--183.

\bibitem[Ste70]{Stewart}
I.~N. Stewart, \emph{An algebraic treatment of {M}alcev's theorems concerning
  nilpotent {L}ie groups and their {L}ie algebras}, Compositio Math.
  \textbf{22} (1970), 289--312.

\bibitem[Stu92]{Stuck-growth}
G.~Stuck, \emph{Growth of homogeneous spaces, density of discrete subgroups and
  {K}azhdan's property ({T})}, Invent. Math. \textbf{109} (1992), no.~3,
  505--517.

\bibitem[Wol68]{Wolf}
J.~A. Wolf, \emph{Growth of finitely generated solvable groups and curvature of
  {R}iemannian manifolds}, J. Differential Geometry \textbf{2} (1968),
  421--446.

\bibitem[Zar]{MO-zarem}
M.~Zaremsky, \emph{Finitely presented group containing every
  $\mathrm{GL}_n(\mathbb{Z})$}, MathOverflow,
  URL:https://mathoverflow.net/q/405967 (version: 2021-10-11).

\end{thebibliography}

\end{document}